\documentclass[12pt, reqno]{amsart}
\usepackage{amssymb, mathrsfs}
\usepackage{latexsym}
\usepackage{amsmath}
\usepackage{amsthm}
\usepackage{amsfonts}
\usepackage{dsfont, upgreek}
\usepackage{mathtools}
\usepackage{epsfig}
\usepackage{cite}
\usepackage{amscd}
\usepackage{graphicx}
\usepackage{tikz-cd}
\usepackage{enumitem}
\setlist[enumerate]{itemsep=0.5ex}

\usepackage{stmaryrd}

\usepackage{color}
\usepackage{tikz}
\usetikzlibrary{patterns}

\usepackage[left=1.4in,right=1.4in,top=1.2in,bottom=1.2in]{geometry}

\usepackage{tikz}
\usetikzlibrary{decorations.markings}
\usetikzlibrary{arrows.meta}

\usepackage{extarrows}

\usepackage{adjustbox}
\usepackage{pgfplots}
\usepgfplotslibrary{colormaps}
\usepackage{slashed}

\usepackage[colorlinks=true,linkcolor=blue,citecolor=blue]{hyperref}

\usepackage[colorinlistoftodos,prependcaption,textsize=tiny]{todonotes}  

\usepackage[all,cmtip]{xy}
\theoremstyle{plain}
\newtheorem{theorem}{Theorem}[section]
\newtheorem{proposition}[theorem]{Proposition}
\newtheorem{lemma}[theorem]{Lemma}
\newtheorem{corollary}[theorem]{Corollary}

\theoremstyle{definition} 
\newtheorem{definition}[theorem]{Definition}

\newtheorem*{claim*}{Claim}
\newtheorem{claim}[theorem]{Claim}

\theoremstyle{remark} 
\newtheorem{remark}[theorem]{Remark}

\numberwithin{equation}{section}

\newcommand{\Sc}{\mathrm{Sc}}

\newcommand{\R}{\mathbb{R}}
\newcommand{\tr}{\mathrm{tr}}

\newcommand{\sph}{\mathbb{S}}
\newcommand{\n}{\mathbf n}
\newcommand{\second}{\mathrm{II}}

\newcommand{\interior}[1]{%
	{\kern0pt#1}^{\mathrm{\,o}}%
}

\makeatletter
\let\save@mathaccent\mathaccent
\newcommand*\if@single[3]{%
	\setbox0\hbox{${\mathaccent"0362{#1}}^H$}%
	\setbox2\hbox{${\mathaccent"0362{\kern0pt#1}}^H$}%
	\ifdim\ht0=\ht2 #3\else #2\fi
}
\newcommand*\rel@kern[1]{\kern#1\dimexpr\macc@kerna}
\newcommand*\wideaccent[2]{\@ifnextchar^{{\wide@accent{#1}{#2}{0}}}{\wide@accent{#1}{#2}{1}}}
\newcommand*\wide@accent[3]{\if@single{#2}{\wide@accent@{#1}{#2}{#3}{1}}{\wide@accent@{#1}{#2}{#3}{2}}}
\newcommand*\wide@accent@[4]{%
	\begingroup
	\def\mathaccent##1##2{%
		\let\mathaccent\save@mathaccent
		\if#42 \let\macc@nucleus\first@char \fi
		\setbox\z@\hbox{$\macc@style{\macc@nucleus}_{}$}%
		\setbox\tw@\hbox{$\macc@style{\macc@nucleus}{}_{}$}%
		\dimen@\wd\tw@
		\advance\dimen@-\wd\z@
		\divide\dimen@ 3
		\@tempdima\wd\tw@
		\advance\@tempdima-\scriptspace
		\divide\@tempdima 10
		\advance\dimen@-\@tempdima
		\ifdim\dimen@>\z@ \dimen@0pt\fi
		\rel@kern{0.6}\kern-\dimen@
		\if#41
		#1{\rel@kern{-0.6}\kern\dimen@\macc@nucleus\rel@kern{0.4}\kern\dimen@}%
		\advance\dimen@0.4\dimexpr\macc@kerna
		\let\final@kern#3%
		\ifdim\dimen@<\z@ \let\final@kern1\fi
		\if\final@kern1 \kern-\dimen@\fi
		\else
		#1{\rel@kern{-0.6}\kern\dimen@#2}%
		\fi
	}%
	\macc@depth\@ne
	\let\math@bgroup\@empty \let\math@egroup\macc@set@skewchar
	\mathsurround\z@ \frozen@everymath{\mathgroup\macc@group\relax}%
	\macc@set@skewchar\relax
	\let\mathaccentV\macc@nested@a
	\if#41
	\macc@nested@a\relax111{#2}%
	\else
	\def\gobble@till@marker##1\endmarker{}%
	\futurelet\first@char\gobble@till@marker#2\endmarker
	\ifcat\noexpand\first@char A\else
	\def\first@char{}%
	\fi
	\macc@nested@a\relax111{\first@char}%
	\fi
	\endgroup
}
\makeatother

\newcommand\overbar{\wideaccent\overline}

\makeatletter
\newcommand*{\transpose}{%
	{\mathpalette\@transpose{}}%
}
\newcommand*{\@transpose}[2]{%
	\raisebox{\depth}{$\m@th#1\intercal$}%
}
\makeatother

\begin{document}
\title[Scalar-mean rigidity theorem]{Scalar-mean  rigidity theorem and Llarull's theorem for nonspin manifolds}

\author{Gaoming Wang}
\address[Gaoming Wang]{Beijing Institute of Mathematical Sciences and Applications}
\email{wanggaoming@bimsa.cn}
\thanks{}

\author{Jinmin Wang}
\address[Jinmin Wang]{Institute of Mathematics, Chinese Academy of Sciences}
\email{jinmin@amss.ac.cn}
\thanks{}

\author{Zhizhang Xie}
\address[Zhizhang Xie]{Department of Mathematics, Texas A\&M University}
\email{xie@tamu.edu}
\thanks{}

\author{Bo Zhu}
\address[Bo Zhu]{Yau Mathematical Sciences Center, Tsinghua University }
\email{zhub@tsinghua.edu.cn}
\thanks{}
 
\begin{abstract}
We prove Llarull's scalar curvature rigidity theorem for spheres and
scalar-mean rigidity for strictly convex Euclidean domains in all dimensions
without the spin assumption. 
\end{abstract}
\maketitle

\section{Introduction}

A main purpose of this paper is to prove Llarull's scalar curvature rigidity
theorem for spheres in all dimensions without spin assumption.
Llarull's theorem states that a closed spin
manifold with scalar curvature at least that of the unit round sphere cannot
admit a distance non-increasing map of nonzero degree to the sphere unless the
map is an isometry \cite[Theorem B]{Llarull}. Listing subsequently strengthened this result by
replacing the separate scalar curvature comparison and distance-nonincreasing assumptions
with a single inequality relating the scalar curvature to the pointwise
dilation of the map \cite{Listing:2010te}. 
These results reveal a beautiful
interplay between metric invariant, curvature, and topology. Recall that their proofs use the Dirac operator and hence rely on spin assumption. Removing this assumption has been a central problem
in scalar curvature geometry. 

\medskip

In this paper, we prove  Llarull's theorem 
in all dimensions without the spin assumption. More precisely, we have the following main theorem of the paper. 
\begin{theorem}\label{thm:llarull}
	Let $(M^n,g)$ be a smooth connected closed  Riemannian manifold, and let $F\colon (M,g)\to (\sph^n,g_{\sph^n})$ be a smooth map with non-zero degree, where 
    $(\sph^n,g_{\sph^n})$ is the standard $n$-dimensional round unit sphere with the round metric $g_{\sph^n}$. If
	\begin{itemize}
		\item $F$ is $1$-Lipschitz, namely $g\geq F^*g_{\sph^n}$, and
		\item $\Sc(g)\geq n(n-1)$.
	\end{itemize}
	Then $F$ is an isometry.
\end{theorem}
In fact, the same techniques allow us to prove Listing's strengthening in all dimensions without the spin assumption. 
\begin{theorem}\label{thm:listing}
	Let $(M^n,g)$ be a smooth connected closed Riemannian manifold of dimension $n\geq 2$, and let $F\colon (M,g)\to (\sph^n,g_{\sph^n})$ be a smooth map with non-zero degree, where 
    $(\sph^n,g_{\sph^n})$ is the standard $n$-dimensional round unit sphere with the round metric $g_{\sph^n}$. Assume that
\[ \Sc(g)\geq \|dF\|^2\cdot n(n-1),\]
where $\|dF\|$ denotes the pointwise operator norm
    of $dF$ with respect to $g$ and $g_{\sph^n}$.
	Then $F$ is an isometry up to a constant rescaling of $g$.
\end{theorem}

An early step toward removing the spin assumption was the work of
Cecchini--Wang--Xie--Zhu \cite{CWXZLlarull4}, who established Llarull's
theorem in dimension four without spin assumption. Their argument
uses Gromov's $\mu$-bubbles and a suitable conformal deformation to reduce
the problem to Listing's theorem on a three-dimensional hypersurface.
Since every oriented three-manifold is spin, Listing's theorem applies
and completes the proof. In higher dimensions, however, the hypersurfaces
arising from this construction need not be spin, so the reduction does
not eliminate the original obstruction. Moreover, the resulting curvature
inequality involves the pointwise dilation of the map and therefore
requires Listing's strengthening rather than Llarull's theorem in its
original form. It is unclear how to iterate the $\mu$-bubble construction
under this more general curvature condition to carry out dimension descent.

Inspired by Gromov's work, Wang--Wang--Zhu \cite{WangWangZhu} established
scalar-mean curvature rigidity for Euclidean balls in dimensions up to
four without spin assumption, using capillary hypersurfaces and
dimension descent. Their work suggested a new approach to spherical
rigidity: passing from spheres to balls by taking metric cones.
Indeed, the  metric cone over the unit round $n$-sphere is
the Euclidean unit $(n+1)$-ball. More generally, the cone metric
$g_C=dr^2+r^2g$ on $(0,1]\times M$ satisfies
\[
    \Sc(g_C)=r^{-2}\bigl(\Sc(g)-n(n-1)\bigr),
\]
and its outer boundary $\{1\}\times M$ has mean curvature $n$. Thus, the scalar curvature lower
bound in Llarull's theorem becomes nonnegative scalar curvature on
the cone, while the original distance-nonincreasing map supplies
the boundary comparison with the Euclidean ball. This leads naturally
to a reformulation of Llarull's theorem as a scalar-mean curvature
rigidity problem in one higher dimension.

However, There are still two major obstacles to carrying out this strategy. First, the
cone over a general Riemannian manifold has a singularity at its vertex,
so the standard existence and regularity theory for capillary
hypersurfaces in smooth ambient manifolds does not apply directly.
Second, in higher dimensions,
capillary hypersurfaces may themselves develop singularities, including
boundary singularities, even when the ambient manifold is smooth. These singularities obstruct a direct iteration of the
Schoen--Yau dimension-descent argument \cite{SY_pmt_1,Schoen_Yau_imcompressible,Schoen_Yau_psc_higher}.

Here, in this paper, we overcome both obstacles by conformally blowing up the singular sets
and establishing a weighted scalar-mean curvature comparison theorem
for complete manifolds. The blow-up sends the singular sets to infinity
and equips the regular parts with complete conformal metrics. The
weighted formulation allows us to preserve the curvature inequalities
under conformal blow-up and to pass them to capillary hypersurfaces.
We can therefore carry out the dimension-descent argument in arbitrary
dimensions without  spin assumption.

Our approach builds on recent developments in conformal blow-up methods
for scalar curvature problems. Chu--Lee--Zhu \cite{ChuLeeZhu24} used conformal blow-up to study scalar
curvature rigidity for $L^\infty$-metrics with isolated singularities.
Wang--Xie \cite{WangXie24} extended this approach to singular sets of
dimension up to roughly half the ambient dimension.
Bi--Hao--He--Shi--Zhu \cite{BHHSZ} subsequently adapted the method to
minimal-hypersurface approaches to the positive mass theorem. Inspired by the work of Bi--Hao--He--Shi--Zhu \cite{BHHSZ},
Brendle--Wang \cite{BrendleWang} introduced a weighted scalar curvature
framework compatible with both conformal blow-up and dimension descent
to prove the Riemannian positive mass theorem in all dimensions.
Further developments
include results on Schoen's conjecture for $L^\infty$-metrics on tori
\cite{WangWangXie}, the positive mass theorem in the presence of
low-codimension singularities \cite{KhuriWangWang}, and generalized 
Geroch conjecture with no spin assumption \cite{BiZhu}. We adapt and extend this framework to scalar-mean curvature comparison,
where one must also control the boundary geometry and the singularities
of capillary hypersurfaces.

\begin{remark}
    In dimensions $n\geq 3$, Llarull's theorem in the spin setting holds
    under the weaker assumption that $F$ is area-nonincreasing
    \cite[Theorem C]{Llarull}. Likewise, Listing's theorem admits an
    area-dilation version, with $\|dF\|^2$ replaced by
    $\|\wedge^2 dF\|$ \cite{Listing:2010te}. The corresponding extensions
    to nonspin manifolds remain open in dimensions at least four.
\end{remark}

As explained above, our approach to Theorem \ref{thm:llarull} proceeds
through scalar-mean curvature comparison for manifolds with boundary.
Such comparison problems naturally couple the scalar curvature in the
interior with the mean curvature of the boundary
\cite{Gromov_mean_light_scalar}. In the spin setting, index-theoretic
methods have yielded scalar-mean rigidity theorems for smooth, compact
Riemannian manifolds with nonempty boundary (see for example
\cite{Lottboundary,Wang:2022vf,MR3257837,Simone24,Wang:2021tq}).
Related comparison results have also been established for certain
possibly degenerate warped product metrics
\cite{Cecchini2021:vs,ChaiWan24} and, more generally, for suitable
conformal metrics \cite{WangXieConformal}.

Gromov conjectured that scalar-mean rigidity should hold without the
spin assumption and proposed an approach based on capillary
$\mu$-bubbles \cite[Section 5.8.1]{Gromov_four_lectures}.
Building on this approach and the work of Wang--Wang--Zhu
\cite{WangWangZhu}, we combine capillary $\mu$-bubbles with conformal
blow-up and weighted scalar-mean curvature comparison to carry out
dimension descent in arbitrary dimensions. In particular, we obtain
the following scalar-mean rigidity theorem without the spin assumption.

\begin{theorem}\label{thm:scmean_ball}
	Let $(M^n,\partial M,g)$ be a smooth compact connected Riemannian manifold with boundary of dimension $n\geq 3$, and let $F\colon (\partial M,g)\to (\sph^{n-1},g_{\sph^{n-1}})$ be a  smooth map with non-zero degree.  Suppose that
    \[
        \Sc(g)\geq 0 \textup{ on } M
        \textup{ and }
        H_{\partial M}(g)\geq (n-1)\|dF\|
        \textup{ on } \partial M,
    \]
    where  $H_{\partial M}(g)$ is the mean curvature with respect to
    the outward unit normal.
    Then there exists a constant $c>0$ such that $(M,c\cdot g )$ is isometric
    to the Euclidean unit ball and
    \[
        F\colon (\partial M,c\cdot g)
        \longrightarrow (\sph^{n-1},g_{\sph^{n-1}})
    \]
    is an isometry.
\end{theorem}

More generally, the same scalar-mean curvature rigidity result holds
for smooth, bounded, strictly convex Euclidean domains.

\begin{theorem}\label{thm:scmean_convexdomain}
    Let $(M^n,\partial M, g)$ be a smooth compact connected Riemannian
    manifold  with  boundary of dimension $n\geq 3$, and let
    $\Omega\subset\mathbb R^n$ be a bounded, strictly convex domain
    with smooth boundary, equipped with the Euclidean metric
    $g_{\mathrm{eu}}$. Let
    $F\colon (\partial M,g)
    \to (\partial\Omega,g_{\mathrm{eu}})$
    be a smooth map of nonzero degree. Assume that
    \[
        \Sc(g)\geq 0  \textup{ on } M,
        H_{\partial M}(g)\geq
        \|dF\|\bigl(H_{\partial\Omega}(g_{\mathrm{eu}})\circ F\bigr)
        \quad \text{on } \partial M.
    \]
    Then there exists a constant $c>0$ such that $(M,c\cdot g)$ is isometric
    to $(\overline{\Omega},g_{\mathrm{eu}})$ and
    \[
        F\colon (\partial M,c\cdot g_{\partial M})
        \longrightarrow (\partial\Omega,g_{\mathrm{eu}})
    \]
    is an isometry. In particular, $g$ is flat.
\end{theorem}
 

We remark either Theorem \ref{thm:scmean_ball} or  Theorem \ref{thm:scmean_convexdomain} implies the Riemannian positive mass theorem 
in all dimensions. If $n=2$ in Theorem \ref{thm:scmean_ball} and Theorem \ref{thm:scmean_convexdomain}, then $M$ is flat by the Gauss--Bonnet theorem, but $F$ need not be an isometry. 

All our main results follow from a scalar-mean curvature comparison
theorem for complete, possibly noncompact manifolds
(Theorem~\ref{thm:noncompactScMean}). For
Theorem~\ref{thm:scmean_ball}, conformal deformations convert a failure
of rigidity into strict scalar and mean curvature inequalities.
For Theorem~\ref{thm:scmean_convexdomain}, we first compose the boundary
comparison map with the Gauss map to obtain a trace norm comparison
with the sphere, and then apply the same deformation argument.
For Theorem~\ref{thm:llarull}, we begin with the Riemannian product
$M\times[0,\infty)$, which is conformally equivalent to the punctured
 cone over $M$, and apply conformal deformations to obtain
a complete noncompact manifold whose boundary satisfies the required
trace norm comparison. The complete-manifold formulation thus provides
a common framework for both  compact scalar-mean rigidity and spherical rigidity problems.

Here is an outline of our proof of the scalar-mean curvature comparison
theorem for complete, possibly noncompact manifolds
(Theorem~\ref{thm:noncompactScMean}). To carry out dimension
reduction, we first establish a weighted scalar-mean comparison theorem
(Theorem~\ref{thm:noncompactScMeanWeighted}) involving weighted scalar
curvature in the interior and weighted mean curvature on the boundary.
The stability inequality for a capillary $\mu$-bubble gives a spectral
form of this comparison.   We then construct the blow-up function using a
Green function.  To estimate this blow-up function near the singular set, we use an Assouad dimension bound, inspired by \cite{CheegerNaber,NaberValtorta,BiZhu},  which gives quantitative covering information at different scales. After the conformal blow-up, we modify the boundary comparison map while preserving the degree information needed for induction. Finally, we use a positive function obtained from the spectral inequality to adjust the weight and obtain the desired pointwise weighted inequalities in one lower dimension. 

The paper is organized as follows. In Section \ref{sec:strict-noncompact}, we reduce the main theorems to one trace scalar-mean comparison theorem for complete manifolds (Theorem \ref{thm:noncompactScMean}). In Section \ref{sec:weighted}, we introduce weighted scalar curvature and mean curvature, and formulate a weighted scalar-mean comparison theorem for non-compact manifolds (Theorem \ref{thm:noncompactScMeanWeighted}), which implies Theorem \ref{thm:noncompactScMean}. We also describe the dimension reduction scheme and prove its two-dimensional base case. Section~\ref{sec:dimensionReduction} carries out the induction and proves Theorem \ref{thm:noncompactScMeanWeighted}.

\subsection*{AI disclosure:}

AI assistance was used to help check the Assouad dimension estimates for the singular set of a capillary $\mu$-bubble, as well as various other computations throughout the paper. The authors have independently reviewed, re-derived, and verified all mathematical arguments and computations and take full responsibility for their correctness.

	\subsection*{Acknowledgments}

	We thank Zhichao Wang and Jian Wang for helpful comments. Jinmin Wang is partially supported by NSFC 12501169. Zhizhang Xie is partially by NSF DMS-2247322. Bo Zhu is partially supported by NSFC 12501066.

\section{Reduction to a noncompact trace comparison}
\label{sec:strict-noncompact}
In this section, we reduce the main theorems to a scalar-mean comparison theorem for complete, possibly noncompact manifolds. We first state the common comparison result and then give the three reductions. 
\begin{theorem}\label{thm:noncompactScMean}
	Let $(M^n,\partial M,g)$ be a complete Riemannian manifold with compact boundary. Let $F\colon \partial M\to \sph^{n-1}$ be a smooth map. If there exists $\delta>0$ such that
	\begin{itemize}
		\item $H_{\partial M}(g)\geq \|dF\|_\tr+\delta$, where $\|\cdot\|_\tr$ denotes the trace norm, and
		\item $\Sc(g)\geq \delta$,
	\end{itemize}
	then the degree of $F$ is zero.
\end{theorem}
Here the manifold is allowed to be noncompact, and the boundary comparison is in terms of the trace norm. We now derive all the main results from Theorem~\ref{thm:noncompactScMean}.

\subsection{From scalar-mean comparison to trace norm comparison}

We first deduce Theorems~\ref{thm:scmean_ball} and~\ref{thm:scmean_convexdomain} from Theorem~\ref{thm:noncompactScMean}. The key step is a conformal deformation that makes both curvature inequalities strict. We recall the construction from \cite[Proofs of Theorems~4.1 and~4.2]{WangWangZhu}.

Let $(M^n,\partial M,g)$ be a compact manifold with boundary and $n\geq 3$, and let $F\colon \partial M\to\sph^{n-1}$ be a smooth map. Consider the following Robin problem for the lowest eigenvalue:
\begin{equation} \label{eq: NB_H}
	\begin{cases}
		-\displaystyle \frac{4(n-1)}{n-2}\Delta u + \Sc(g) u = \lambda u,\\
		\displaystyle\frac{\partial u}{\partial \nu}=-\frac{n-2}{2(n-1)}(H(g)-\| dF\|_{\tr})u,
	\end{cases}
\end{equation}
where $\nu$ is the outward unit normal vector field along $\partial M$.
\begin{lemma}\label{lemma:conformalRobin}
	If the lowest eigenvalue of \eqref{eq: NB_H} is positive, then there exist a metric $\hat g$ conformal to $g$ and a constant $\delta>0$ such that
	\[
		\Sc(\hat g)\geq\delta \quad\text{on }M,
		\qquad
		H_{\partial M}(\hat g)\geq\|dF\|_{\tr,\hat g}+\delta \quad\text{on }\partial M.
	\]
	In particular, $(M^n,\partial M,\hat g)$ and $F$ satisfy the hypotheses of Theorem~\ref{thm:noncompactScMean}.
\end{lemma} 
\begin{proof}
	Let $\lambda>0$ be the lowest eigenvalue of \eqref{eq: NB_H}, and let $u$ be a corresponding positive eigenfunction.
	We first consider the conformal metric on $M$ given by
	$$(M, g_u) := (M, u^{\frac{4}{n-2}}g).$$
	The scalar curvature of $g_u$ on $M$ is given by
		\begin{equation}
			\Sc(g_u)=u^{-\frac{n+2}{n-2}}\Big(-\frac{4(n-1)}{n-2}\Delta u+\Sc(g)u\Big)=\lambda u^{-\frac{4}{n-2}}	\geq\delta_1>0
		\end{equation}
		for some $\delta_1>0$, since $M$ is compact. On the boundary, we have
		\begin{equation}
			\| dF\|_{\tr,g_u}=u^{-\frac{2}{n-2}}\| dF\|_{\tr,g|_{\partial M}}.
		\end{equation}
	The mean curvature of $\partial M$ with respect to $g_u$ is given by
    \begin{equation}
		\begin{split}
		H_{\partial M}(g_u)&=u^{-\frac{2}{n-2}}\big(H_{\partial M}(g)+\frac{2(n-1)}{n-2}\frac{1}{u}\frac{\partial u}{\partial \nu}\big)\\
		&=u^{-\frac{2}{n-2}}\| dF\|_{\tr,g|_{\partial M}}=\|dF\|_{\tr,g_u}.
		\end{split}
	\end{equation}

	\vspace{1mm}
	
	Next, we deform $(M^n,\partial M,g_u)$ to increase the boundary mean curvature by taking advantage of the positive scalar curvature in the interior. Let $\nu_{g_u}$ be the outward unit normal to $\partial M$ with respect to $g_u$, and let $w$ be a smooth function on $M$ such that
	$\nu_{g_u}(w)=1.$
	For sufficiently small $\varepsilon>0$, we have $1+\varepsilon w>0$ and may define
	\[
		g_w=(1+\varepsilon w)^{\frac{4}{n-2}}g_u.
	\]
 The scalar curvature of $g_w$ on $(M, g_w)$ is given by
		\begin{equation}
			\Sc(g_w)=(1+\varepsilon w)^{-\frac{n+2}{n-2}}\Big(\varepsilon\big(-\frac{4(n-1)}{n-2}\Delta^{g_u}w+\Sc(g_u)w\big)+\Sc(g_u)\Big),
		\end{equation}
		which is strictly positive if we choose $\varepsilon>0$ to be sufficiently small. On the boundary, we have
		\begin{equation}
			\| dF\|_{\tr,g_w}=(1+\varepsilon w)^{-\frac{2}{n-2}}\| dF\|_{\tr,g_u}.
		\end{equation}
The mean curvature of $\partial M$ with respect to $g_w$ is given by
		\begin{equation}
			\begin{split}
				H(g_w)&=(1+\varepsilon w)^{-\frac{2}{n-2}}\big(H(g_u)+{\frac{2(n-1)}{n-2}\frac{\varepsilon}{1+\varepsilon w} \nu_{g_u}(w)} \big)\\
				&=\| dF\|_{\tr,g_w}+\frac{2(n-1)}{n-2}\varepsilon\cdot (1+\varepsilon w)^{-\frac{n}{n-2}}.
			\end{split}
		\end{equation}
	
	To summarize, we obtain a compact manifold $(M,\partial M,g_w)$ that satisfies all the assumptions of Theorem \ref{thm:noncompactScMean}.
\end{proof}

We now apply Lemma~\ref{lemma:conformalRobin} to the rigidity theorem for the Euclidean ball (Theorem \ref{thm:scmean_ball}). Since the conformal deformation leaves $F$ unchanged, a positive lowest eigenvalue would contradict Theorem~\ref{thm:noncompactScMean} whenever $\deg F\neq 0$.
\begin{lemma}\label{lemma:sc->ball}
	Theorem \ref{thm:noncompactScMean} implies  Theorem \ref{thm:scmean_ball}.
\end{lemma}

\begin{proof}	
	Assume that $n\geq 3$ and that $(M^n,\partial M,g)$ satisfies all the assumptions of Theorem \ref{thm:scmean_ball}. We show that either the rigidity conclusion holds or we obtain a contradiction to Theorem \ref{thm:noncompactScMean}.

	By H\"older's inequality, we have
	\begin{equation}\label{eq:HOlder}
		\|dF\|_\tr\leq \|dF\|\cdot (n-1).
	\end{equation}
	Therefore, we have
	\begin{equation}\label{eq:sc>=,H>=}
		\Sc(g) \geq 0\text{ on }M \textup{ and } \ H(g) \geq  \|dF\|\cdot (n-1)\geq \| dF\|_{\tr} \text{ on }\partial M.
	\end{equation}
	If any of the above inequalities is strict at any point, then the lowest eigenvalue of \eqref{eq: NB_H} is strictly positive. This contradicts Theorem \ref{thm:noncompactScMean} by Lemma \ref{lemma:conformalRobin}. Therefore, we have
	\begin{equation}
		\Sc(g) =0\text{ on }M \textup{ and }  H(g) = \|dF\|\cdot (n-1)=\| dF\|_{\tr} \text{ on }\partial M.
	\end{equation}
	Equality in H\"older's inequality implies that $dF$ is a homothety. Thus $$h\coloneqq \|dF\|=\frac{H(g)}{n-1}$$
	is a smooth function on $\partial M$. In this case, \eqref{eq: NB_H} becomes the standard Neumann problem, whose lowest eigenvalue is zero, with  eigenfunction $u_0 =1$.
	
	The classical work of Kazdan \cite{MR675736} relates the first variation of the lowest Neumann eigenvalue to the Ricci curvature. Here, we consider a deformation of \eqref{eq: NB_H}. More precisely, let $g_t=g+tk$ for a symmetric $2$-tensor field $k$, and consider \eqref{eq: NB_H} with respect to $g_t$:
	\begin{equation} \label{eq: NB_Ht}
		\begin{cases}
			-\displaystyle \frac{4(n-1)}{n-2}\Delta^{g_t} u_t + \Sc(g_t) u_t = \lambda_t u_t,\\
			\displaystyle\nu_t(u_t)=-\frac{n-2}{2(n-1)}(H(g_t)-\| dF\|_{\tr,g_t})u_t,
		\end{cases}
	\end{equation}
	Differentiating \eqref{eq: NB_Ht}, the variations of the Laplacian and normal derivative acting on the constant $u_0=1$ vanish. Integrating the differentiated interior equation and using the differentiated boundary
	condition gives
	\begin{equation}\label{eq:variation_lambda}
		\begin{split}
\operatorname{vol}_g(M)\frac{d\lambda_t}{dt}=\int_M \frac{d\Sc(g_t)}{dt}d\mathcal H^n_g+2\int_{\partial M} \frac{d(H(g_t)-\|dF\|_{\tr,g_t})}{dt}d\mathcal H^{n-1}_g
		\end{split}
	\end{equation}
	at $t=0$. 
   
   Write $\gamma=g|_{T\partial M}$ and $\gamma_t=g_t|_{T\partial M}$.  For the symmetric $2$-tensor $k$, we write  $k^\top=k|_{T\partial M}$. Recall that 
	\begin{equation}\label{eq:variation_scalar}
	    \frac{d\Sc(g_t)}{dt}=-\langle\operatorname{Ric}_g,k\rangle
		+\operatorname{div}_g\operatorname{div}_gk-\Delta_g\operatorname{tr}_gk.
	\end{equation}
and 
\begin{equation}\label{eq:variation_mean}
    2\frac{dH(g_t)}{dt}
=\nu(\operatorname{tr}_g k)
-(\operatorname{div}_g k)(\nu)
-\operatorname{div}_{\partial M}X
-\langle\second,k^\top\rangle_\gamma.
\end{equation}
where $\second$ denotes the second fundamental form of $\partial M$ and $X$ is the tangential vector field determined by
$\gamma(X,V)=k(\nu,V)$ for $V\in T\partial M$.
   
Adding \eqref{eq:variation_scalar} and \eqref{eq:variation_mean} and using
$\int_{\partial M}\operatorname{div}_{\partial M}X\,dA_g=0$,
we obtain
\begin{equation}\label{eq:variation_scalar_mean}
		\begin{split}
		   &  \int_M \frac{d\Sc(g_t)}{dt}\,d\mathcal H_g^n+2\int_{\partial M}\frac{dH(g_t)}{dt} \,d\mathcal H_g^{n-1} \\
		  = & -\int_M\langle\operatorname{Ric}_g,k\rangle\,d\mathcal H_g^n
		-\int_{\partial M}\langle\second,k^\top\rangle_\gamma\,d\mathcal H_g^{n-1}.
		\end{split}
	\end{equation}
  Note that 
    \[
		\|dF\|_{\tr,g_t}
		=h\,\operatorname{tr}\big((\gamma_t^{-1}\gamma)^{1/2}\big).
	\]Differentiating the trace-norm formula gives
	\begin{equation}\label{eq:d_of_dF}
	    \frac{d{\|dF\|}_{\tr, g_t}}{dt}=-\frac h2\operatorname{tr}_\gamma k^\top.
	\end{equation}
    where $h = \|dF\|_g =  \frac{H(g)}{n-1}$.
	Combining \eqref{eq:variation_lambda}, \eqref{eq:variation_scalar_mean} and \eqref{eq:d_of_dF}  yields
	\begin{equation}
		\begin{split}
			\operatorname{vol}_g(M)\frac{d\lambda_t}{dt}
			={}&-\int_M\langle\operatorname{Ric}_g,k\rangle\,d\mathcal H_g^n -\int_{\partial M}\langle\second-h\gamma,k^\top\rangle_\gamma\,d\mathcal H_g^{n-1}.
		\end{split}
	\end{equation}
    
	
	Now if $\operatorname{Ric}_g\ne 0$, then there exists a smooth $2$-tensor field $k$ vanishing on $\partial M$ such that the lowest eigenvalue of \eqref{eq: NB_H} for the  deformed metric $g_t$ is positive. If $\operatorname{Ric}_g\equiv 0$ but $\second \neq h\gamma$, one also finds a smooth $2$-tensor field $k$ supported near $\partial M$ such that the lowest eigenvalue of \eqref{eq: NB_H} for the  deformed metric $g_t$ is positive. In either case, we obtain a new metric on $M$ that contradicts Theorem \ref{thm:noncompactScMean}. Therefore, we conclude that $(M,\partial M,g)$ is Ricci-flat, and $\partial M$ is umbilic, namely $\second= h \gamma = h\cdot g_{\partial M}$.
	
	\begin{claim}
	     $(M,\partial M,g)$ is isometric to a Euclidean ball.
	\end{claim}  
    By the Codazzi equation, 
	\begin{equation}
		\begin{split}
			\langle R_g(X,Y)Z,\nu\rangle&=\big((\nabla_Y \second)(X,Z)-(\nabla_X \second)(Y,Z)\big)\\&=Y(h)\langle X,Z\rangle-X(h)\langle Y,Z\rangle
		\end{split}
	\end{equation}
	for tangent vectors $X,Y,Z$ in $T(\partial M)$. Setting $Z=X$ and taking the trace over $X$, we obtain
	\begin{equation}
		0=\langle\operatorname{Ric}_g(Y),\nu\rangle=(n-1)Y(h)-Y(h)=(n-2)Y(h).
	\end{equation}
	Since $n\geq 3$, we obtain that $h$ is a constant on each connected component of the boundary.  At least one component has $h>0$, since $h = \|dF\|_g$ and $\deg F\neq0$. It follows from Kasue's splitting theorem \cite{Kasue1983} that $\partial M$ is connected.  


    Since $dF$ is a homothety and $h = \|dF\|_g $ is a constant function on $\partial M$, after possibly scaling the metric, we obtain that $F\colon \partial M\to \sph^{n-1}$ is a local isometry, thus a global isometry since $\mathbb S^{n-1}$ is simply connected for $n\geq 3$. After this normalization, we have 
   \begin{equation}
       \operatorname{Ric}_{g}=0,\qquad
		\second_{g}= g|_{T\partial M},\qquad
		H_{\partial M}(g)=n-1.
   \end{equation}
    Therefore,
	\begin{equation}\label{eq:integralMean}
		\int_{\partial M} H_{\partial M}(g)d\mathcal H_{g}^{n-1}= (n-1)\cdot\operatorname{Vol}(\sph^{n-1}).
	\end{equation}

    Now we finish the proof using Ricci flatness. Let $x_i\colon \R^n\to \R$ be the coordinate functions. Restrict each $x_i$ to $\sph^{n-1}$, and solve
	\[
		\Delta^{g}u_i=0 \quad\textup{on }M,
		\qquad u_i|_{\partial M}=x_i\circ F,\qquad 1\leq i\leq n.
	\]
 We obtain $n$ harmonic functions $u_i$ with $u_i|_{\partial M}=x_i\circ F$. Sinc $F\colon (\partial M,g|_{T\partial M})
\to(\mathbb S^{n-1},g_{\mathbb S^{n-1}})$
is an isometry, and 
\[
\Delta_{\mathbb S^{n-1}}x_i=-(n-1)x_i,
\qquad
|\nabla^{\mathbb S^{n-1}}x_i|^2=1-x_i^2,
\]
pulling back by $F$ gives
\begin{equation}\label{eq:boundarylaplace}
\Delta^{\partial M}u_i=-(n-1)u_i,
\qquad
\sum_{i=1}^n|\nabla^{\partial M}u_i|^2=n-1,
\qquad
\sum_{i=1}^n u_i^2=1
\textup{ on }\partial M.
\end{equation}

By the Reilly formula \cite{Reilly1977}, since $\operatorname{Ric}_g\equiv 0$, we have
	\begin{equation}\label{eq:Bochner}
		\begin{split}
			&-\int_{M}|\operatorname{Hess}(u_i)|^2d\mathcal H^n_g\\
			=&~\int_{\partial M}\left( H_{\partial M}(g)\cdot (\nu(u_i))^2+
			2\nu(u_i)\cdot \Delta^{\partial M}u_i+\second(\nabla^{\partial M} u_i,\nabla^{\partial M} u_i)
			\right)d\mathcal H^{n-1}_g.
		\end{split}
	\end{equation}

	Using the fact $H_{\partial M}(g)=n-1$ and $\second=g|_{\partial M}$, 
 substituting  \eqref{eq:boundarylaplace} into \eqref{eq:Bochner}
and summing over $i$, we obtain
	\begin{equation}
		\begin{split}
			&-\sum_{i=1}^n\int_{M}|\operatorname{Hess}(u_i)|^2d\mathcal H^n_g\\
			=&~(n-1)\sum_{i=1}^n\int_{\partial M}\left( (\nu(u_i))^2-
			2u_i\nu(u_i)+u_i^2
			\right)d\mathcal H^{n-1}_g\geq 0.
		\end{split}
	\end{equation}
	It follows that $u_i$ has vanishing Hessian, and $\nu(u_i)=u_i$ for every $i=1,\ldots,n$. Consequently, the map $U=(u_1,\cdots,u_n)\colon M\to \R^n$ induces an isometry between $M$ and the unit ball, and $U|_{\partial M}=F$ by construction. This finishes the proof.
\end{proof}
\begin{remark}
    For $n=3$, Shi--Tam in \cite[Theorem 1]{Shi_Tam} prove the required rigidity from the equation \eqref{eq:integralMean}. For $n\leq 7$, see \cite[Theorem 2]{Shi_Tam_extension}. In fact, these results reduce the comparison of the integral of the mean curvature to the positive mass theorem. Therefore, in arbitrary dimension, the rigidity follows from the positive mass theorem in \cite{BrendleWang} and the rigidity statement in \cite{KhuriWangWang}.  
\end{remark}

\begin{lemma}\label{lemma:sc->convexDomain}
	Theorem \ref{thm:noncompactScMean} implies Theorem \ref{thm:scmean_convexdomain}
\end{lemma}
\begin{proof}
	Assume that $n\geq 3$ and that $(M^n,\partial M,g)$ satisfies all the assumptions of Theorem \ref{thm:scmean_convexdomain}. We show that either the rigidity conclusion holds or we obtain a contradiction to Theorem \ref{thm:noncompactScMean}.
	
	Let us first reduce the boundary comparison to the trace-norm comparison against the standard sphere. Let
	\begin{equation}
		G\colon \partial\Omega\to \sph^{n-1},~G(x)=\n(x)
	\end{equation}
	be the Gauss map, where $\n$ denotes the outward unit normal vector to $\partial\Omega$. We have
	\begin{equation}
		\langle dG(X),Y\rangle=\second_{\partial\Omega}(X,Y)
	\end{equation}
	for any tangent vectors $X,Y$ in $T_x\partial\Omega=T_{G(x)}\sph^{n-1}$, where $\second_{\partial\Omega}$ denotes the second fundamental form of $\partial\Omega$. Therefore,
	\begin{equation}
		\|dG\|_{\tr}=H_{\partial\Omega}(g_{eu}),
	\end{equation}
	since $\second_{\partial\Omega}\geq 0$ by assumption.
	Therefore, if we define $\widehat F=G\circ F$, then by H\"older's inequality
	\begin{equation}
		\|d\widehat F\|_{\tr}\leq \|dF\|\cdot(\|dG\|_{\tr}\circ F)=\|dF\|\cdot(H_{\partial\Omega}\circ F).
	\end{equation}
	Hence, by assumption,
	\begin{equation}
		H(g)\geq \|d\widehat F\|_\tr,
	\end{equation}
	for the map $\widehat F\colon (\partial M,g)\to (\sph^{n-1},g_{\sph^{n-1}})$ with nonzero degree. Therefore, by Lemma \ref{lemma:conformalRobin}, a failure of either of the following equalities 
    \begin{enumerate}
        \item $\Sc(g) = 0\text{ on }M,$
        \item $\ H_{\partial M}(g) = \| d\widehat F\|_{\tr} = \|dF\|\cdot(H_{\partial\Omega}(g_{eu})\circ F)\text{ on }\partial M$.
    \end{enumerate}
	leads to a contradiction with Theorem \ref{thm:noncompactScMean}. Therefore, we conclude that $\Sc(g)\equiv 0$, $dF$ is a homothety, and
	$$h\coloneqq \|dF\|=\frac{H_{\partial M}(g)}{H_{\partial\Omega}(g_{eu})\circ F}$$
	is smooth, since $\partial\Omega$ is strictly convex.
	
	The remaining argument is the same as that in the proof of Lemma \ref{lemma:sc->ball}. We only sketch the key steps here. Deforming the metric $g$ to $g+tk$ and computing the variation of the lowest eigenvalue, we obtain that $g$ is Ricci-flat and the second fundamental form of $\partial M$ satisfies
	\begin{equation}
		\second_{\partial M}=\begin{cases}
			0,& \textup{ where } h=0,\\
			h^{-1}F^*(\second_{\partial\Omega}),& \textup{ where } h\ne 0.
		\end{cases}
	\end{equation}
	The right-hand side above is indeed a smooth tensor field, since $h$ is smooth and $F^*(\second_{\partial\Omega})$ is $O(h^2)$ when $h\to 0$. The Codazzi equation then implies that $h$ is equal to one after rescaling if necessary. By solving for $n$ harmonic functions on $M$ whose boundary values are the coordinate functions on $\partial M=\partial\Omega$, we obtain that $M$ is isometric to $\Omega$. This finishes the proof.
\end{proof}
\subsection{Deduction of scalar rigidity on spheres from non-compact scalar-mean comparison} 
We now deduce Theorem \ref{thm:listing} from the noncompact comparison in Theorem~\ref{thm:noncompactScMean}.

\begin{lemma}
	Theorem~\ref{thm:noncompactScMean} implies Theorem~\ref{thm:listing}.
\end{lemma}

\begin{proof}
The dichotomy argument in \cite[Lemma 2.1]{WangXie24} and \cite[Theorem 2.3]{WangWangXie} based on a conformal change shows that if $(M,g)$ in Theorem \ref{thm:listing} fails to be isometric to a sphere, then one finds another metric where the inequality is strict. Adjusting
	the parameter if necessary, we argue by contradiction and assume
	that there exist a closed manifold $(M^n,g)$ and a map
	$F\colon M\to\sph^n$ of non-zero degree such that
	\begin{equation}
		\Sc(g)-\delta\geq(\|dF\|+\delta)^2\cdot n(n-1)+\delta
	\end{equation}
	for some $\delta>0$.
	
	Set $N=(-\infty,0]\times M$ and equip it with the product metric
	$$\overbar g=dt^2+g.$$
	Clearly, $\Sc(\overbar g)=\Sc(g)$, $\partial N=M$, and
	$H(\overbar g)=0$.
	
	We want a conformal metric
	on $N$ with uniformly positive scalar curvature and strict mean curvature comparison, when compared to the standard Euclidean unit ball, as in Lemma ~\ref{lemma:conformalRobin}. We therefore consider the bottom eigenvalue of the
	Robin problem:
	\begin{equation}\label{eq:noncompactNeumann}
		\begin{cases}
			\displaystyle
			-\frac{4n}{n-1}\Delta^{\overbar g}u+\Sc(\overbar g)u=\lambda u
			&\text{in }N^\circ,\\[4pt]
			\displaystyle
			\frac{\partial u}{\partial \nu}=\frac{n-1}{2n}
			\sqrt{\frac{n}{n-1}(\Sc(\overbar g)-\delta)}\,u
			&\text{on }\partial N,
		\end{cases}
	\end{equation}
	where $\nu=\partial_t$ is the outward unit normal vector to $\partial N$.
	
	\begin{claim}\label{claim:ncptRobin}
	    There exist a bounded positive function $u\in C^\infty(N)$ and a number
	$\lambda\geq\delta>0$ satisfying \eqref{eq:noncompactNeumann}.
	\end{claim} 
	
	Assuming the claim with $ 0< u\leq C$ for some $C>0$, consider the conformal metric
	$\overbar g_u=u^{\frac{4}{n-1}}\overbar g$ on $N$.
	The same computation as in Lemma~\ref{lemma:conformalRobin} gives
	\begin{equation}
		\Sc(\overbar g_u)=\lambda u^{-\frac{4}{n-1}}>0,
	\end{equation}
	and, for $0<\delta'\leq\delta$,
	\begin{equation}
		\begin{split}
			H_{\partial N}(\overbar g_u)
			&=u^{-\frac{2}{n-1}}\cdot\frac{2n}{n-1}\frac{\frac{\partial u}{\partial \nu}}{u}
			=u^{-\frac{2}{n-1}}
			\sqrt{\frac{n}{n-1}(\Sc(\overbar g)-\delta)}\\
			&\geq u^{-\frac{2}{n-1}}\cdot n(\|dF\|_g+\delta')
			\geq\|dF\|_{\tr,\overbar g_u}
			+n\delta'u^{-\frac{2}{n-1}}.
		\end{split}
	\end{equation}
	However, $\overbar g_u$ need not be complete. To address this, set
	$$
	w=\varepsilon+u,\qquad
	\overbar g_w=w^{\frac{4}{n-1}}\overbar g
	$$
	for some $\varepsilon >0$. Clearly, $w\geq\varepsilon>0$.
	\begin{itemize}
		\item For every $\varepsilon>0$, $(N,\overbar g_w)$ is complete.
		
		\item The scalar curvature of $\overbar g_w$ is given by
		\begin{equation}
			\begin{split}
				\Sc(\overbar g_w)
				=w^{-\frac{n+3}{n-1}}
				\bigl(\lambda u+\varepsilon\Sc(g)\bigr)
				\geq\delta w^{-\frac{4}{n-1}} \geq \delta (C+\varepsilon)^{-\frac{4}{n-1}} >0.
			\end{split}
		\end{equation}
		
		\item The mean curvature of $\overbar g_w$ satisfies
		\begin{equation}
			\begin{split}
				H_{\partial N}(\overbar g_w)
				&=w^{-\frac{2}{n-1}}\cdot\frac{u}{w}
				\cdot\sqrt{\frac{n}{n-1}(\Sc(g)-\delta)}\\
				&\geq w^{-\frac{2}{n-1}}\frac{u}{w}
				\cdot n(\|dF\|+\delta)\\
				&\geq\|dF\|_{\tr,\overbar g_w}
				+w^{-\frac{n+1}{n-1}}
				\bigl(nu\delta-\varepsilon n\|dF\|\bigr).
			\end{split}
		\end{equation}
	\end{itemize}
	Since $\partial N=M$ is compact,
	$\inf(u|_{\partial N})>0$ and $\|dF\|$ is bounded there.
	Thus the last term is strictly positive for sufficiently small
	$\varepsilon>0$.
	
	It remains to prove the claim. It suffices to verify the positivity of the
	corresponding energy functional, given by
	\begin{equation*}
		Q(\varphi)=\int_N
		\Bigl(\frac{4n}{n-1}|\nabla^{\overbar g}\varphi|^2
		+\Sc(g)\varphi^2\Bigr)d\mathcal H^{n+1}_{\overbar g}
		-2\int_{\partial N}
		\sqrt{\frac{n}{n-1}(\Sc(g)-\delta)}\,
		\varphi^2d\mathcal H^n_g
	\end{equation*}
	for $\varphi\in C_c^\infty(N)$. Since $\Sc(g)$ is independent of
	$t\in(-\infty,0]$, we have
	\begin{equation*}
		-\int_{\partial N}
		\sqrt{\frac{n}{n-1}(\Sc(g)-\delta)}\,
		\varphi^2d\mathcal H^n_g
		=-2\int_N
		\sqrt{\frac{n}{n-1}(\Sc(g)-\delta)}\,
		\varphi(\partial_t\varphi)d\mathcal H^{n+1}_{\overbar g}.
	\end{equation*}
	Using
	$|\nabla^{\overbar g}\varphi|^2
	=|\nabla^g\varphi|^2+(\partial_t\varphi)^2$,
	we obtain
	\begin{equation*}
		\begin{split}
			Q(\varphi)
			={}&\frac{4n}{n-1}\int_N
			\Bigl(
			|\nabla^g\varphi|^2+(\partial_t\varphi)^2
			+\frac{n-1}{4n}(\Sc(g)-\delta)\varphi^2\\
			&\qquad
			-\sqrt{\frac{n-1}{n}(\Sc(g)-\delta)}\,
			\varphi(\partial_t\varphi)
			+\frac{n-1}{4n}\delta\varphi^2
			\Bigr)d\mathcal H^{n+1}_{\overbar g}\\
			={}&\frac{4n}{n-1}\int_N
			\Bigl(
			|\nabla^g\varphi|^2+\frac{n-1}{4n}\delta\varphi^2
			+\Bigl|\partial_t\varphi
			-\sqrt{\frac{n-1}{4n}(\Sc(g)-\delta)}\,\varphi\Bigr|^2
			\Bigr)d\mathcal H^{n+1}_{\overbar g}.
		\end{split}
	\end{equation*}
	Consequently,
	\begin{equation}
		Q(\varphi)\geq \delta 
		\int_N\varphi^2d\mathcal H^{n+1}_{\overbar g}.
	\end{equation} 
    On the other hand, since $\Sc(g)$ is a uniformly bounded function on $N$, we have
    \begin{equation}
        Q(\varphi)\leq C\|\varphi\|_{H^1(N)}^2
    \end{equation}
    for any $\varphi\in C_c^\infty(N)$.

    Therefore, the bottom of the spectrum of the operator $-\frac{4n}{n-1}\Delta^{\overbar g}+\Sc(\overbar g)$ subject to the Robin boundary condition in \eqref{eq:noncompactNeumann} satisfies
	\[
		\lambda_0
		=\inf_{\varphi\in H^1(N)\setminus\{0\}}
		\frac{Q(\varphi)}{\|\varphi\|_{L^2(N)}^2}
		\geq\delta.
	\]
    Minimizing this energy in $H^1(N)$ gives the desired function $u\in C^\infty(N)\cap H^1(N)$ satisfying \eqref{eq:noncompactNeumann}. 
    We show this infimum is an eigenvalue. Let $\mu_0$ and $\psi>0$ be the first eigenvalue and an associated eigenfunction with $\|\psi\|_{L^2(M)} = 1$ of
	\[
		-\frac{4n}{n-1}\Delta^g+\Sc(g)
	\]
	on $M$. For each $s>0$, testing with $\varphi(t,x)=e^{st}\psi(x)$, which clearly lies in $H^1(N) = H^1((-\infty, 0]\times M)$,  gives
	\[
		\begin{aligned}
			\frac{Q(\varphi)}{\|\varphi\|_{L^2(N)}^2}
			={}&\mu_0+\frac{4n}{n-1}s^2 -4s\int_M\sqrt{\frac{n}{n-1}\bigl(\Sc(g)-\delta\bigr)}
			\,\psi^2\,d\mathcal H^n_g.
		\end{aligned}
	\]
	The integral is strictly positive, so this quotient is less than $\mu_0$ for all sufficiently small $s>0$.
	On $N$, separation of variables shows that the essential spectrum starts at $\mu_0$.  Changing the boundary condition at the compact boundary changes the resolvent
by a compact operator \cite{MR138874}.
Consequently, Weyl's theorem
\cite[Theorem~XIII.14]{MRBS78}
implies that the essential spectrum remains the same.  Since $\lambda_0< \mu_0$, it follows that $\lambda_0$ is an eigenvalue. Its eigenfunction can be chosen positive, and elliptic regularity makes it smooth up to $\partial N$. Uniform local elliptic estimates on the product cylinder, together with its finite $L^2$ norm, give
	\[
		0<u\leq C\qquad\text{on }N.
	\]
	This proves Claim \ref{claim:ncptRobin} with $\lambda=\lambda_0$. 
\end{proof}

\section{Weighted scalar-mean formulation}\label{sec:weighted}
In this section, we shall formulate a weighted scalar-mean comparison theorem
(see Theorem~\ref{thm:noncompactScMeanWeighted}) used in the dimension-reduction
argument.  We first establish the two-dimensional
base case in this section. and then the general inductive step will be carried out in Section~\ref{sec:dimensionReduction}.

We first modify the map in Theorem~\ref{thm:noncompactScMean} in preparation
for the capillary construction. The positive gap in the assumption
$H_{\partial M}(g)\geq \|dF\|_\tr+\delta$ allows us to compose $F$ with
a map that collapses sufficiently small neighborhoods of the two poles
of $\mathbb S^{n-1}$ to the corresponding poles, while preserving the same
inequality for the modified map, possibly with a smaller positive
constant $\delta$. This yields the following lemma. See for example
\cite[Lemma 3.1]{WangWangZhu}.
\begin{lemma}\label{lm:collapse}
	Let $\{\pm\}$ be the two poles of $\sph^{n-1}$. In Theorem \ref{thm:noncompactScMean}, we may furthermore assume that $F^{-1}(\pm)=B_\pm\subset \partial M$ are two closed sets with non-empty interior.
\end{lemma}

We now introduce the following definition of weighted scalar curvature.
\begin{definition}\label{def:weightedscalar}
	Let $(M^n,g)$ be a Riemannian manifold. We define
	$$\Sc_{\mu,f}(g)=\Sc(g)-2\Delta f-\mu|\nabla f|^2.$$
\end{definition}

The corresponding weighted mean curvature  is 
\[ H_{\partial M}(g)+\nu(f),\] where $\nu$ is the outward unit normal vector of $\partial M$. With this convention, the weighted scalar-mean comparison takes the following form.

\begin{theorem}\label{thm:noncompactScMeanWeighted}
	Let $(M^n,\partial M,g)$ be a complete Riemannian manifold with possibly non-compact boundary. Denote by $\{\pm\}$ the two poles of $\sph^{n-1}$. Let $F\colon \partial M\to \sph^{n-1}$ be a smooth map and $F^{-1}(\pm)=B_\pm$.  Assume that  every connected component of $B_\pm$ has non-empty interior in $\partial M$, and $\partial M\setminus B_+$ is pre-compact. If there exist $f\in C^\infty(M)$, $\delta>0$ and $\mu>\frac{n-1}{n}$
	such that
\begin{itemize}
	\item $\Sc_{\mu,f}(g)=\Sc(g)-2\Delta f-\mu |\nabla f|^2\geq\delta$, and
	\item $H_{\partial M}(g)-\|dF\|_\tr+\nu (f)\geq \delta$, 
\end{itemize}
then the degree of $F$ is zero.
\end{theorem}

By taking $f\equiv 0$, Theorem \ref{thm:noncompactScMeanWeighted} together with Lemma \ref{lm:collapse} implies Theorem \ref{thm:noncompactScMean}. 

The proof of Theorem~\ref{thm:noncompactScMeanWeighted} proceeds by
dimension reduction. We begin with a two-dimensional version for compact
manifolds, which suffices for the base case of the induction because
the available regularity theory for capillary hypersurfaces rules out
singularities in low dimensions. Also, this version assumes a weaker spectral
condition.
\begin{lemma}\label{lemma:n=2}
		Let $(M^2,\partial M,g)$ be a compact surface with boundary, and let $F\colon \partial M\to \sph^1$ be a smooth map. If there exist $f\in C^\infty(M)$, $\delta>0$ and $\mu>\frac{1}{2}$
	such that
	such that
\[
\begin{split}
0\leq{}&\int_M e^f\left[
|\nabla\varphi|^2+\left(\frac12\Sc_{\mu,f}(g)-\delta\right)\varphi^2
\right]d\mathcal H_g^2\\
&+\int_{\partial M}e^f\big[
H_{\partial M}(g)-\|dF\|+\nu(f)-\delta
\big]\varphi^2\,d\mathcal H_g^1, 
\end{split}
\]
then the degree of $F$ is zero.
\end{lemma}
\begin{proof}
	Let us take $\varphi=e^{-f/2}$. By assumption,
	\begin{equation}
\begin{split}
			0\leq& \int_M(\frac 1 4|\nabla f|^2+\frac 1 2\Sc(g)-\Delta f-\frac 1 2\mu|\nabla f|^2-\delta)d\mathcal H^2_g
		\\&+\int_{\partial M}(H_{\partial M}(g)-\|dF\|+\nu(f)-\delta)d\mathcal H^1_g\\
		\leq &\int_M(\frac 1 2\Sc(g)-\delta)d\mathcal H^2_g+\int_{\partial M}(H_{\partial M}(g)-\|dF\|-\delta)d\mathcal H^1_g,
\end{split}
	\end{equation}
	where the last inequality uses the divergence theorem. By the Gauss--Bonnet theorem, we have 
	\begin{equation}
		\int_M\frac 1 2\Sc(g)d\mathcal H^2_g+\int_{\partial M}H_{\partial M}(g)d\mathcal H^1_g=2\pi\chi(M)\leq 2\pi.
	\end{equation}
	By the definition of the degree, we have
    \[ \int_{\partial M}F^*(d\theta) = \deg(F) \]
    where $d\theta$ is the length form of the circle. To summarize, we have 
	\begin{equation}
		2\pi\geq 2\pi \chi(M) > \int_{\partial M}\|dF\|d\mathcal H^1_g \geq \Big|\int_{\partial M}F^*(d\theta)\Big|=2\pi\cdot |\deg(F)|.
	\end{equation}
	It follows that $|\deg(F)|<1$. Therefore $\deg F=0$.
\end{proof}

\section{dimension reduction for weighted scalar-mean comparison}\label{sec:dimensionReduction}

In this section, we prove Theorem~\ref{thm:noncompactScMeanWeighted}
by dimension reduction. We first outline the argument.
\begin{enumerate}
    \item Starting from an $n$-dimensional manifold, we construct a potential $\Psi$ on $M^n$ that serves as a barrier for the noncompact end. We use this potential to define a weighted energy functional and obtain a compact minimizing capillary bubble $Y^{n-1}$, which may have singularities.
      \item Set $m=n-1=\dim Y$. The Laplacian and Green function are well defined on $Y$, and its singular set is shown to have Assouad dimension at most $m-4$. This bound allows us to conformally
    blow up the singular set, producing a complete smooth manifold. The stability inequality for $Y$ also yields a spectral weighted scalar-mean positivity with threshold $\mu>\frac{m}{m+1}$. We modify the boundary map induced by $F$ so that it maps the noncompact ends of $\partial Y$ to a pole of $\mathbb S^{m}$.
    \item An energy minimization argument on $Y$ then gives pointwise weighted scalar-mean positivity. This reduces the problem from dimension $n$ to dimension $m=n-1$.
	\item Iterating this procedure reduces the problem to dimension
    two, where Lemma~\ref{lemma:n=2} applies.
\end{enumerate}

\subsection{Weighted capillary $\mu$-bubble}
In this subsection, we introduce weighted capillary $\mu$-bubbles,
whose existence follows by adapting the classical theory of capillary
hypersurfaces.

We begin with two lemmas that provide a barrier function $\Psi$ on $M$ with
vanishing normal derivative along $\partial M$. This function will be
used in the construction of the weighted capillary $\mu$-bubbles.
\begin{lemma}\label{lemma:functionZeroDerivative}
    Let $(M^n,g)$ be a complete Riemannian manifold with smooth boundary.
    Then there exists a smooth, proper, nonnegative function
    $h\colon M\to[0,\infty)$ such that
    \[
        |dh|_g\leq1 \quad\text{on }M,
        \qquad
        \partial_\nu h=0 \quad\text{on }\partial M,
    \]
    where $\nu$ is the outward unit normal.
\end{lemma}

\begin{proof}
    Let us choose a collar neighborhood of the boundary in normal coordinates, whose width may vary along
    $\partial M$. In these coordinates, write
    \[
        g=dt^2+g_t,
    \]
    where $g_0$ is the induced metric on $\partial M$. After possibly shrinking
    the collar size pointwise, the metrics $dt^2+g_t$ and $dt^2+g_0$ are
    uniformly equivalent in this collar neighborhood. A smooth interpolation between the two metrics 
    produces a metric $g'$ that equals $dt^2+g_0$ near $\partial M$,
    agrees with $g$ outside the collar neighborhood, and satisfies
    \[
        C^{-1}g\leq g'\leq Cg
    \]
    for some $C\geq1$. Moreover, $g'$ agrees with $g$ as a bilinear
    form on $TM|_{\partial M}$. In particular, $g'$ is complete and
    has the same unit normal vector field as $g$ along $\partial M$.

    Let $DM$ be the double of $M$, equipped with the smooth complete
    metric $\widehat g$ obtained by doubling $g'$. The reflection
    $\iota\colon DM\to DM$ is an isometry. Standard smoothing of
    distance functions (up to rescaling)  gives a
    smooth, proper, nonnegative function $h_1$ on $DM$ with
    $|dh_1|_{\widehat g}\leq1$. Define
    \[
        h_2:=\frac12\bigl(h_1+h_1\circ\iota\bigr).
    \]
    Then $h_2$ is a smooth, proper and nonnegative function on $M$ such that $
        |dh_2|_{\widehat g}\leq1.$
    Moreover, $h_2\circ\iota=h_2$, so its normal derivative vanishes
    along the fixed-point set $\partial M$.

    After a rescaling if necessary, $h \coloneqq  h_2|_{M}$ is the desired function. In  particular,  since the unit normal vectors for $g$ and $g'$ along
    $\partial M$ coincide, we have  $\partial_\nu h=0$.
\end{proof}

\begin{lemma}\label{lemma:barrier}
Assume the hypotheses of Theorem
\ref{thm:noncompactScMeanWeighted}. Let
$\mathbb S^{n-1}_-$ denote the closed hemisphere centered at the
pole $\{-\}$, and set
\[
A:=F^{-1}(\mathbb S^{n-1}_-)\subset\partial M.
\]
For every $\alpha_1,\alpha_2>0$, there exist a relatively open set
$U\subset M$ with compact closure and a function
$\Psi:M\to[-\infty,0]$ such that
\begin{enumerate}
    \item $U=\{\Psi>-\infty\}$, and $\Psi|_U$ is smooth,
    \item on $U$,
    \[
        \alpha_1\Psi^2-|\nabla\Psi|_g+\alpha_2
        \geq \frac{\alpha_2}{2},
    \]
    \item $\partial_\nu\Psi=0$ on $U\cap\partial M$,
    where $\nu$ is the outward unit normal  of $\partial M$,
    \item $\Psi\equiv0$ on a relative open neighborhood of $A$
    in $M$,
    \item $\Psi=-\infty$ on $M\setminus U$, and 
    $\Psi(x)\to-\infty$ as $x\in U$ approaches $\overbar U\backslash U$,  where $\overbar U$ denotes the closure of $U$ in $M$.
\end{enumerate}
In particular, $\Psi=-\infty$ outside a compact subset of $M$.
\end{lemma}

\begin{proof}
We use the same notation as in Theorem \ref{thm:noncompactScMeanWeighted}. Note that  
\[
A\subset\partial M\setminus B_+.
\]
Moreover, $A$ is closed in $\partial M$, hence closed in $M$.
The hypothesis that $\partial M\setminus B_+$ is precompact
therefore implies that $A$ is compact.

By Lemma \ref{lemma:functionZeroDerivative}, there exists a smooth,
proper, nonnegative function $h:M\to[0,\infty)$ satisfying
\[
|\nabla h|_g\leq1,
\qquad
\partial_\nu h=0\quad\text{on }\partial M.
\]

Choose $\zeta\in C^\infty(\mathbb R)$ such that
\[
\zeta(t)=0, \quad \forall t\leq0; \qquad
0<\zeta(t) < 1,\quad \forall t>0; \qquad
\zeta(t)=1,\quad \forall t\geq1.
\]
Define
\[
\varphi(t):=\int_0^t\zeta(s)\,ds,
\qquad
b:=\sqrt{\frac{\alpha_2}{2\alpha_1}},
\qquad
c:=\sqrt{\frac{\alpha_1\alpha_2}{2}}.
\]
There is a unique finite $T>0$ satisfying
\[
\varphi(T)=\frac{\pi}{2c}.
\]
For $t<T$, define
\[
v(t):=b\tan\bigl(cI(t)\bigr).
\]
Then $v\in C^\infty((-\infty,T))$, $v\geq0$, $v(t) = 0$ for all $t\leq 0$ and 
\[
\lim_{t\to  T}v(t)=+\infty.
\]
In particular, $v$ is smooth across $t=0$.

Since $bc=\alpha_2/2$ and $\alpha_1b^2=\alpha_2/2$,
differentiation gives
\[
v'(t)
=bc\,\zeta(t)\sec^2\bigl(c\varphi(t)\bigr)
=\zeta(t)\left(\alpha_1v(t)^2+\frac{\alpha_2}{2}\right).
\]
It follows that 
\[
0\leq v'(t)\leq\alpha_1v(t)^2+\frac{\alpha_2}{2}, 
\qquad \forall t<T.
\]

By compactness of $A$, choose $T_0>0$ such that
$A\subset\{h<T_0\}$. Set
\[
L:=T_0+T,\qquad
U:=\{h<L\},\qquad
K:=\{h\leq L\},
\]
and define
\[
\Psi(x):=
\begin{cases}
-v\bigl(h(x)-T_0\bigr),&x\in U,\\
-\infty,&x\in M\setminus U.
\end{cases}
\]
Properness of $h$ implies that $K$ is compact, and
$\overline U\subset K$. The function $\Psi|_U$ is smooth and vanishes on the relative
open neighborhood $\{h<T_0\}$ of $A$.

On $U$, we have 
\[
\begin{aligned}
|\nabla\Psi|_g
&=v'(h-T_0)|\nabla h|_g \leq v'(h-T_0)\\
&\leq\alpha_1v(h-T_0)^2+\frac{\alpha_2}{2} =\alpha_1\Psi^2+\frac{\alpha_2}{2}.
\end{aligned}
\]
 Likewise,
on $U\cap\partial M$,
\[
\partial_\nu\Psi
=-v'(h-T_0)\,\partial_\nu h=0.
\]
Finally, as $x\in U$ approaches $\overbar U\backslash U$, $h(x)-T_0$ goes to $T$
from below, and hence
\[
\Psi(x)=-v\bigl(h(x)-T_0\bigr)\to -\infty.
\]
This completes the proof. 
\end{proof}

The constants  $\alpha_1,\alpha_2$ will be fixed later, only depending on the given $\mu$, $\delta$ and the dimension $n$ in Theorem \ref{thm:noncompactScMeanWeighted}. 

Recall that $\{-\}$ is a pole of $\sph^{n-1}$. Consider the distance function 
\begin{equation}\label{eq:distance}
    \psi(y)=d_{\sph^{n-1}}(-,y) 
\end{equation}
for $y\in\sph^{n-1}$, 
and define the boundary weight
\begin{equation}\label{eq:capillary-boundary-weight}
	\mu_\partial
	=\cos(\psi\circ F).
\end{equation}
Thus $\mu_\partial$ is smooth, $|\mu_\partial|\leq1$, and
\[
 \{\mu_\partial=1\}=B_-,\qquad
 \{\mu_\partial=-1\}=B_+.
\]

Let $\mathcal C$ be the collection of Caccioppoli sets $\Omega$ in $M$ such that
\begin{itemize}
	\item $\mathcal H^{n-1}_g(\partial^*\Omega\cap\mathring M) <\infty$, where $\partial^*\Omega$ is the reduced boundary of $\Omega$,
	\item $B_-\subset\Omega$, and
	\item $\Omega\subset U = \{\Psi>-\infty\}$.
\end{itemize}
In particular, any such $\Omega$ is relatively compact in $M$.
We consider the functional
\begin{equation}
	\mathcal A(\Omega)=\int_{\partial^*\Omega\cap\mathring M}\rho d\mathcal H^{n-1}_g-\int_{\Omega}\rho \Psi d\mathcal H^{n}_g-\int_{\partial^*\Omega\cap\partial M}\rho\mu_\partial d\mathcal H^{n-1}_g,
\end{equation}
where $\rho=e^f$. 

\begin{lemma}\label{lemma:existence}
	In the metric-measure sense, a minimizer  of the functional $\mathcal A$ exists in the class $\mathcal C$.  A minimizing representative $\Omega$ can be chosen such that
    \begin{enumerate}
        \item there exists an open set $V_-\subset M$ such that
\[
B_-\subset V_-\subset\Omega,
\]
\item $\overbar \Omega$ is compact and $\overbar \Omega\subset U\setminus B_+$.
    \end{enumerate}
Set $
 Y=\overbar{\partial^*\Omega\cap M^\circ}$. Then $Y$ is compact, $Y\cap(B_-\cup B_+)=\emptyset$, and there
exists $\tau>0$ such that
\[
 |\mu_\partial|\leq1-\tau\quad\textup{ on } Y\cap\partial M.
\]
\end{lemma}
\begin{proof}
In the construction of $\Psi$, choose $T_0$ so that
$\overline{\partial M\setminus B_+}\subset\{h<T_0\}$ and the
 $L=T_0+T$ is regular value of $h$. Set $U=\{h<L\}$.
Thus $\overline U$ is compact, $\Psi=0$ near $B_-$, and
$\mu_\partial=-1$ on $\partial M\cap\{h\geq T_0\}$.

Recall that $\rho =e^f$. Set 
\[
\widetilde g=\rho^{2/(n-1)}g,
\qquad
\widetilde\Psi=\rho^{-1/(n-1)}\Psi.
\]
The functional becomes the following unweighted capillary functional:
\[
	\mathcal A(\Omega)
	=\mathcal H_{\widetilde g}^{n-1}
	(\partial^*\Omega\cap\mathring M)
	-\int_\Omega\widetilde\Psi\,d\mathcal H_{\widetilde g}^{n}
	-\int_{\partial^*\Omega\cap\partial M}
	\mu_\partial\,d\mathcal H_{\widetilde g}^{n-1}.
	\]
Moreover, we have 
\begin{equation}\label{eq:strictmean}
    \begin{split}
H_{\partial M}(\widetilde g)-\|dF\|_{\operatorname{tr},\widetilde g}
&=\rho^{-1/(n-1)}
\bigl(H_{\partial M}(g)+\nu(f)-\|dF\|_{\operatorname{tr},g}\bigr)\\
&\geq\rho^{-1/(n-1)}\delta>0.
\end{split}
\end{equation}

Since every member of $\mathcal C$ is supported in the fixed compact set
$\overline U$, the classical existence of minimizer argument applies. See for example 
 \cite[Lemma~2.4]{WangWangZhu}.

The standard local boundary comparisons using the strict inequality \eqref{eq:strictmean} show that a minimizer contains a small neighborhood of $B_-$  and disjoint from some neighborhood of $B_+\cap U$. 
See for example \cite[proof of Theorem~1.3, Steps~4--5]{Wu_capillarysurfaces}
and \cite[Appendix~B, Claim~1]{WangWangZhu} for the comparison argument.

Consequently, $Y = \overbar{\partial^*\Omega\cap M^\circ}$ is compact and $Y\cap(B_-\cup B_+)=\emptyset$.
Continuity of $\mu_\partial$ therefore gives
$|\mu_\partial|\leq1-\tau$ on $Y\cap\partial M$ for some $\tau>0$.
\end{proof}

Fix a minimizer $\Omega$ given by Lemma~\ref{lemma:existence}, and set
\[
Y=\overline{\partial^*\Omega\cap M^\circ},\qquad m=n-1.
\]
Orient the regular part of $Y$ by the unit normal $\nu_Y$ pointing
out of $\Omega$, and let $\nu$ be the outward unit normal of
$\partial M$. 
 All variation formulas in the following lemma are taken on the regular part of
$Y$.

\begin{lemma}\label{lemma:variation}
With the same notation as above, let $\nu$ be the outward unit normal of $\partial M$. Then the following assertions hold.
\begin{enumerate}[label=\textup{(\roman*)}]
\item There is a closed singular set $S\subset Y$ such that\footnote{By \cite[Corollary~1.17]{WangZhangRegularity}, one in fact has $\mathcal H_g^{m-5+\varepsilon}(S)=0$ for every $\varepsilon>0$ and $m\geq5$.}
\[
\mathcal H_g^{m-4+\varepsilon}(S)=0
\quad\text{for every }\varepsilon>0\text{ if }m\geq4,
\]
while $S=\emptyset$ if $m\leq3$. The regular part
$Y_{\mathrm{reg}}:=Y\setminus S$ is a smooth embedded hypersurface
with boundary $
\partial Y_{\mathrm{reg}}
\coloneqq Y_{\mathrm{reg}}\cap\partial M.$

\item On $Y_{\mathrm{reg}}$, its mean curvature satifies 
\[
H_Y(g)+\nu_Y(f)-\Psi=0.
\]
\item Along $\partial Y_{\mathrm{reg}}$, we define the contact angle
		$J\in(0,\pi)$ by $\cos J\coloneqq -\langle\nu_Y,\nu\rangle$.
		Then $\cos J=\mu_\partial=\cos(\psi\circ F),$ thus 
		\begin{equation*}
			J=\psi\circ F.
		\end{equation*}

\item Let $\mathbf n$ be the outward unit normal to $\partial Y_{\mathrm{reg}}$ as the boundary of $\partial M\cap \Omega$. For every smooth function $\varphi$ on $Y_{\mathrm{reg}}$
with compact support in $Y_{\mathrm{reg}}$,
\[
\begin{split}
0\leq{}&\int_{Y_{\mathrm{reg}}}\rho\Bigl[
|\nabla^Y\varphi|^2
-\bigl(\operatorname{Ric}_g(\nu_Y,\nu_Y)+|A_Y|^2
-\nabla_g^2f(\nu_Y,\nu_Y)+\nu_Y(\Psi)\bigr)\varphi^2
\Bigr]\,d\mathcal H_g^m\\
&+\int_{\partial Y_{\mathrm{reg}}}\rho\Bigl[
H_{\partial Y}(g)-\frac{H_{\partial M}(g)}{\sin J}
-(\cot J)H_Y(g)
+\frac{\langle\mathbf n,
\nabla^{\partial M}(\psi\circ F)\rangle}{\sin J}
\Bigr]\varphi^2\,d\mathcal H_g^{m-1}.
\end{split}
\]
Here $\psi$ is the distance function from 
\eqref{eq:distance}, $\nabla_g^2f$ is the ambient Hessian, and
$H_{\partial Y}(g)$ denotes the mean curvature of
$\partial Y_{\mathrm{reg}}\subset Y_{\mathrm{reg}}$.
\end{enumerate}
\end{lemma}

\begin{proof}
	By Lemma \ref{lemma:existence}, on the minimizer $Y$ we have
	$|\mu_\partial|<1$.  The first assertion follows from \cite[Theorem~1.1]{chodosh2024_improvedregularity}
	 and the classical interior
	regularity of minimizing hypersurfaces.  As seen in the
	proof of Lemma \ref{lemma:existence}, the functional can be transformed
	to an unweighted one.  Therefore the first and second variation
	formulas are standard.  See for example \cite[Lemmas~A.3--A.4]{WangWangZhu}. This shows assertions (ii), (iii) and (iv).
\end{proof}

\subsection{Schoen--Yau estimate on capillary bubble}
We now convert the stability inequality of Lemma~\ref{lemma:variation}
into a strict spectral scalar--mean inequality needed for the inductive
step. 

Recall that $m=n-1$ and
\begin{equation}\label{eq:wcomparisons}
    \Sc_{\mu,f}(g)\geq\delta,
\qquad
H_{\partial M}(g)+\nu(f)
\geq\|dF\|_{\operatorname{tr},g}+\delta,
\qquad
\mu>\frac{m}{m+1}.
\end{equation}
where $\nu$ is the outward unit normal of $\partial M$.
Since $m\mu-m+1>0$, we may set
\[
\delta_0:=\frac{(m+1)\mu-m}{2(m\mu-m+1)}>0.
\]
We fix the constants in Lemma~\ref{lemma:barrier} by taking
$\alpha_1=\delta_0$ and $\alpha_2=\delta/4$, and let $\Psi$ be the corresponding function given by Lemma \ref{lemma:barrier}.  In particular, we have 
\begin{equation}\label{eq:sy-barrier-estimate}
\delta_0\Psi^2-|\nabla^M\Psi|+\frac{\delta}{4}\geq0
\end{equation}
on the region where $\Psi$ is finite. Throughout this
subsection, $\Omega$ denotes a minimizer from Lemma \ref{lemma:variation}, using this $\Psi$.  

Recall that $Y = \overbar{\partial^*\Omega\cap M^\circ}$. All calculations below take place on the regular part $Y_{\mathrm{reg}}$ of $Y$. Let us write 
\[
\Sc^Y_{\mu,f}(g)
:=\Sc_Y(g)-2\Delta^Y f-\mu|\nabla^Y f|^2.
\]
The Gauss formula gives us
\begin{equation}
	-\operatorname{Ric}_g(\nu_Y,\nu_Y)-|A_Y|^2=\frac 1 2 \Sc_Y(g)-\frac 1 2 \Sc(g)-\frac 1 2 |A_Y|^2-\frac 1 2 H_Y^2.
\end{equation}
The decomposition of the ambient Laplacian gives 
\begin{equation}
	\nabla_g^2f(\nu_Y,\nu_Y)=-\Delta^Y(f)+\Delta^M(f)-H_Y\nu_Y(f).
\end{equation}
Put $a:=\nu_Y(f)$, so that $H_Y=\Psi-a$ by
Lemma~\ref{lemma:variation}.
Combining the above two equations, we get
\begin{equation}\label{eq:wscalarbound}
	\begin{split}
		&-\operatorname{Ric}_g(\nu_Y,\nu_Y)-|A_Y|^2+\nabla_g^2f(\nu_Y,\nu_Y)-\nu_Y(\Psi)\\
		=&\frac 1 2 \Sc^Y(g)-\frac 1 2 \Sc(g)-\frac 1 2 |A_Y|^2-\frac 1 2 H_Y^2
		-\Delta^Y(f)+\Delta^M(f)-H_Ya-\nu_Y(\Psi)\\
		=&\frac 1 2\Sc^Y_{\mu,f}(g)-\frac 1 2\Sc_{\mu,f}(g)-\Big(\frac{\mu}{2}a^2+\frac 1 2 |A_Y|^2+\frac 1 2 H_Y^2
		+H_Ya\Big)-\nu_Y(\Psi)
	\end{split}
\end{equation}
where we used $|\nabla^M f|^2=|\nabla^Y f|^2+\nu_Y(f)^2 = |\nabla^Y f|^2+a^2$.

The Cauchy--Schwarz inequality gives
\begin{equation}
	|A_Y|^2\geq \frac{H_Y^2}{m}.
\end{equation}
Since $H_Y = \Psi -a$, it follows that 
\begin{align*}
\frac{\mu}{2}a^2+\frac12|A_Y|^2+\frac12H_Y^2+H_Ya
&\geq\frac{\mu}{2}a^2+\frac{m+1}{2m}(\Psi-a)^2+(\Psi-a)a\\
&=\frac{m+1}{2m}\Psi^2-\frac1m\Psi a+\frac{m\mu-m+1}{2m}a^2\\
&=\frac{m\mu-m+1}{2m}\left(a-\frac{\Psi}{m\mu-m+1}\right)^2
  +\delta_0\Psi^2\\
&\geq\delta_0\Psi^2.
\end{align*}
Plugging  the above inequality into \eqref{eq:wscalarbound} and applying \eqref{eq:sy-barrier-estimate}, we obtain 
\begin{equation}
\begin{split}
	&-\operatorname{Ric}_g(\nu_Y,\nu_Y)-|A_Y|^2+\nabla_g^2f(\nu_Y,\nu_Y)-\nu_Y(\Psi)\\
	\leq&~\frac 1 2\Sc^Y_{\mu,f}(g)-\frac 1 2\Sc_{\mu,f}(g)-\delta_0\Psi^2+|\nabla^M \Psi|\\
	\leq&~\frac 1 2\Sc^Y_{\mu,f}(g)-\frac 1 2\Sc_{\mu,f}(g)+\frac{\delta}{4}\leq\frac 1 2\Sc^Y_{\mu,f}(g)-\frac{\delta}{4} .
\end{split}
\end{equation}

Let us consider the boundary term in (iv) of Lemma \ref{lemma:variation}. Using  $H_Y=\Psi-\nu_Y(f)$ and the second inequality of \eqref{eq:wcomparisons},  we have
\begin{equation}\label{eq:bdrymeancomparison}
	\begin{split}
		&\frac{-H_{\partial M}(g)-H_Y\cos J+\langle\mathbf  n,\nabla^{\partial M}(\psi\circ F)\rangle}{\sin J}\\
		\leq&\frac{-\Psi\cos J+\nu_Y(f) \cos J-\|dF\|_{\operatorname{tr},g}+\nu(f)-\delta+\langle\mathbf  n,\nabla^{\partial M}(\psi\circ F)\rangle}{\sin J}
	\end{split}
\end{equation}
Since the contact angle between $Y$ and $\partial M$ is precisely $J$, we have
\begin{equation}
	\nu=-(\cos J)\nu_Y+(\sin J)\nu_{\partial Y}
\end{equation}
where $\nu_{\partial Y}$ is the outward unit normal to $\partial Y_{\mathrm{reg}}$ as the boundary of $Y_{\mathrm{reg}}.$
Thus
\begin{equation}
	\frac{(\cos J)\nu_Y(f)+\nu(f)}{\sin J}=\nu_{\partial Y}(f).
\end{equation}
Combining this with \eqref{eq:bdrymeancomparison}, we obtain 
\begin{equation}\label{eq:SY-boundary-preliminary}
	\begin{aligned}
		&H_{\partial Y}(g)-\frac{H_{\partial M}(g)}{\sin J}
		-(\cot J)H_Y
		+\frac{\langle\mathbf n,\nabla^{\partial M}(\psi\circ F)\rangle}{\sin J}\\
		&\leq H_{\partial Y}(g)
		+\nu_{\partial Y}(f)
		+
		\frac{-\|dF\|_{\tr,g}
		+\langle\mathbf n,\nabla^{\partial M}(\psi\circ F)\rangle
		-\Psi\cos J-\delta}{\sin J}.
	\end{aligned}
\end{equation}

Let $m\geq2$ and let $P\colon\sph^m\setminus\{\pm\}\to\sph^{m-1}$ be the projection onto the equator. In polar coordinates $g_{\sph^m}=d\psi^2+\sin^2 \psi\,g_{\sph^{m-1}}$ with $\psi$ from \eqref{eq:distance},  we have 
\[
	\qquad P(\psi,z)=z.
\]
Define
\begin{equation}\label{eq:boundarymap}
    \hat F=P\circ F|_{\partial Y}.
\end{equation}
The image $F(\partial Y)$ stays a positive distance from the poles $\{\pm\}$, so $\hat F$ is Lipschitz on $\partial Y$ and smooth on its regular part.

With the notation of  Lemma \ref{lemma:variation}, we define
\[
K:=S\cap\partial M.
\]
The orientation of $M$ (resp. $\Omega$)  induces an  orientation on
$\partial M$ (resp. $Y_{\mathrm{reg}}$). The orientation of $Y$ in turn induces an orientation of  $\partial Y_{\mathrm{reg}}$. We orient the equator of $\sph^m$ as the boundary
of the hemisphere containing $\{+\}$. 

If $K$ is nonempty, Lemma~\ref{lemma:variation}(i)
and the fact that $\hat F$ is Lipschitz give
\[
\dim_{\mathcal H}\hat F(K)\leq m-4<m-2.
\]
Consequently, $\mathbb S^{m-1}\setminus\hat F(K)$ is connected.
The restriction
\[
\hat F\colon  \partial Y_{\mathrm{reg}}\setminus\hat F^{-1}(\hat F(K))
\longrightarrow\mathbb S^{m-1}\setminus\hat F(K)
\]
is proper.
Define $\deg(\hat F)$ to be the degree of this proper smooth map.
For $K=\emptyset$, this is the usual degree of the map on $\partial Y_{\mathrm{reg}}$.

\begin{lemma}\label{lemma:FcomposeProjection}
	With the above orientations, we have $\deg(\hat F)=\deg(F)$. Moreover, on $\partial Y_{\mathrm{reg}}$,  we have
\begin{equation}\label{eq:meridional-trace-estimate}
\frac{-\|dF\|_{\operatorname{tr},g}
 +\langle\mathbf n,\nabla^{\partial M}(\psi\circ F)\rangle}
 {\sin J}
\leq-\|d\hat F\|_{\operatorname{tr},g}.
\end{equation}
\end{lemma}
\begin{proof}
	For $p\in \partial Y_{\mathrm{reg}}$, define
\[
e:=\nabla^{\mathbb S^m}\psi\big|_{F(p)}.
\]
Let $\pi$ denote the orthogonal projection onto $\{e\}^\perp$ in
$T_{F(p)}\mathbb S^m$. Since
$T_p\partial Y_{\mathrm{reg}}=\{\mathbf n\}^\perp\subset T_p\partial M$, we may choose
orthonormal bases $u_1,\ldots,u_{m-1}$ of $\{\mathbf n\}^\perp$ and
$v_1,\ldots,v_{m-1}$ of $\{e\}^\perp$ such that 
\[ \|\pi\circ dF_p|_{\{\mathbf n\}^\perp}\|_{\operatorname{tr}}  = \sum_{i}^{m-1} \langle \pi\circ  dF_p (u_i), v_i\rangle.  \] Completing these bases by
$u_m=\mathbf n$ and $v_m=e$, the definition of the trace norm gives
\begin{align*}
\|dF_p\|_{\operatorname{tr}}
&\geq \sum_{i=1}^{m}
       |\langle dF_p(u_i),v_i\rangle|\\
&=|\langle dF_p(\mathbf n),e\rangle|
  +\|\pi\circ dF_p|_{\{\mathbf n\}^\perp}\|_{\operatorname{tr}}\\
&\geq
  \langle\mathbf n,\nabla^{\partial M}(\psi\circ F)\rangle
  +\|\pi\circ dF_p|_{\{\mathbf n\}^\perp}\|_{\operatorname{tr}}.
\end{align*}

Note that $dP(e)=0$ and  $dP|_V$ is
$1/\sin\psi$ times an isometry. Since
$J=\psi\circ F$ by Lemma~\ref{lemma:variation}(iii), we obtain
\[
\|\pi\circ dF_p|_W\|_{\operatorname{tr}}
=\sin J\,\|d\hat F_p\|_{\operatorname{tr}}.
\]
Consequently,
\[
\frac{-\|dF\|_{\operatorname{tr}}
+\langle\mathbf n,\nabla^{\partial M}(\psi\circ F)\rangle}
{\sin J}
\leq-\|d\hat F\|_{\operatorname{tr}}.
\]
	The identity $\deg(\hat F)=\deg(F)$ follows from a direct application of Stokes' theorem to suitable differential forms.
\end{proof}

By construction, $\Psi=0$ wherever $\psi\circ F\leq\pi/2$, and $\Psi\leq0$ everywhere. Hence
\[
	-\Psi\cos(\psi\circ F)=-\Psi \cos J \leq0\quad\text{on }\partial Y.
\]
Using Lemma~\ref{lemma:FcomposeProjection} and $0<\sin J\leq1$ in \eqref{eq:SY-boundary-preliminary}, we obtain
\[
	\begin{aligned}
		&H_{\partial Y}(g)-\frac{H_{\partial M}(g)}{\sin J}
		-(\cot J)H_Y
		+\frac{\langle\mathbf n,\nabla^{\partial M}(\psi\circ F)\rangle}{\sin J}\\
		&\quad\leq H_{\partial Y}(g)
		+\nu_{\partial Y}f
		-\|d\hat F\|_{\tr,g}-\delta.
	\end{aligned}
\]

To summarize, we have obtained the following spectral inequality on $Y$:

\begin{lemma}\label{lemma:2ndVariationImproved}
For every smooth function $\varphi$ with compact support in
$Y_{\mathrm{reg}}$,
\[
\begin{split}
0\leq{}&
\int_{Y_{\mathrm{reg}}}\rho\left[
|\nabla^Y\varphi|^2
+\left(\frac12\Sc^Y_{\mu,f}(g)
-\frac{\delta}{4}\right)\varphi^2
\right]\,d\mathcal H_g^m\\
&+\int_{\partial Y_{\mathrm{reg}}} \rho\left[
H_{\partial Y}(g)-\|d\hat F\|_{\operatorname{tr},g}
+\nu_{\partial Y}(f)-\delta
\right]\varphi^2\,d\mathcal H_g^{m-1}.
\end{split}
\]
Here $f=\log\rho$ and $\mu>m/(m+1)$.
\end{lemma}

If $m=2$, then $Y$ is a compact smooth surface by
Lemmas~\ref{lemma:existence} and~\ref{lemma:variation}(i).
Since $\deg(\hat F)=\deg(F)\neq0$, some connected component
of $Y$ has a boundary map of nonzero degree. On this component,
Lemma~\ref{lemma:2ndVariationImproved} implies the hypothesis
of Lemma~\ref{lemma:n=2}, with $\delta$ replaced by $\delta/4$. This gives a contradiction.
Thus, in the remaining inductive step, we assume $m\geq3$.

\subsection{Volume and singular-set estimates}
\label{subsec:singular-set-inputs}

We collect the estimates used in the conformal construction below. Their proofs are given in Section~\ref{sec:singular-set-estimates}.
Recall $Y_{\mathrm{reg}}=Y\setminus S$ is the regular part of $Y$ and $\partial Y_{\mathrm{reg}}=\partial Y\setminus S$.
Throughout Subsections~\ref{subsec:singular-set-inputs}--\ref{subsec:spectral-to-pointwise},
functions in $C_c^\infty(Y_{\mathrm{reg}})$ are smooth up to the regular boundary,
 Integrals over $Y$ and $\partial Y$
are taken over their regular parts.

Let $d_Y$ be the length metric induced by $g$ on $Y$, with
$d_Y(x,y)=+\infty$ if no rectifiable curve in $Y$ joins $x$ to $y$. We 
write $B_r^Y(x)$ for intrinsic balls in $Y$, $B_r^M(x)$ for ambient balls,
and $N_r^M(E)$ for the least number of open ambient $r$-balls
needed to cover $E$. Compactness of $Y$ implies that $d_Y$ is
lower semicontinuous in the ambient topology. In particular,
the intrinsic balls are Borel sets.

\begin{definition}[Assouad dimension]\label{def:Assouad}
Let $E$ be a subset of a metric space.
Its \emph{Assouad dimension} $\dim_{\mathcal A}E$ is the infimum
of the exponents $q\geq0$ for which there exists a constant $C_q$
such that
\[
N_r\bigl(E\cap B_R(x)\bigr)
\leq C_q\left(\frac Rr\right)^q
\qquad\text{for all }x\in E,\quad 0<r<R.
\]
Here $N_r$ denotes the covering number by open $r$-balls, and
$C_q$ is independent of $x$, $r$, and $R$.
Equivalently, the cardinality of every $r$-separated subset of
$E\cap B_R(x)$ satisfies the same bound, after changing the
constant.
\end{definition}

Assouad dimension controls local covering numbers uniformly over
all locations and scales. For totally bounded sets,
\[
\dim_{\mathcal H}E
\leq\overbar{\dim}_{\mathcal M}E
\leq\dim_{\mathcal A}E,
\]
where $\overbar{\dim}_{\mathcal  M}$ denotes upper Minkowski
dimension. A metric space has finite Assouad dimension if and
only if it is doubling  \cite{Fraser_2020}.
These dimensions can differ even for the compact set
$E=\{0\}\cup\{1/j:j\in\mathbb N\}$:
\[
\dim_{\mathcal H}E=0,\qquad
\overline{\dim}_{\mathcal M}E=\frac12,\qquad
\dim_{\mathcal A}E=1.
\]
Indeed, the points with $N\leq j\leq2N$ have successive gaps
comparable to $N^{-2}$ and lie in an interval of length comparable
to $N^{-1}$. At these scales, local covering numbers are comparable
to $R/r$, despite the smaller global Minkowski dimension.

Since $\rho$ is smooth and positive near
the compact set $Y$, the length metrics on $Y$ induced by $g$ and
$\widetilde g=\rho^{2/(n-1)}g$ are bi-Lipschitz equivalent.
They therefore give the same Assouad dimensions.

Lemma~\ref{prop:localVolume} constants $c,C,r_0>0$
such that, for every $x\in Y$ and $0<r<r_0$,
\[
 cr^m\leq\mathcal H_g^m(B_r^Y(x))\leq Cr^m,
 \qquad
 \mathcal H_g^{m-1}(B_r^Y(x)\cap\partial Y)\leq Cr^{m-1}.
\]

For $m\leq3$, Lemma~\ref{lemma:variation} gives $S=\emptyset$.
Assume $m\geq4$. For every $\varepsilon>0$, there are
$C_\varepsilon,R_\varepsilon>0$ such that the estimates below hold
for all $0<r<R<R_\varepsilon$, uniformly in the centers and separated
sets.

\smallskip
\noindent\emph{Intrinsic packing.}
For $s\in S$ and any $r$-separated set $A\subset S$, meaning that
$d_Y(a,b)\geq r$ for distinct $a,b\in A$, one has
\begin{equation}\label{eq:intrinsic-packing-summary}
 \#\bigl(A\cap B_R^Y(s)\bigr)
 \leq C_\varepsilon(R/r)^{m-4+\varepsilon}
\end{equation}
(see Proposition~\ref{prop:Assouad}).
Consequently, $\dim_{\mathcal A}(S,d_Y)\leq m-4$.

\smallskip
\noindent\emph{Tubular volume and boundary area.}
For every $x\in Y$,
\begin{align}
 \mathcal H_g^m\bigl(B_R^Y(x)
 \cap\{y:d_Y(y,S)\leq r\}\bigr)
 &\leq C_\varepsilon R^{m-4+\varepsilon}r^{4-\varepsilon},
       \label{eq:interior-tubular-summary}\\
 \mathcal H_g^{m-1}\bigl(B_R^Y(x)
 \cap\{y:d_Y(y,S)\leq r\}\cap\partial Y\bigr)
 &\leq C_\varepsilon R^{m-4+\varepsilon}r^{3-\varepsilon}.
       \label{eq:boundary-tubular-summary}
\end{align}
These estimates follow from \eqref{eq:Assouad-tubular} and
Corollary~\ref{cor:boundary-tubular}. The extension to arbitrary
centers is given at the end of Section~\ref{sec:singular-set-estimates}.

\subsection{Conformal blow-up}
Let $Y$ be as in the previous subsection, with singular set $S$ and
regular part $Y_{\mathrm{reg}}=Y\setminus S$. To continue the
dimension-descent argument, we must first equip $Y_{\mathrm{reg}}$
with a complete metric by conformal blow-up. In this subsection, we
derive the conformal transformation formulas for weighted scalar
curvature and weighted mean curvature and determine how the inequality
in Lemma~\ref{lemma:2ndVariationImproved} transforms under such a
blow-up. The blow-up function will be constructed in the next
subsection.

As established above, we assume $m\geq3$. Decreasing $\mu$ increases $\Sc_{\mu,f}(g)$ (see Definition \ref{def:weightedscalar}) and therefore preserves
the inequality of Lemma~\ref{lemma:2ndVariationImproved} on the
already chosen minimizer. Replacing $\mu$ by a smaller value if
necessary, we may thus assume that
\begin{equation}\label{eq:conformal-mu-range}
\frac{m}{m+1}<\mu
\leq\frac{m}{m+1}+\frac{9}{(m+1)(m-2)^2}
=\frac{(m-2)(m-3)+3}{(m-2)^2}.
\end{equation}

Let $W\geq1$ be a smooth function on $Y_{\mathrm{reg}}$, and define
\[
\hat g=W^{4/3}g,\qquad
\hat\rho=\rho W^{-2(m-2)/3},\qquad
\hat f=\log\hat\rho=f-\frac23(m-2)\log W.
\]
A direct computation gives
\begin{equation*}
	|\nabla^{\hat g}\varphi|_{\hat g}^2=W^{-4/3}|\nabla \varphi|^2.
\end{equation*}
and
\begin{equation*}
	\Sc(\hat g)=W^{-4/3}\left(\Sc(g)-\frac 4 3(m-1)\frac{\Delta W}{W}+(m-1)\Big(\frac 4 3-\frac{4}{9}(m-2)\Big)\frac{|\nabla W|^2}{W^2}\right).
\end{equation*}
Moreover, 
\begin{equation}
\begin{split}
		\Delta^{\hat g}\hat f=&W^{- 4/3}\left(\Delta f-\frac 2 3(m-2)\frac{\Delta W}{W}+\frac 2 3(m-2)\langle\nabla f,\frac{\nabla W}{W}\rangle\right.\\
		&+ \left. \frac 2 3(m-2)\Big( 1- \frac 2 3(m-2)\Big)\frac{|\nabla W|^2}{W^2}\right)
\end{split}
\end{equation}
and
\begin{equation}
	|\nabla^{\hat g}\hat f|_{\hat g}^2=W^{-4/3}\left(|\nabla f|^2+\frac 4 9(m-2)^2\frac{|\nabla W|^2}{|W^2|}-\frac 4 3(m-2)\langle\nabla f,\frac{\nabla W}{W}\rangle\right).
\end{equation}

Combine the equations above, we obtain
\begin{equation}\label{eq:weightedSc}
	\begin{split}
		&\Sc_{\mu,\hat f}(\hat g)=\Sc(\hat g)-2\Delta^{\hat g}(\hat f)-\mu|\nabla^{\hat g}(\hat f)|_{\hat g}^2\\
		=&~W^{-4/3}\Big(\Sc_{\mu,f}(g)-\frac 4 3\frac{\Delta W}{W}+\frac 4 3(m-2)(\mu -1)\langle\nabla f,\frac{\nabla W}{W}\rangle +B\frac{|\nabla W|^2}{W^2}\Big).
	\end{split}
\end{equation}
where
\[
B:=\frac49\left((m-2)(m-3)+3-\mu(m-2)^2\right).
\]
The upper bound in \eqref{eq:conformal-mu-range} is precisely
the condition $B\geq0$.

Recall that $\nu_{\partial Y}$ is the outward unit normal of
$\partial Y_{\mathrm{reg}}$ in $(Y_{\mathrm{reg}},g)$, and let
\[
\hat\nu_{\partial Y}=W^{-2/3}\nu_{\partial Y}
\]
be the corresponding unit normal for $\hat g$. We have 
\begin{equation*}
	H_{\partial Y}(\hat g)=W^{-2/ 3}\left(H_{\partial Y}(g)+\frac 2 3(m-1)\frac{\nu_{\partial Y}(W)}{W}\right),
\end{equation*}
\[
	\hat \nu_{\partial Y}(\hat f)=W^{-2/ 3}\left( \nu_{\partial Y}(f)-\frac 2 3(m-2)\frac{\nu_{\partial Y}(W)}{W}\right), 
\]
\[  \|d\hat F\|_{\tr,\hat g}=W^{- 2/3}\|d\hat F\|_{\tr,g}\,.\]
It follows that 
\begin{equation}
	\begin{split}
		&H_{\partial Y}(\hat g)-\|d\hat F\|_{\tr,\hat g}+\nu_{\partial Y}(\log\hat\rho)\\
		=&~
		W^{-2/3}\Big(H_{\partial Y}(g)-\|d\hat F\|_{\tr,g}+\nu_{\partial Y}(f)+\frac 2 3\frac{\nu_{\partial Y}(W)}{W}	
		\Big)
	\end{split}
\end{equation}

 The weighted volume and boundary area measures satisfy
\begin{equation}
	\hat\rho d\mathcal H^{m}_{\hat g}=\rho W^{-\frac 2 3(m-2)}\cdot W^{\frac 2 3 m}d\mathcal H^{m}_{ g}=W^{4/3}\rho d\mathcal H^{m}_{ g},
\end{equation}
and
\begin{equation}
	\hat\rho d\mathcal H^{m-1}_{\hat g}=\rho W^{- \frac 2 3 (m-2)}\cdot W^{\frac 2 3 (m-1)}d\mathcal H^{m-1}_{ g}=W^{2/3}\rho d\mathcal H^{m-1}_{ g}.
\end{equation}

The above discussion gives the following consequence of
Lemma~\ref{lemma:2ndVariationImproved}.

\begin{lemma}\label{lemma:conformalWeightedSpectral}
For every $\varphi\in C_c^\infty(Y_{\mathrm{reg}})$, we have 
\[
\begin{aligned}
0\leq{}&
\int_{Y}\hat\rho\left(
|\nabla^{\hat g}\varphi|_{\hat g}^2
+\big(
\frac12\Sc_{\mu,\hat f}(\hat g)
-\frac{\delta}{4}W^{-4/3}
\big)\varphi^2
\right)\,d\mathcal H_{\hat g}^m\\
&+\int_{\partial Y}\hat\rho\left(
H_{\partial Y}(\hat g)
-\|d\hat F\|_{\operatorname{tr},\hat g}
+\hat\nu_{\partial Y}(\hat f)
-\delta W^{-2/3}
\right)\varphi^2\,d\mathcal H_{\hat g}^{m-1}\\
&-\frac23\int_{Y_{\mathrm{reg}}}\hat\rho W^{-4/3}\left(
-\frac{\Delta W}{W}
+(m-2)(\mu-1)
\langle\nabla f,\frac{\nabla W}{W}\rangle
\right)\varphi^2\,d\mathcal H_{\hat g}^m\\
&-\frac23\int_{\partial Y_{\mathrm{reg}}}\hat\rho W^{-2/3}
\frac{\nu_{\partial Y}(W)}{W}
\varphi^2\,d\mathcal H_{\hat g}^{m-1}.
\end{aligned}
\]
Here $\mu$ satisfies \eqref{eq:conformal-mu-range},
$\hat g=W^{4/3}g$, $\hat\rho=\rho W^{-2(m-2)/3}$,
and $\hat f=\log\hat\rho$.
\end{lemma}

\subsection{Construction of the blow-up function via Green function}
The preceding subsection leaves the conformal factor $W$ to be
chosen. We construct it using the Green function, with growth and
derivative estimates that ensure completeness and control the
interior and boundary terms in the conformal spectral inequality.
If $S=\emptyset$, we take $W=1$ and the construction of this subsection is not needed.
Throughout this subsection, we therefore assume $S\neq\emptyset$,
so that $m\geq4$.

  The third term in the inequality  of Lemma~\ref{lemma:conformalWeightedSpectral} leads to the following  operator:
\begin{equation}
	\mathcal Lu\coloneqq -\Delta^Y u+ (m-2)(\mu-1)\langle\nabla^Y f,\nabla u\rangle.
\end{equation}
This is a smooth linear elliptic operator on $Y_{\mathrm{reg}} = Y \backslash S$.
Since $f = \log \rho$ is smooth on an ambient neighborhood of the compact
set $Y$, the weight
\[
\omega:=\rho^{-(m-2)(\mu-1)}
\]
is bounded above and below by positive constants.
 Moreover, we have
\begin{equation}
\mathcal Lu=-\omega^{-1}\operatorname{div}_Y(\omega\nabla^Yu).
\end{equation}

We will construct $W\geq1$ so that $\hat g=W^{4/3}g$ is complete
and the quantities
\[
W^{-4/3}\left(
\frac{\delta}{4}+\frac23\frac{\mathcal LW}{W}
\right),
\qquad
W^{-2/3}\left(
\delta+\frac23\frac{\nu_{\partial Y}(W)}{W}
\right)
\]
have uniform positive lower bounds. Throughout this subsection, $A\lesssim B$ means $A\leq \alpha B$,
where $\alpha>0$ may change from line to line and depends only on
the fixed geometric data and parameters.

\begin{lemma}\label{lemma:Green}
The operator  $\mathcal L+1$ subject to the Neumann boundary condition  has a nonnegative
symmetric Green kernel $G(x,y)$ with respect to the measure
$\omega\,d\mathcal H_g^m$ such that $G(x,y)$ is smooth off the diagonal
on the regular part $Y_{\mathrm{reg}}$. 
For an interior pole $x\in Y_{\mathrm{reg}}\setminus \partial Y_{\mathrm{reg}}$, it satisfies
\[
(\mathcal L_y+1)G(x,y)=\delta_x \textup{ and }
\nu_{\partial Y}(G(x,\cdot))=0 \textup{ on } \partial Y_{\mathrm{reg}},
\]
where the equation is understood in the weighted weak sense:
\[
\int_{Y_{\mathrm{reg}}}\omega(y)\left[
\langle\nabla_yG(x,y),\nabla\varphi(y)\rangle
+G(x,y)\varphi(y)
\right]\,d\mathcal H_g^m(y)=\varphi(x)
\]
for every $\varphi\in C_c^\infty(\partial Y_{\mathrm{reg}})$.

There exist constants $C_1,C_2,r_0>0$ such that 
    \begin{enumerate}
        \item For  $x,y\in Y_{\mathrm{reg}} $, we have 
\[
0\leq G(x,y)\leq C_1d_Y(x,y)^{2-m}.
\]    
\item For every $x\in Y_{\mathrm{reg}}$ and $0<r<r_0$,
\[
\int_{B_r^Y(x)\cap Y_{\mathrm{reg}}}
G(x,y)\omega(y)\,d\mathcal H_g^m(y)\geq C_2r^2.
\]
\end{enumerate}
The second estimate also holds without the factor $\omega(y)$,
after  choosing a different constant $C_2$.
\end{lemma}
\begin{proof}
The existence of a Green function and its various estimates follow from \cite[Theorem 5.2]{Guo24}
(see also \cite{MR0657523}). Here we briefly
sketch an argument using heat kernel estimates.

Let $dV_\omega=\omega\,d\mathcal H_g^m$. All $L^p$ norms and
inner products below are taken with respect to this measure.
Since $S$ has zero $m$-dimensional measure,
$L^2(Y,dV_\omega)$ is canonically identified with
$L^2(Y_{\mathrm{reg}},dV_\omega)$.
	The symmetric form induced by   $\mathcal L + 1$ (subject to the Neumann boundary condition) is given 
	\begin{equation}
Q(\varphi)=\int_Y (|\nabla\varphi|^2+\varphi^2)dV_\omega,
	\end{equation}
	which satisfies that $Q(\varphi)\gtrsim \|\varphi\|_{L^2}^2$.
	By Lemma \ref{prop:localVolume} and Proposition \ref{prop:Assouad}, the singular set $S$ has vanishing $2$-capacity. There is a unique associated nonnegative self-adjoint operator, still denoted by $\mathcal L+1$,  on $L^2(Y,dV_\omega)$. For details, see for example \cite[Chapter 1]{Fukushima_1994}. 

    Applying \eqref{eq:Assouad-Sobolev} to
$|\varphi|^{2(m-1)/(m-2)}$ and using H\"older's inequality give
\[
\left(
\int_{Y_{\mathrm{reg}}}|\varphi|^{2m/(m-2)}dV_\omega
\right)^{(m-2)/m}
\lesssim Q(\varphi).
\]
Interpolation between $L^1$ and $L^{2m/(m-2)}$ then yields a Nash-type inequality
\[
\|\varphi\|_2^{2+4/m}
\lesssim Q(\varphi)
\|\varphi\|_1^{4/m}.
\]
	By \cite[Theorem 2.1]{MR898496}, the heat operator $P_t=e^{-t\mathcal L}$ satisfies
	\begin{equation}
		\|P_t\|_{L^1\to L^\infty}\lesssim e^t t^{-m/2}.
	\end{equation}
	Thus it follows from \cite[Theorem 1.2]{MR4249776} that the Schwartz kernel $q_t(x,y)$ of the operator $e^{-t\mathcal L}$ on $Y$ exists and satisfies
	\begin{equation}
		0\leq q_t(x,y)\leq C e^t t^{-\frac{m}{2}}\exp(-\frac{d_Y(x,y)^2}{ct})
	\end{equation}
	for some $c>0$. Spectral functinal calculus shows that  
    $P_t(1)=1$, which implies that 
\[
\int_{Y_{\mathrm{reg}}}q_t(x,y)\,dV_\omega(y)=1.
\]
Since $e^{-t(\mathcal L+1)}=e^{-t}P_t$, 
	the Schwartz kernel $p_t(x,y)$ of $e^{-t(\mathcal L +1)}$ on $Y$ exists and satisfies
	\begin{equation}\label{eq:p_t(x,y)}
		0\leq p_t(x,y) \leq C t^{-\frac{m}{2}}\exp(-\frac{d_Y(x,y)^2}{ct})
	\end{equation}
and 
\begin{equation}\label{eq:heatkernelintegral}
\int_{Y_{\mathrm{reg}}}p_t(x,y)\,dV_\omega(y)=e^{-t}.
\end{equation}

	Now the Green function $G(x, y)$ is defined by 
	\begin{equation}
		G(x,y)\coloneqq \int_0^\infty p_t(x,y)dt.
	\end{equation}
	Integrating both sides of \eqref{eq:p_t(x,y)}, we obtain
	\begin{equation}
		0\leq G(x,y)\lesssim d_Y(x,y)^{2-m}
	\end{equation}
	as desired. See also \cite[Theorem 5.2 (iii)]{Guo24}.

    By Lemma~\ref{prop:localVolume}, the boundedness of $\omega$,
and the finite volume of $Y$, after enlarging the constant we have
\[
V_\omega(B_s^Y(x))\leq C_Vs^m
\qquad\text{for every }s>0.
\]
For $j\geq0$, set
\[
A_j=\{y\in Y_{\mathrm{reg}}:2^jr\leq d_Y(x,y)<2^{j+1}r\}.
\]
By \eqref{eq:p_t(x,y)}, for $r>0$ and $0<t\leq r^2$,
\[
\begin{aligned}
\int_{A_j}p_t(x,y)\,dV_\omega(y)
&\leq C_0t^{-m/2}(2^{j+1}r)^m
       \exp\!\left(-\frac{4^jr^2}{c_0t}\right)\\
&\leq C_1
       \exp\!\left(-\frac{4^jr^2}{2c_0t}\right),
\end{aligned}
\]
where the last inequality follows from
$z^{m/2}e^{-z}\leq C_me^{-z/2}$.

Set $c_1=2c_0$. For $a=r^2/(c_1t)\geq1/c_1$, we have
\[
\sum_{j=0}^{\infty}e^{-4^ja}
\leq \frac{e^{-a}}{1-e^{-3a}}
\leq \frac{e^{-a}}{1-e^{-3/c_1}}.
\]
Consequently, summing over $A_j$ gives
\[
\int_{Y_{\mathrm{reg}}\setminus B_r^Y(x)}p_t(x,y)\,dV_\omega(y)
\leq C_2\exp\!\left(-\frac{r^2}{c_1t}\right),
\qquad 0<t\leq r^2.
\]

Choose $0<\alpha\leq1$ such that
\[
C_2\exp\!\left(-\frac{1}{c_1\alpha}\right)\leq\frac14,
\]
and then choose $r_0>0$ such that
$\alpha r_0^2\leq\log(4/3)$.
For $0<r<r_0$ and $0<t\leq\alpha r^2$, the  identity \eqref{eq:heatkernelintegral}
gives 
\[
\begin{aligned}
\int_{B_r^Y(x)\cap Y_{\mathrm{reg}}}p_t(x,y)\,dV_\omega(y)
&=e^{-t}
  -\int_{Y_{\mathrm{reg}}\setminus B_r^Y(x)}p_t(x,y)\,dV_\omega(y)\\
&\geq\frac34-\frac14=\frac12.
\end{aligned}
\]
Integrating along $0\leq t\leq \alpha r^2$, we obtain that, we obtain
\[
\begin{aligned}
\int_{B_r^Y(x)\cap Y_{\mathrm{reg}}}G(x,y)\,dV_\omega(y)
&\geq
\int_0^{\alpha r^2}
\int_{B_r^Y(x)\cap Y_{\mathrm{reg}}}p_t(x,y)\,dV_\omega(y)\,dt\\
&\geq\frac{\alpha}{2}r^2.
\end{aligned}
\]
The unweighted estimate also follows because $\omega$ is bounded above and below by positive constants.
\end{proof}

We next integrate the above Green function against interior and boundary densities to
construct a function with the required growth rate and boundary derivative estimate
for the conformal construction.

\begin{lemma}\label{lemma:integralOfGreen}
	Assume that the Assouad dimension of $S$ is at most $m-k$, where
	$k>2$. Let $r(x)=d_Y(x,S)$. For every $0<\varepsilon<k-2$, there exist a nonnegative function $u\in C^\infty(Y_{\mathrm{reg}})$, smooth
up to $\partial Y_{\mathrm{reg}}$, and constants $c,C,r_0>0$ such that,
\begin{enumerate}
\item $c\,r(x)^{-k+2+\varepsilon}\leq u(x)\leq Cr(x)^{-k+2+\varepsilon}$ whenever
$0<r(x)<r_0$,
\item $c\,r(x)^{-k+\varepsilon}\leq\mathcal Lu(x)\leq Cr(x)^{-k+\varepsilon}$
whenever $0<r(x)<r_0$,
\item on $\partial Y_{\mathrm{reg}}\cap\{0<r<r_0\}$,
\[
\frac{\nu_{\partial Y}(u)}{u}\geq\frac{c}{r}.
\]
\end{enumerate}
Here $\nu_{\partial Y}$ denotes the outward unit normal
of $\partial Y_{\mathrm{reg}}$ in $Y_{\mathrm{reg}}$.
\end{lemma}
\begin{proof}
Let us denote 
\[
dV_\omega=\omega\,d\mathcal H_g^m,\qquad
dA_\omega=\omega\,d\mathcal H_g^{m-1}.
\]
Choose $0<\tau<\varepsilon$.
By Lemma~\ref{prop:localVolume} and the covering arguments in
Proposition~\ref{prop:Assouad} and
Corollary~\ref{cor:boundary-tubular}, there exist $C,R_\tau>0$
such that, for every $x\in Y$ and $0<t\leq R<R_\tau$,
\begin{align}
\int_{B_R^Y(x)\cap\{r\leq t\}}dV_\omega
&\leq C R^{m-k+\tau}t^{k-\tau},
\label{eq:Green-volume-tube}\\
\int_{\partial Y_{\mathrm{reg}}\cap B_R^Y(x)\cap\{r\leq t\}}dA_\omega
&\leq C R^{m-k+\tau}t^{k-1-\tau}.
\label{eq:Green-boundary-tube}
\end{align}
Fix $0<r_1<\min\{R_\tau, r_0\}$, where $r_0$ is the constant from
Lemmas~\ref{lemma:Green} and \ref{prop:localVolume}.
Choose a smooth positive function $\sigma$ on $Y_{\mathrm{reg}}$,
smooth up to $\partial Y_{\mathrm{reg}}$, such that
\begin{equation}\label{eq:Green-sigma}
\frac12\min\{r(x),r_1\}\leq\sigma(x)
\leq2\min\{r(x),r_1\}.
\end{equation}
Such a function can be obtained by a smooth partition of unity.

With $G$ as in Lemma \ref{lemma:Green}, we define
\begin{equation}\label{eq:u=}
u(x)=
\int_{Y_{\mathrm{reg}}}G(x,y)\sigma(y)^{-k+\varepsilon}\,dV_\omega(y)
+\int_{\partial Y_{\mathrm{reg}}} G(x,y)\sigma(y)^{-k+1+\varepsilon}\,dA_\omega(y).
\end{equation}
Since $G\geq0$, we have $u\geq0$.

We first prove the lower bound for $u$.
If $0<r(x)<r_1$ and $y\in B_{r(x)/4}^Y(x)$, then
\[
r(y)\leq r(x)+d_Y(x,y)\leq\frac54r(x),
\qquad
\sigma(y)\leq2r(y)\leq\frac52r(x).
\]
Part (2) of Lemma~\ref{lemma:Green} gives
\begin{equation}\label{eq:Green-potential-lower}
\begin{aligned}
u(x)
&\geq c r(x)^{-k+\varepsilon}
\int_{B_{r(x)/4}^Y(x)}G(x,y)\,dV_\omega(y)\\
&\geq c r(x)^{-k+2+\varepsilon}.
\end{aligned}
\end{equation}

For the upper bound of $u$, let $0<R<r_1$ and $j\geq0$.
On the set $T_j = \{2^{-j-1}R<r\leq2^{-j}R\}$, \eqref{eq:Green-sigma} gives
\[
2^{-j-2}R\leq\sigma\leq2^{-j+1}R.
\]
Using \eqref{eq:Green-volume-tube} and
\eqref{eq:Green-boundary-tube}, we obtain
\begin{align}
\int_{B_R^Y(x)\cap T_j}
\sigma^{-k+\varepsilon}\,dV_\omega
&\leq
C(2^{-j}R)^{-k+\varepsilon}
R^{m-k+\tau}(2^{-j}R)^{k-\tau} \notag\\
& =C2^{-j(\varepsilon-\tau)}R^{m-k+\varepsilon},
\label{eq:Green-interior-layer}\\
\int_{\partial Y_{\mathrm{reg}}\cap B_R^Y(x)\cap T_j}
\sigma^{-k+1+\varepsilon}\,dA_\omega
&\leq
C(2^{-j}R)^{-k+1+\varepsilon}
R^{m-k+\tau}(2^{-j}R)^{k-1-\tau} \notag\\
& =C2^{-j(\varepsilon-\tau)}R^{m-k+\varepsilon}.
\label{eq:Green-boundary-layer}
\end{align}
On $\{r>R\}$, we have $\sigma\geq R/2$, so
Lemma~\ref{prop:localVolume} gives
\begin{align*}
\int_{B_R^Y(x)\cap\{r>R\}}\sigma^{-k+\varepsilon}\,dV_\omega
&\leq C R^{-k+\varepsilon}R^m
=C R^{m-k+\varepsilon},\\
\int_{\partial Y_{\mathrm{reg}}\cap B_R^Y(x)\cap\{r>R\}}
\sigma^{-k+1+\varepsilon}\,dA_\omega
&\leq C R^{-k+1+\varepsilon}R^{m-1}
=C R^{m-k+\varepsilon}.
\end{align*}
Since $\sum_{j=0}^\infty2^{-j(\varepsilon-\tau)}<\infty$,
summing these inequalities yields
\begin{equation}\label{eq:Green-source-ball}
\int_{B_R^Y(x)}\sigma^{-k+\varepsilon}\,dV_\omega
+\int_{\partial Y_{\mathrm{reg}}\cap B_R^Y(x)}
\sigma^{-k+1+\varepsilon}\,dA_\omega
\leq C R^{m-k+\varepsilon}.
\end{equation}
A finite cover of $Y$ by balls of a fixed radius less than $r_1$ gives
\[
\int_{Y_{\mathrm{reg}}}\sigma^{-k+\varepsilon}\,dV_\omega
+\int_{\partial Y_{\mathrm{reg}}}\sigma^{-k+1+\varepsilon}\,dA_\omega<\infty.
\]
Thus, after increasing $C$, \eqref{eq:Green-source-ball}
holds for every $R>0$.

Fix $x\in Y_{\mathrm{reg}}$ and set $d=\min\{r(x),r_1\}$.
For $y\in B_{d/2}^Y(x)$, the triangle inequality and
\eqref{eq:Green-sigma} give
\begin{equation}\label{eq:Green-local-sigma}
\left|\min\{r(y),r_1\}-d\right|<\frac d2,
\qquad
\frac d4\leq\sigma(y)\leq3d.
\end{equation}
For $0<s<r_1$, taking the upper bound of $G(x, y)$ from Lemma \ref{lemma:Green}(1) and summing over the annuli
$\{y: 2^{-\ell-1}s\leq d_Y(x,y)<2^{-\ell}s\}$, $\ell\geq0$,
it follows from  Lemma~\ref{prop:localVolume} that
\begin{align}
\int_{B_s^Y(x)}G(x,y)\,dV_\omega(y)
&\leq C\sum_{\ell=0}^\infty(2^{-\ell}s)^2
\leq Cs^2,
\label{eq:Green-local-volume-kernel}\\
\int_{\partial Y_{\mathrm{reg}}\cap B_s^Y(x)}G(x,y)\,dA_\omega(y)
&\leq C\sum_{\ell=0}^\infty2^{-\ell}s
\leq Cs.
\label{eq:Green-local-boundary-kernel}
\end{align}
Consequently,
\begin{align}
\int_{B_{d/2}^Y(x)}
G(x,y)\sigma(y)^{-k+\varepsilon}\,dV_\omega(y)
&\leq C d^{-k+\varepsilon}d^2
=C d^{-k+2+\varepsilon},
\label{eq:Green-near-interior}\\
\int_{\partial Y_{\mathrm{reg}}\cap B_{d/2}^Y(x)}
G(x,y)\sigma(y)^{-k+1+\varepsilon}\,dA_\omega(y)
&\leq C d^{-k+1+\varepsilon}d
=C d^{-k+2+\varepsilon}.
\label{eq:Green-near-boundary}
\end{align}

For $j\geq0$, set
\[
A_j=\{y:2^{j-1}d\leq d_Y(x,y)<2^jd\}.
\]
It follows from \eqref{eq:Green-source-ball} and Lemma~\ref{lemma:Green} that 
\begin{equation}\label{eq:Green-far-annulus}
\begin{aligned}
&\int_{A_j}G(x,y)\sigma(y)^{-k+\varepsilon}\,dV_\omega(y)
+\int_{\partial Y_{\mathrm{reg}}\cap A_j}
G(x,y)\sigma(y)^{-k+1+\varepsilon}\,dA_\omega(y)\\
&\qquad\leq C(2^jd)^{2-m}(2^jd)^{m-k+\varepsilon}
=C2^{-j(k-2-\varepsilon)}d^{-k+2+\varepsilon}.
\end{aligned}
\end{equation}
Combining equations 
\eqref{eq:Green-near-interior}--\eqref{eq:Green-far-annulus} gives
\begin{equation}\label{eq:Green-potential-upper}
u(x)\leq C d^{-k+2+\varepsilon}
\left(1+\sum_{j=0}^\infty2^{-j(k-2-\varepsilon)}\right)
\leq C\min\{r(x),r_1\}^{-k+2+\varepsilon}<\infty.
\end{equation}
This shows that $u$ is well defined and also completes the proof of part (1).

We next prove part (2) and (3). By standard approximation using smooth cutoff functions, with the
boundary integral justified by symmetry of $G$, we obtain smooth
approximating solutions. The bound \eqref{eq:Green-potential-upper}
and local elliptic regularity imply convergence to $u$ smoothly on
compact subsets of $Y_{\mathrm{reg}}$, including its regular boundary.
Thus $u$ is smooth up to $\partial Y_{\mathrm{reg}}$ and satisfies
\begin{equation}\label{eq:Green-potential-equations}
\begin{aligned}
(\mathcal L+1)u&=\sigma^{-k+\varepsilon}
&&\textup{ on }Y_{\mathrm{reg}},\\
\nu_{\partial Y}(u)&=\sigma^{-k+1+\varepsilon}
&&\textup{ on }\partial Y_{\mathrm{reg}}.
\end{aligned}
\end{equation}
Finally, \eqref{eq:Green-sigma} and
\eqref{eq:Green-potential-upper} give
\[
\frac{u(x)}{\sigma(x)^{-k+\varepsilon}}\leq Cr(x)^2
\qquad \forall 0<r(x)<r_1.
\]
Choose $0<r_0<r_1$ such that $Cr_0^2\leq1/2$.
For $0<r(x)<r_0$, \eqref{eq:Green-potential-equations} yields
\begin{equation}\label{eq:Green-Lu-bounds}
c r(x)^{-k+\varepsilon}
\leq \mathcal Lu(x)=\sigma(x)^{-k+\varepsilon}-u(x)
\leq C r(x)^{-k+\varepsilon}.
\end{equation}
For $x\in\partial Y_{\mathrm{reg}}$ with $0<r(x)<r_0$, we also have
\begin{equation}\label{eq:Green-normal-ratio}
\frac{\nu_{\partial Y}(u)(x)}{u(x)}
=\frac{\sigma(x)^{-k+1+\varepsilon}}{u(x)}
\geq\frac{c r(x)^{-k+1+\varepsilon}}
{C r(x)^{-k+2+\varepsilon}}
\geq\frac{c}{r(x)}.
\end{equation}
Together with \eqref{eq:Green-potential-lower} and
\eqref{eq:Green-potential-upper}, these prove the lemma.
\end{proof}

If the singular set $S$ of $Y$ is empty, then there is no need to perform the conformal change in Lemma \ref{lemma:conformalWeightedSpectralComplete} below. Assume henceforth that $S\neq\emptyset$, and set
\[
r(x)=d_Y(x,S).
\]
Choose $k=4$ and $\varepsilon=1/2$ in
Lemma~\ref{lemma:integralOfGreen}. There are constants
$c,C,r_0>0$ such that
\begin{equation}\label{eq:critical-blowup-input}
c r^{-3/2}\leq u\leq C r^{-3/2},\qquad
\frac{\mathcal Lu}{u}\geq\frac{c}{r^2},\qquad
\frac{\nu_{\partial Y}(u)}{u}\geq\frac{c}{r}
\end{equation}
for $0<r<r_0$.
Here we use the same $\mu$ as in \eqref{eq:conformal-mu-range}.

Choose a smooth cutoff $\chi$ on $Y_{\mathrm{reg}}$, smooth up to
$\partial Y_{\mathrm{reg}}$, satisfying
\[
0\leq\chi\leq1,\qquad
\chi=1\ \text{for }r<r_0/2,\qquad
\chi=0\ \text{for }r\geq3r_0/4,
\]
and define
\begin{equation}\label{eq:complete-conformal-factor}
W=1+\eta\chi u.
\end{equation}
for $\eta>0$.
 The metric $\hat g=W^{\frac 4 3}g$ on $ Y_{\mathrm{reg}} = Y\setminus S$ is complete by part (1) of Lemma \ref{lemma:integralOfGreen} for any $\eta>0$.
We have the following lemma. 

\begin{lemma}\label{lemma:conformalWeightedSpectralComplete}
	There exists a smooth function $W\in C^\infty(Y\setminus S)$ such that $\hat g=W^{\frac 4 3}g$ on $Y_{\mathrm{reg}}$ is complete. Moreover, there exists $\delta'>0$ such that, for every
$\varphi\in C_c^\infty(Y_{\mathrm{reg}})$, 
\[
\begin{aligned}
0\leq{}&
\int_{Y_{\mathrm{reg}}}\hat\rho\left[
|\nabla^{\hat g}\varphi|_{\hat g}^2
+\left(\frac12\Sc_{\mu,\hat f}(\hat g)-\frac{\delta}{8}\right)
\varphi^2\right]\,d\mathcal H_{\hat g}^m\\
&+\int_{\partial Y_{\mathrm{reg}}}\hat\rho\left[
H_{\partial Y}(\hat g)-\|d\hat F\|_{\operatorname{tr},\hat g}
+\hat\nu_{\partial Y}(\hat f)
-\frac{\delta}{2}-\delta'W^{-2/3}r^{-1}
\right]\varphi^2\,d\mathcal H_{\hat g}^{m-1}.
\end{aligned}
\]
Here
\[
\hat\rho=\rho W^{-2(m-2)/3},\qquad
\hat f=\log\hat\rho,\qquad
\hat\nu_{\partial Y}=W^{-2/3}\nu_{\partial Y},
\]
and $\mu$ satisfies \eqref{eq:conformal-mu-range}.
\end{lemma}
\begin{proof}
The completeness of the metric $\hat g$ follows part (1) of Lemma \ref{lemma:integralOfGreen}.

By Lemma~\ref{lemma:conformalWeightedSpectral}, the spectral
inequality follows once we prove
\begin{align}
\frac{\delta}{4}+\frac23\frac{\mathcal LW}{W}
&\geq\frac{\delta}{8}W^{4/3},
\label{eq:critical-uniform-interior}\\
\delta+\frac23\frac{\nu_{\partial Y}(W)}{W}
&\geq\frac{\delta}{2}W^{2/3}+\frac{\delta'}r
\quad \textup{ on }\partial Y_{\mathrm{reg}}.
\label{eq:critical-uniform-boundary}
\end{align}

For $r<r_0/2$, we have
$\mathcal L(\chi u)=\mathcal Lu\geq0$ and
$\nu_{\partial Y}(\chi u)=\nu_{\partial Y}u\geq0$.
The set $\{r_0/2\leq r\leq3r_0/4\}$ is compact in
$Y_{\mathrm{reg}}$, and
$\chi u=0$ for $r\geq3r_0/4$.
After increasing the constant $C$ from \eqref{eq:critical-blowup-input}, we  have
\begin{equation}\label{eq:complete-cutoff-bounds}
\begin{aligned}
\mathcal L(\chi u)&\geq-C
&&\text{on }Y_{\mathrm{reg}},\\
\nu_{\partial Y}(\chi u)&\geq-C
&&\text{on }\partial Y_{\mathrm{reg}},\\
\chi u&\leq C
&&\text{when }r\geq r_0/2.
\end{aligned}
\end{equation}

Choose $\eta>0$ so small that
\[
\begin{gathered}
C\eta\leq\min\left\{\frac14,\frac{3\delta}{32}\right\},
\qquad
(4c\eta)^{2/3}<\frac{r_0}{2}, \qquad 
\frac{\delta}{8}(5C\eta)^{4/3}\leq\frac{2c}{15},
\qquad
\frac{\delta}{2}(5C\eta)^{2/3}\leq\frac{c}{15},
\end{gathered}
\]
and then choose
\[
0<\delta'\leq
\min\left\{\frac{c}{15},
\frac{\delta}{8}(4c\eta)^{2/3}\right\}.
\]
By \eqref{eq:complete-cutoff-bounds} and $W\geq1$, we have 
\begin{equation}\label{eq:complete-global-errors}
\frac23\frac{\mathcal LW}{W}\geq-\frac{\delta}{16},
\qquad
\frac23\frac{\nu_{\partial Y}(W)}{W}\geq-\frac{\delta}{16}.
\end{equation}

First suppose $W>5/4$.
The choice $C\eta\leq1/4$ and
\eqref{eq:complete-cutoff-bounds} imply $r<r_0/2$.
Hence $\chi=1$, $\eta u>1/4$, and
\[
W=1+\eta u\leq 5\eta u\leq 5C\eta r^{-3/2},
\qquad
\frac{\eta u}{W}\geq\frac15.
\]
Using \eqref{eq:critical-blowup-input} and the choices of
$\eta$ and $\delta'$, we obtain
\begin{align*}
\frac{\delta}{8}W^{4/3}
&\leq\frac{2c}{15r^2}
\leq\frac23\frac{\mathcal LW}{W},\\
\frac{\delta}{2}W^{2/3}+\frac{\delta'}r
&\leq\frac{2c}{15r}
\leq\frac23\frac{\nu_{\partial Y}W}{W}.
\end{align*}
Thus \eqref{eq:critical-uniform-interior} and
\eqref{eq:critical-uniform-boundary} hold when $W>5/4$.

Now suppose $W\leq5/4$.
If $r<r_0/2$, then
\[
c\eta r^{-3/2}\leq\eta u=W-1\leq\frac14,
\]
so $r\geq(4c\eta)^{2/3}$.
If $r\geq r_0/2$, then our choice of $\eta$ gives directly 
\[   r\geq \frac{r_0}{2} > (4c\eta)^{2/3}.\]
Consequently,
\[
\frac{\delta'}r\leq\frac{\delta}{8}.
\]
By \eqref{eq:complete-global-errors},
\begin{align*}
\frac{\delta}{4}+\frac23\frac{\mathcal LW}{W}
&\geq\frac{3\delta}{16}
\geq\frac{\delta}{8}\left(\frac54\right)^{4/3}
\geq\frac{\delta}{8}W^{4/3},\\
\delta+\frac23\frac{\nu_{\partial Y}W}{W}
&\geq\frac{15\delta}{16}
\geq\frac{\delta}{2}\left(\frac54\right)^{2/3}
+\frac{\delta}{8}
\geq\frac{\delta}{2}W^{2/3}+\frac{\delta'}r.
\end{align*}
This proves \eqref{eq:critical-uniform-interior} and
\eqref{eq:critical-uniform-boundary} also when $W\leq5/4$,
and completes the proof.
\end{proof}

\subsection{Modification of the map} 
We now modify the map $\hat F$ in \eqref{eq:boundarymap} so that it is constant outside a
compact set, while preserving its degree and a strict spectral
inequality. Throughout this subsection, write
\[
r(x)=d_Y(x,S),
\]
 If the singular set $S=\emptyset $, then one can proceed directly to
Lemma~\ref{lemma:completeSpectralMap}. So for Lemma \ref{lemma:modifiedMap} below, we assume $S\neq \emptyset$.

Since $S\cap\partial Y$ is compact and $\hat F\colon \partial Y \to \sph^{m-1}$ is Lipschitz,
Lemma~\ref{lemma:variation} gives
\[
\dim_{\mathcal H}\hat F(S\cap\partial Y)
\leq \dim_{\mathcal H}(S\cap\partial Y)
\leq m-4<m-1.
\]
Thus $\hat F(S\cap\partial Y)$ is a proper compact subset
of $\sph^{m-1}$. Choose a regular value $p_-$ outside $F(S\cap\partial Y)$ and let $p_+$ be its antipodal point. Choose $R_0>0$ and $\theta_0\in(0,\pi/2)$ such that
\begin{equation}\label{eq:map-pole-separation}
d_{\mathbb S^{m-1}}(\hat F(x),p_-)>\theta_0
\quad\textup{ for $x\in\partial Y$ with }  r(x)\leq R_0.
\end{equation}

\begin{lemma}\label{lemma:modifiedMap}
There is a smooth map
$F_1:\partial Y_{\mathrm{reg}}\to\mathbb S^{m-1}$ such that
\begin{enumerate}
\item $F_1=\hat F$ on $\{r\geq R_0\}$,
\item $F_1=p_+$ outside a compact subset of
$\partial Y_{\mathrm{reg}}$,
\item on $\partial Y_{\mathrm{reg}}$,
\begin{equation}\label{eq:modified-map-differential}
\|dF_1\|_{\operatorname{tr},\hat g}
\leq \|d\hat F\|_{\operatorname{tr},\hat g}
+\frac{\delta'}2 W^{-2/3}r^{-1},
\end{equation}
\item $\deg(F_1)=\deg(\hat F)\ne0$.
\end{enumerate}
\end{lemma}

\begin{proof}
Since $\hat F$ is
Lipschitz on $\partial Y$, there is $\ell\geq1$ such that
\begin{equation}\label{eq:map-original-bound}
\|d\hat F\|_{\operatorname{tr},g}\leq\ell
\quad\text{on }\partial Y_{\mathrm{reg}}.
\end{equation}
Since $\hat g=W^{\frac 4 3}g$, it follows that 
\[  \|d\hat F\|_{\operatorname{tr},\hat g}\leq W^{-2/3}\ell
\quad\text{on }\partial Y_{\mathrm{reg}}. \]

	Let $Z:\mathbb S^{m-1}\setminus\{p_-\}\to\mathbb R^{m-1}$
be stereographic projection, with $Z(p_+)=0$, and define
\[
H(t,z)=Z^{-1}\bigl((1-t)Z(z)\bigr),
\qquad 0\leq t\leq1.
\]
Then $H(0,z)=z$, $H(1,z)=p_+$, and $H(t,z)\ne p_-$.
On $[0,1]\times
(\mathbb S^{m-1}\setminus B_{\theta_0}(p_-))$, choose $C\geq1$ such that
\begin{equation}\label{eq:map-homotopy-bounds}
\|\partial_tH\|\leq C,\qquad
\|d(H(t,\cdot))\|\leq C.
\end{equation}

Recall that $r$ is the function given by $r(x) = d_Y(x, S)$. Choose $\sigma_1\in C^\infty(\partial Y_{\mathrm{reg}})$ satisfying
\begin{equation}\label{eq:map-smoothed-distance}
\frac12\min\{r,R_0\}\leq\sigma_1\leq2\min\{r,R_0\},
\qquad \|d\sigma_1\|_g\leq2.
\end{equation}

We choose $r_+$ and $r_-$ such that 
\begin{equation}\label{eq:map-cutoff-radii}
0<r_+<\min\left\{\frac{R_0}{4},\frac{\delta'}{8C\ell}\right\},
\qquad
0<r_-<r_+,\qquad
\log\frac{r_+}{r_-}\geq\frac{32C}{\delta'}.
\end{equation}
Let $\beta\in C^\infty(\mathbb R)$ satisfy
\[
0\leq\beta\leq1,\qquad
\beta=1\text{ on }(-\infty,0],\qquad
\beta=0\text{ on }[1,\infty),\qquad
|\beta'|\leq2,
\]
and set
\[
\chi_1(x)\coloneqq 
\beta\left(\frac{\log(\sigma_1(x)/r_-)}
{\log(r_+/r_-)}\right).
\]
Then $\chi_1=1$ when $\sigma_1\leq r_-$ and
$\chi_1=0$ when $\sigma_1\geq r_+$.
On $\{\sigma_1<r_+\}$, the inequality \eqref{eq:map-smoothed-distance} gives
$r<2r_+<R_0/2$, and
\begin{equation}\label{eq:map-cutoff-gradient}
\|d\chi_1\|_g
\leq\frac{4}{\sigma_1\log(r_+/r_-)}
\leq\frac{8}{r\log(r_+/r_-)}
\leq\frac{\delta'}{4Cr}.
\end{equation}

We define 
\[
F_1(x)=
\begin{cases}
H(\chi_1(x),\hat F(x)),&r(x)<R_0,\\
\hat F(x),&r(x)\geq R_0,
\end{cases}
\]
where $R_0$ is the constant from  \eqref{eq:map-pole-separation}.
The two expressions agree where $2r_+<r<R_0$, because $\chi_1=0$
there. Hence $F_1$ is smooth and equals $\hat F$ when $r\geq R_0$.
For $r\leq r_-/2$, we have $\sigma_1\leq r_-$ and hence $F_1=p_+$. In particular, this verifies part (1). 

Note that 
\[
\{x\in\partial Y:r(x)\geq r_-/2\}
\]
is a compact subset of $\partial Y_{\mathrm{reg}}$.
This verifies part (2).

For every $v\in T_x\partial Y_{\mathrm{reg}}$, the chain rule gives
\[
dF_1(v)=d\chi_1(v)\,\partial_tH+d_zH\bigl(d\hat F(v)\bigr).
\]
The map $v\mapsto d\chi_1(v)\,\partial_tH$ has rank at most one
and trace norm $\|\partial_tH\|\,\|d\chi_1\|_g$. Therefore, on $\{\sigma_1<r_+\}$, we have
\[
\|dF_1\|_{\operatorname{tr},g}
\leq \|\partial_tH\|\,\|d\chi_1\|_g
+\|d_zH\|\,\|d\hat F\|_{\operatorname{tr},g}
\leq C\|d\chi_1\|_g+C\|d\hat F\|_{\operatorname{tr},g}.
\]

Using \eqref{eq:map-original-bound},
\eqref{eq:map-cutoff-radii}, and
\eqref{eq:map-cutoff-gradient}, we obtain
\[
\|dF_1\|_{\operatorname{tr},g}
\leq\frac{\delta'}{4r}+C\ell
\leq\frac{\delta'}{4r}+\frac{\delta'}{8r_+}
\leq\frac{\delta'}{2r}.
\]
When $\sigma_1\geq r_+$, the map equals $\hat F$.
Multiplying by $W^{-2/3}$ proves
\eqref{eq:modified-map-differential}.

Finally, $\hat F^{-1}(p_-)$ is a finite subset of
$\partial Y_{\mathrm{reg}}$, since $p_-$ is a regular value
outside the image of the singular set.
The map $F_1$ agrees with $\hat F$ near every point of this set,
and the map $H$ does not take the value $p_-$.
Thus
\[
F_1^{-1}(p_-)=\hat F^{-1}(p_-),
\]
with the same local orientation signs. It follows that 
$\deg(F_1)=\deg(\hat F)\ne0$.
\end{proof}

The construction above, followed by a possible modification near the two poles as in Lemma \ref{lm:collapse} (also see \cite[Lemma 3.1]{WangWangZhu}), gives the following lemma. 
\begin{lemma}\label{lemma:completeSpectralMap}
There is a smooth map
$\tilde F\colon \partial Y_{\mathrm{reg}}\to\mathbb S^{m-1}$
of nonzero degree with 
$B_\pm\coloneqq \tilde F^{-1}(p_\pm)$ such that every connected component of
$B_-$ and $B_+$ has nonempty interior in $\partial Y_{\mathrm{reg}}$,
and $\partial Y_{\mathrm{reg}}\setminus B_+$ is relatively compact.
Moreover, for every $\varphi\in C_c^\infty(Y_{\mathrm{reg}})$,
\[
\begin{aligned}
0\leq{}&
\int_{Y_{\mathrm{reg}}}\hat\rho\left(
|\nabla^{\hat g}\varphi|_{\hat g}^2
+\frac12\bigl(\Sc_{\mu,\hat f}(\hat g)-\delta'' \bigr)\varphi^2
\right)\,d\mathcal H_{\hat g}^m\\
&+\int_{\partial Y_{\mathrm{reg}}}\hat\rho\left(
H_{\partial Y}(\hat g)-\|d\tilde F\|_{\operatorname{tr},\hat g}
+\hat\nu_{\partial Y}(\hat f)-\delta''
\right)\varphi^2\,d\mathcal H_{\hat g}^{m-1},
\end{aligned}
\]
where $\delta''=\delta/4$, $\hat f=\log\hat\rho$ and
$\mu$ is the constant fixed in \eqref{eq:conformal-mu-range}.
\end{lemma}


\subsection{From spectral inequality to pointwise inequality}\label{subsec:spectral-to-pointwise}
In this subsection, we turn the spectral positivity of Lemma~\ref{lemma:completeSpectralMap} into pointwise weighted positivity. This supplies the curvature hypotheses needed to apply Theorem~\ref{thm:noncompactScMeanWeighted} in the next lower dimension.

\begin{lemma}\label{lemma:spectralToPointwise}
Assume $m\geq3$ and the hypotheses of
Lemma~\ref{lemma:completeSpectralMap}. Then there exist
$\tilde f\in C^\infty(Y_{\mathrm{reg}})$ and
$\tilde\mu>(m-1)/m$ such that
\[
\Sc_{\tilde\mu,\tilde f}(\hat g)\geq\frac{\delta''}{2}
\quad\text{on }Y_{\mathrm{reg}},
\]
and
\[
H_{\partial Y}(\hat g)-|d\tilde F|_{\operatorname{tr},\hat g}
+\hat\nu_{\partial Y}(\tilde f)\geq\frac{\delta''}{2}
\quad\text{on }\partial Y_{\mathrm{reg}}.
\]
Thus these data satisfy the hypotheses of
Theorem~\ref{thm:noncompactScMeanWeighted} in dimension $m$, while
$\deg(\tilde F)\neq0$.
\end{lemma}

\begin{proof}
	Denote $\hat f=\log\hat \rho$. All differential operators, inner products, and norms in this proof are taken with respect to $\hat g$.

    By smooth approximation, choose
$h\in C^\infty(\partial Y_{\mathrm{reg}})$ such that
\[
\|d\tilde F\|_{\operatorname{tr},\hat g}
\leq h\leq \|d\tilde F\|_{\operatorname{tr},\hat g}
+\frac{\delta''}{4}.
\]
Define
\[
\begin{aligned}
\mathcal Q(\varphi)={}&
\int_{Y_{\mathrm{reg}}}\hat\rho\left(
|\nabla\varphi|^2+
\frac12\left(\Sc_{\mu,\hat f}(\hat g)
-\frac{\delta''}{2}\right)\varphi^2
\right)d\mathcal H_{\hat g}^m\\
&+\int_{\partial Y_{\mathrm{reg}}}\hat\rho\left(
H_{\partial Y}(\hat g)+\hat\nu_{\partial Y}(\hat f)
-h-\frac{\delta''}{2}\right)\varphi^2
d\mathcal H_{\hat g}^{m-1}.
\end{aligned}
\]
Lemma~\ref{lemma:completeSpectralMap} implies
\begin{equation}\label{eq:pointwise-form-gap}
\mathcal Q(\varphi)\geq
\frac{\delta''}{4}\int_{Y_{\mathrm{reg}}}
\hat\rho\varphi^2\,d\mathcal H_{\hat g}^m
+\frac{\delta''}{4}\int_{\partial Y_{\mathrm{reg}}}
\hat\rho\varphi^2\,d\mathcal H_{\hat g}^{m-1}
\end{equation}
for every compactly supported smooth $\varphi$, including those
whose support meets $\partial Y_{\mathrm{reg}}$.

Fix an interior
point $x_0$ and choose relatively open connected domains $U_j$ with
compact closures such that
\[
x_0\in U_1,\qquad
\overline U_j\subset U_{j+1},\qquad
\bigcup_j U_j=Y_{\mathrm{reg}},
\]
and furthermore 
$A_j=\overline{\partial U_j\setminus\partial Y_{\mathrm{reg}}}$
are smooth and transverse to $\partial Y_{\mathrm{reg}}$.
Let $\mathcal Q_j$ be $\mathcal Q$ with the integrals restricted to
$U_j$ and $U_j\cap\partial Y_{\mathrm{reg}}$, respectively, and define 
\[
\lambda_j=\inf_{0\ne \varphi\in V_j}
\frac{\mathcal Q_j(\varphi)}
{\displaystyle\int_{U_j}\hat\rho \varphi^2\,d\mathcal H_{\hat g}^m}.
\]
which is positive by \eqref{eq:pointwise-form-gap}. 

By the strong maximal principle, there is a positive function $v_j$ on $U_j$ such that 
\[
\begin{cases}
-\Delta v_j-\langle\nabla\hat f,\nabla v_j\rangle
+\dfrac12\left(\Sc_{\mu,\hat f}(\hat g)
-\dfrac{\delta''}{2}\right)v_j=\lambda_jv_j
&\textup { on }U_j\setminus\partial Y_{\mathrm{reg}},\\[2pt]
\hat\nu_{\partial Y}(v_j)+\left(
H_{\partial Y}(\hat g)+\hat\nu_{\partial Y}(\hat f)
-h-\dfrac{\delta''}{2}\right)v_j=0
&\textup{ on }U_j\cap\partial Y_{\mathrm{reg}},\\[2pt]
v_j=0&\textup{ on }A_j.
\end{cases}
\]
Note that 
\[
\frac{\delta''}{4}\leq\lambda_{j+1}\leq\lambda_j\leq\lambda_1.
\]
Let $\lambda = \lim_{j\to \infty} \lambda_j$. Then $\lambda \geq \delta''/4$.
Normalize $v_j(x_0)=1$. A  subsequence of $\{v_j\}$ converges to $v>0$ in $C^\ell(K)$
for every $K$ and $\ell$. Therefore, we have 
\begin{align}
-\Delta v-\langle\nabla\hat f,\nabla v\rangle
+\frac12\left(\Sc_{\mu,\hat f}(\hat g)
-\frac{\delta''}{2}\right)v
&=\lambda v,
\label{eq:pointwise-positive-solution}\\
\hat\nu_{\partial Y}(v)+\left(
H_{\partial Y}(\hat g)+\hat\nu_{\partial Y}(\hat f)
-h-\frac{\delta''}{2}\right)v
&=0.
\label{eq:pointwise-robin}
\end{align}

Let us choose $\tilde \mu$ such that
\[
\frac{m-1}{m}<\tilde\mu<2-\frac1\mu,
\]
which is possible because $\mu>m/(m+1)$, and set
$\tilde f=\hat f+\log v$. Dividing
\eqref{eq:pointwise-positive-solution} by $v$ gives
\[
\Sc_{\mu,\hat f}(\hat g)-2\Delta\log v
=2\lambda+\frac{\delta''}{2}
+2|\nabla\log v|^2
+2\langle\nabla\hat f,\nabla\log v\rangle.
\]
Consequently,
\[
\begin{aligned}
\Sc_{\tilde\mu,\tilde f}(\hat g)
={}&2\lambda+\frac{\delta''}{2}
+(2-\tilde\mu)|\nabla\log v|^2
+(\mu-\tilde\mu)|\nabla\hat f|^2\\
&+2(1-\tilde\mu)
\langle\nabla\hat f,\nabla\log v\rangle.
\end{aligned}
\]
Completing the square yields
\[
\begin{aligned}
\Sc_{\tilde\mu,\tilde f}(\hat g)
={}&2\lambda+\frac{\delta''}{2}
+(2-\tilde\mu)\left|
\nabla\log v+\frac{1-\tilde\mu}{2-\tilde\mu}\nabla\hat f
\right|^2\\
&+\frac{\mu(2-\tilde\mu)-1}{2-\tilde\mu}
|\nabla\hat f|^2
\geq\delta'',
\end{aligned}
\]
because $2-\tilde\mu>0$ and $\mu(2-\tilde\mu)-1>0$.
The boundary equation~\eqref{eq:pointwise-robin} yields
\[
H_{\partial Y}(\hat g)-\|d\tilde F\|_{\operatorname{tr},\hat g}
+\hat\nu_{\partial Y}(\tilde f)
=\frac{\delta''}{2}+h-\|d\tilde F\|_{\operatorname{tr},\hat g}
\geq\frac{\delta''}{2}.
\]
Note that we have only changed the weight function $\hat f$ to $\tilde f$ and the constant $\mu$ to $\tilde \mu$. The metric $\hat g$ and the map $\tilde F$ remain unchanged. In particular, the completeness
of $\hat g$ and all properties of $\tilde F$ from
Lemma~\ref{lemma:completeSpectralMap} are preserved. This completes the proof. 
\end{proof}

Let us summarize the proof of Theorem \ref{thm:noncompactScMeanWeighted}.

\begin{proof}[Proof of Theorem \ref{thm:noncompactScMeanWeighted}]
Starting with a counterexample in dimension $n\geq3$, the preceding
construction produces a counterexample in dimension $n-1$ whenever
$n-1\geq3$. Repeating the construction, we reach ambient dimension
three. Its minimizing bubble is a compact smooth surface by
Lemmas~\ref{lemma:existence} and~\ref{lemma:variation}.
On a connected component carrying nonzero  degree boundary map,
Lemma~\ref{lemma:2ndVariationImproved} implies the hypotheses of
Lemma~\ref{lemma:n=2}.
This is a contradiction and proves
Theorem~\ref{thm:noncompactScMeanWeighted} for $n\geq3$.
\end{proof}

\section{Quantitative estimates for the singular set}\label{sec:singular-set-estimates}
We establish the local volume, intrinsic packing, and boundary
tubular estimates used in Subsection~\ref{subsec:singular-set-inputs}.
Throughout this section, $Y$ is the compact capillary bubble from
Lemma~\ref{lemma:existence}, $S$ is its singular set, and
$m=\dim Y=n-1$. We retain the intrinsic metric $d_Y$, ambient balls
$B_r^M$, and covering numbers $N_r^M$ defined there.

Compactness of $Y$ and Arzel\`a--Ascoli, applied to curves of
bounded length, show that $d_Y$ is lower semicontinuous in the
ambient topology. In particular, intrinsic balls are Borel sets.

We first prove the volume bounds needed to convert ambient covers
into intrinsic packing estimates.

\begin{lemma}\label{prop:localVolume}
There are $c,C,r_0>0$ such that, for every $x\in Y$ and $0<r<r_0$,
\[
 cr^m\leq\mathcal H_g^m(B_r^Y(x))\leq Cr^m,
\]
and
\begin{equation}\label{eq:local-boundary-area}
 \mathcal H_g^{m-1}(B_r^Y(x)\cap\partial Y)\leq Cr^{m-1}.
\end{equation}
The constants are uniform, including at singular centers.
Moreover, $Y$ is compact in its intrinsic topology.
\end{lemma}

\begin{proof}
For $m\leq4$, the result follows from smoothness and compactness.
Assume $m\geq5$. Interior and boundary area estimates
\cite{DEGL,DPMreg} give
\begin{equation}\label{eq:Assouad-upper-area}
 \mathcal H_g^m(Y\cap B_t^M(x))\leq Ct^m
 \qquad(x\in Y)
\end{equation}
for all sufficiently small $t$. Since $d_M\leq d_Y$, this proves
the intrinsic upper bound.

\smallskip
\noindent\emph{Sobolev and trace inequalities.}
Put $\partial Y_{\mathrm{reg}}=Y_{\mathrm{reg}}\cap\partial M$. The submanifold Sobolev inequality
\cite{MichaelSimon,Brendle} gives, for nonnegative smooth $u$
supported away from $S$,
\begin{equation*}
 \|u\|_{L^{m/(m-1)}(Y)}
 \leq C\left(\int_Y(|\nabla^Yu|+u)\,d\mathcal H_g^m
 +\int_{\partial Y_{\mathrm{reg}}} u\,d\mathcal H_g^{m-1}\right).
\end{equation*}
Let $\nu_{\partial Y}$ be the outward normal of $\partial Y_{\mathrm{reg}}$ in $Y_{\mathrm{reg}}$ and $\nu$ be the 
outward normal $\nu$ of $\partial M$.
Young's law gives $\langle\nu,\nu_{\partial Y}\rangle=\sin J\geq c_\tau>0$.
The divergence formula applied to $u\nu$, using bounded mean
curvature and bounded derivatives of $\nu$, therefore yields
\begin{equation}\label{eq:boundary-trace}
 \int_{\partial Y}u\,d\mathcal H_g^{m-1}
 \leq C\int_Y(|\nabla^Yu|+u)\,d\mathcal H_g^m.
\end{equation}
Substitution gives
\begin{equation}\label{eq:Assouad-Sobolev}
 \|u\|_{L^{m/(m-1)}(Y)}
 \leq C\int_Y(|\nabla^Yu|+u)\,d\mathcal H_g^m.
\end{equation}
Both inequalities extend across $S$. Indeed,
$\mathcal H_g^{m-1}(S)=0$ and \eqref{eq:Assouad-upper-area}
provide cutoffs $\chi_\ell$ vanishing near $S$, converging to one
on $Y_{\mathrm{reg}}$, with $\int_Y|\nabla^Y\chi_\ell|\to0$.
Apply the inequalities to $u\chi_\ell$ and pass to the limit.
By approximation, they hold for bounded nonnegative functions
locally Lipschitz on $Y_{\mathrm{reg}}$ with integrable tangential gradient.
Applying \eqref{eq:boundary-trace} to an intrinsic cutoff between
$B_r^Y(x)$ and $B_{2r}^Y(x)$, with gradient bounded by $C/r$,
proves \eqref{eq:local-boundary-area}.

\smallskip
\noindent\emph{Lower volume bound.}
For $y\in Y_{\mathrm{reg}}$, write $V(t)=\mathcal H_g^m(B_t^Y(y))$.
Apply \eqref{eq:Assouad-Sobolev} to the cutoff that is one on
$B_t^Y(y)$ and decreases linearly to zero on
$B_{t+h}^Y(y)\setminus B_t^Y(y)$. Letting $h\to 0$ gives
\begin{equation*}
 V(t)^{(m-1)/m}\leq C\bigl(V'(t)+V(t)\bigr)
 \qquad\text{for a.e. }t.
\end{equation*}
Since $V(t)^{1/m}\leq Ct$, the last term is absorbed for
$t<t_0$, uniformly in $y$. Thus $(V^{1/m})'\geq c$ almost
everywhere. Monotonicity of $V^{1/m}$ makes its singular
distributional derivative nonnegative, so integration and
$V(t)\to0$ as $t\to0$ imply
\begin{equation}\label{eq:regular-noncollapsing}
 \mathcal H_g^m(B_t^Y(y))\geq ct^m
 \qquad(y\in Y_{\mathrm{reg}},\ 0<t<t_0).
\end{equation}

This bound and finite total area make $(Y_{\mathrm{reg}},d_Y)$ totally bounded:
the disjoint balls of radius $\min\{q/3,t_0/2\}$ centered at any
$q$-separated set have a uniform positive volume.
For $y\in S$, choose regular points $y_j\to y$ in the ambient
metric. A subsequence is  Cauchy. Passing to a further
subsequence, join successive points by curves with summable
lengths. Their concatenation extends to $y$, and its tail lengths
give $d_Y(y_j,y)\to0$. Consequently,
$B_{t/2}^Y(y_j)\subset B_t^Y(y)$ for large $j$, and
\begin{equation}\label{eq:Assouad-noncollapsing}
 \mathcal H_g^m(B_t^Y(y))\geq ct^m
 \qquad(y\in S,\ 0<t<t_0).
\end{equation}
This proves the volume bounds. The same concatenation argument
shows that every intrinsically Cauchy sequence converges in $Y$.
Together with total boundedness and intrinsic density of $Y_{\mathrm{reg}}$,
this proves intrinsic compactness. In particular,
$\{x\in Y:d_Y(x,S)\geq s\}$ is compact in $Y_{\mathrm{reg}}$ for each $s>0$.
\end{proof}

Metric-dimension statements are understood on the finite-distance
components of $d_Y$. The preceding lower volume bound and finite
total area imply that there are only finitely many such components.
We now prove the singular-set packing estimate. Recall that
$A\subset S$ is intrinsically $r$-separated if $d_Y(a,b)\geq r$
for distinct $a,b\in A$.

\begin{proposition}\label{prop:Assouad}
If $m\leq4$, then $S=\emptyset$. If $m\geq5$, then for every
$\varepsilon>0$ there are constants $C_\varepsilon>0$ and
$R_\varepsilon>0$ such that
\begin{equation}\label{eq:Assouad-packing}
 \#\bigl(A\cap B_R^Y(s)\bigr)
 \leq C_\varepsilon\left(\frac Rr\right)^{m-5+\varepsilon}
\end{equation}
whenever $s\in S$, $0<r<R<R_\varepsilon$, and $A\subset S$ is
$r$-separated with respect to $d_Y$. The constants are independent of
$s,r,R$, and $A$. In particular,
\begin{equation*}
 \dim_{\mathcal A}(S,d_Y)\leq m-5.
\end{equation*}
\end{proposition}

\begin{proof}
For $m\leq4$, the conclusion follows from
Lemma~\ref{lemma:variation}. Assume henceforth that $m\geq5$.

\hyperlink{step:quantitative-rigidity}{\textbf{Step~1}} shows that a small
density increase implies approximation by a minimizing capillary cone.
\hyperlink{step:boundary-covering}{\textbf{Step~2}} combines this
approximation with cone-splitting to cover the boundary quantitative
strata. \hyperlink{step:ambient-covering}{\textbf{Step~3}} uses
regularity and interior estimates to obtain an ambient cover of all
of $S$. Finally,
\hyperlink{step:intrinsic-packing}{\textbf{Step~4}} uses the volume
bounds in Lemma~\ref{prop:localVolume} to convert this cover into
intrinsic packing.

The quantitative stratification argument follows Cheeger--Naber
(see \cite[Sections~7--8]{CheegerNaber}), with signed density and
capillary compactness at the boundary.

Brendle--Wang use the Cheeger--Naber Minkowski dimension bound for
interior $\mu$-bubbles in their conformal blow-up construction
(see \cite[Theorem~3.34 and Section~3.7]{BrendleWang}). Here we need
a bound uniform in the ratio $R/r$, including at the capillary
boundary, and its conversion to intrinsic packing.

Throughout the proof, constants may depend on the fixed geometric data
and the indicated parameters, but not on the centers or the radii
under consideration. Upper scale bounds are decreased when necessary.

\medskip
\noindent\hypertarget{step:quantitative-rigidity}{\textbf{Step 1. Quantitative rigidity from signed monotonicity.}}
We show that a small increase in the corrected signed density
(defined in \eqref{eq:Assouad-density} below) over a sufficiently wide
range of scales implies closeness to a minimizing capillary cone.

In the conformal metric $\widetilde g$, the capillary functional near
$Y$ is
\begin{equation}\label{eq:capillary-functional}
 \underbrace{P_{\widetilde g}(\Omega)}_{\text{interface area}}
 -\underbrace{\int_\Omega\widetilde\Psi\,d\mathcal H^n_{\widetilde g}}_{\text{volume potential}}
 -\underbrace{\int_{\partial^*\Omega\cap\partial M}
 \mu_\partial\,d\mathcal H^{n-1}_{\widetilde g}}_{\substack{\text{weighted area of}\\\text{the wetted boundary}}}.
\end{equation}
By Lemma~\ref{lemma:existence}, we may choose $\tau>0$ and a compact
neighborhood $K$ of $Y$, disjoint from the barriers, such that
$|\mu_\partial|\leq1-\tau$ on $K\cap\partial M$.
Since $Y$ is disjoint from the barriers, $\Omega$ minimizes
\eqref{eq:capillary-functional} under sufficiently small local
variations near $Y$.

All boundary centers below lie in $K\cap\partial M$. For each such
point $x$, use normalized
Fermi coordinates\footnote{Choose geodesic normal
coordinates $z'=(z^1,\ldots,z^m)$ for the induced metric
$\widetilde g|_{T\partial M}$, centered at $x$, using an orthonormal
basis of $T_x\partial M$. If $q(z')$ is the corresponding boundary
point and $\eta$ is the inward $\widetilde g$-unit normal, set
$\Psi_x(z',s)=\exp^{\widetilde g}_{q(z')}(s\eta_{q(z')})$ for $s\geq0$.
Thus $\Psi_x(0,0)=x$, the boundary is $s=0$, and the pulled-back metric
$G=\Psi_x^*\widetilde g$ satisfies $G(0)=I$, $G_{ss}=1$, and
$G_{as}=0$ for $1\leq a\leq m$. Here ``normalized'' refers to
$G(0)=I$. Moreover, $G_{ij}(z)=\delta_{ij}+O(|z|)$, uniformly in the
neighborhood $K$. The coordinate ball is
$B_t(x)=\Psi_x(\{(z',s):s\geq0,\ |z'|^2+s^2<t^2\})$.
After dilation by $t^{-1}$, the rescaled metric has coefficients
$G_{ij}(tz)$ and converges locally in $C^1$ to the Euclidean metric.} centered at $x$
and let $B_t(x)$ denote the coordinate ball of radius $t$.
These coordinates and the metrics $g$ and $\widetilde g$ are uniformly
comparable on $K$. Thus the coordinate covering estimates below
transfer to the ambient balls $B_t^M(x)$ after fixed changes of radii.
Following the signed boundary monotonicity formula of
De Masi--Edelen--Gasparetto--Li
(see \cite[Lemma~2.9]{DEGL}), we introduce the
\textbf{corrected signed density} by
\begin{equation}\label{eq:Assouad-density}
 \Phi_x(t)=(1+c_1t)^{m+1}t^{-m}
 \left(
 \mathcal H^m_{\widetilde g}(Y\cap B_t(x))
 -\int_{B_t(x)\cap\partial^*\Omega\cap\partial M}
 \mu_\partial\,d\mathcal H^m_{\widetilde g}
 \right)+c_2t+C_0.
\end{equation}
First, the signed boundary monotonicity formula
\cite[Lemma~2.9]{DEGL} implies that there exist constants $c_1,c_2>0$
such that $\Phi_x$ is nondecreasing at sufficiently small scales.
The local area bounds \cite{DPMreg} further imply that there exist
constants $C_0,\Lambda,r_{\mathrm{loc}}>0$ such that
$0\leq\Phi_x(t)\leq\Lambda$ for every $x\in K\cap\partial M$
and $0<t<r_{\mathrm{loc}}$.
All constants $c_1,c_2,C_0,\Lambda$, and $r_{\mathrm{loc}}$ can be
chosen independently of $x$.

In the Fermi chart $\Psi_x\colon U_x\subset\mathbb R^{m+1}_+\to M$
centered at $x$, rescale $\Omega$ by setting
\begin{equation*}
 E_{x,t}:=\{z\in t^{-1}U_x:\Psi_x(tz)\in\Omega\}.
\end{equation*}
Comparison cones $C\subset\mathbb R^{m+1}_+$ are centered at the origin
and minimize the Euclidean capillary functional with zero volume
potential and constant contact coefficient
$\sigma\in[-1+\tau/2,1-\tau/2]$.  They include
$\varnothing$ and $\mathbb R^{m+1}_+$.

For $R_0>1$ and $\theta>0$, the configuration at $x$ is
\emph{$\theta$-close to $C$ at scale $t$ on $B_{R_0}$} if
\begin{equation*}
 \mathcal H^{m+1}\bigl((E_{x,t}\mathbin\triangle C)\cap B_{R_0}\bigr)<\theta.
\end{equation*}
Here $B_{R_0}=B^{\mathbb R^{m+1}}_{R_0}(0)$, $\triangle$ denotes
symmetric difference, and $\mathcal H^{m+1}$ is Euclidean volume.

\begin{claim}\label{claim:conicalApproximation}
    For every $R_0>1$ and $\theta>0$, there exist
$a\in(0,R_0^{-2})$, $\delta>0$, and $r_0\in(0,r_{\mathrm{loc}})$
such that, for all $x\in K\cap\partial M$ and $0<t<ar_0$,
\begin{equation}\label{eq:Assouad-rigidity}
 \Phi_x(a^{-1}t)-\Phi_x(at)\leq\delta
\end{equation}
implies that $E_{x,t}$ is $\theta$-close on $B_{R_0}$ to some
comparison cone $C$.
\end{claim}

\begin{proof}[Proof of the claim]
Suppose the claim fails for some $R_0>1$ and $\theta>0$.
Choose $a_i\to 0$ with $a_i<R_0^{-2}$.
Taking $\delta_i=i^{-1}$ and
$r_{0,i}=\min\{i^{-1},r_{\mathrm{loc}}/2\}$ gives
$x_i\in K\cap\partial M$ and $0<t_i<a_i r_{0,i}$ such that
\begin{equation*}
 \Phi_{x_i}(a_i^{-1}t_i)-\Phi_{x_i}(a_it_i)\leq i^{-1},
\end{equation*}
but $E_i:=E_{x_i,t_i}$ satisfies
\begin{equation*}
 \mathcal H^{m+1}\bigl((E_i\mathbin\triangle C)\cap B_{R_0}\bigr)
 \geq\theta
 \qquad\text{for every comparison cone }C.
\end{equation*}
In particular, $a_i^{-1}t_i\to0$, while
$[a_i,a_i^{-1}]$ exhausts $(0,\infty)$.

Passing to a subsequence, assume $x_i\to x_\infty\in K\cap\partial M$.
Write $G_i=\Psi_{x_i}^*\widetilde g$. The rescaled metrics are
\begin{equation*}
 g_i=\sum_{a,b=1}^{m+1}(G_i)_{ab}(t_i z)\,dz^a\otimes dz^b.
\end{equation*}
Since $G_i(0)=I$ and the coordinate bounds are uniform on $K$,
$g_i\to g_{\mathrm{eucl}}$ locally in $C^1$. The rescaled contact
coefficients converge to $\sigma_\infty:=\mu_\partial(x_\infty)$,
and the rescaled potentials
$t_i\widetilde\Psi(\Psi_{x_i}(t_i z))$ tend to zero.
Capillary compactness \cite[Theorem~2.9]{DPMreg} and a diagonal
argument give a Euclidean capillary minimizer
$E_\infty\subset\mathbb R^{m+1}_+$ with coefficient $\sigma_\infty$
and zero volume potential. The sets and wetted traces converge
locally in $L^1$, and the interior perimeter varifolds converge locally.
Writing $V_\infty$ for the limiting perimeter varifold and $W_\infty$
for the wetted-trace area measure, set
\begin{equation*}
 \Theta_\infty(s)=s^{-m}
 \bigl(\|V_\infty\|(B_s)-\sigma_\infty W_\infty(B_s)\bigr)+C_0.
\end{equation*}

For all but countably many $s>0$,
\begin{equation*}
 \|V_\infty\|(\partial B_s)=W_\infty(\partial B_s)=0.
\end{equation*}
At these radii, measure convergence and
$(1+c_1t_is)^{m+1}\to1$, $c_2t_is\to0$ give
$\Phi_{x_i}(t_is)\to\Theta_\infty(s)$.
Choose two such radii $\alpha<\beta$. For large $i$,
$[\alpha,\beta]\subset[a_i,a_i^{-1}]$, so monotonicity yields
\begin{align*}
 0&\leq\Theta_\infty(\beta)-\Theta_\infty(\alpha)
 =\lim_i\bigl(\Phi_{x_i}(\beta t_i)-\Phi_{x_i}(\alpha t_i)\bigr)\\
 &\leq\lim_i\bigl(\Phi_{x_i}(a_i^{-1}t_i)
 -\Phi_{x_i}(a_it_i)\bigr)=0.
\end{align*}
Monotonicity therefore makes $\Theta_\infty$ constant on $(0,\infty)$.

The equality case of signed monotonicity \cite[Lemma~2.9]{DEGL}
shows that $E_\infty$ is a cone with vertex at the origin.
Since $\sigma_\infty\in[-1+\tau,1-\tau]$, it is a comparison cone.
Local $L^1$ convergence now gives
\[
 \mathcal H^{m+1}\bigl((E_i\mathbin\triangle E_\infty)
 \cap B_{R_0}\bigr)\longrightarrow0,
\]
contradicting the choice of $E_i$.
\end{proof}

Thus small density increase gives a uniform conical approximation.
In \hyperlink{step:boundary-covering}{\textbf{Step~2}}, we compare
these approximations at different centers to construct a cover.

\medskip
\noindent\hypertarget{step:boundary-covering}{\textbf{Step 2. Covering the boundary quantitative strata.}}
We combine the conical approximation from
\hyperlink{step:quantitative-rigidity}{\textbf{Step~1}} with
cone-splitting to prove \eqref{eq:effective-stratum-packing} below.

\smallskip
\noindent\emph{Definitions.}
Recall that a capillary cone is \emph{$k$-symmetric} if it is invariant
under translations along a $k$-dimensional linear subspace of
$\partial\mathbb R^{m+1}_+$. Equivalently, it splits off this subspace
as a Euclidean factor.

Fix $0\leq k\leq m-1$, $R_0>1$, and a sufficiently small $r_*>0$.
For $\eta>0$ and $0<r<r_*$, the
\emph{boundary quantitative $k$-stratum} $\mathcal S^k_{\eta,r}$
consists of all $x\in K\cap\partial M$ satisfying
\begin{equation*}
 \mathcal H^{m+1}\bigl((E_{x,t}\mathbin\triangle C)\cap B_{R_0}\bigr)
 \geq\eta
\end{equation*}
for every $t\in[r,r_*]$ and every $(k+1)$-symmetric comparison cone $C$.
This set may contain regular points and satisfies
\[
 \mathcal S^k_{\eta,r}\subset\mathcal S^k_{\eta,r'}
 \qquad\text{for }0<r\leq r'<r_*.
\]

\smallskip
\begin{claim}\label{claim:coneSplitting}
    Choose a uniform $L_0>2$ with
$t^{-1}\Psi_x^{-1}(B_t^M(x)\cap\partial M)\subset B_{L_0}$
for small $t$.
Given $\eta,\rho>0$ and $R_0>1$, there exist
$R_{\mathrm{cs}}>R_0+L_0+1$ and $\theta,t_{\mathrm{cs}}>0$ such that:
if $E=E_{x,t}$, $x\in K\cap\partial M$, $0<t<t_{\mathrm{cs}}$,
and minimizing capillary cones $C_i$ with vertices
$y_0=0,y_1,\ldots,y_{k+1}\in B_{L_0}\cap\partial\mathbb R^{m+1}_+$ satisfy
\begin{align*}
 \operatorname{dist}\bigl(y_i,\operatorname{Aff}
 (y_0,\ldots,y_{i-1})\bigr)&\geq\rho
 &&(i=1,\ldots,k+1),\\
 \mathcal H^{m+1}\bigl((E\mathbin\triangle C_i)
 \cap B_{R_{\mathrm{cs}}}\bigr)&<\theta
 &&(i=0,\ldots,k+1),
\end{align*}
then some $(k+1)$-symmetric minimizing capillary cone $C$ at the origin satisfies
\begin{equation*}
 \mathcal H^{m+1}\bigl((E\mathbin\triangle C)\cap B_{R_0}\bigr)<\eta.
\end{equation*}
Here $\operatorname{Aff}$ denotes affine span.
\end{claim}

\begin{proof}[Proof of the claim]
Suppose the claim fails for some fixed $\eta,\rho>0$ and $R_0>1$.
Following the compactness argument of \cite[Lemma~8.2]{CheegerNaber},
choose $R_j>R_0+L_0+1$ tending to infinity and $\theta_j\to 0$.
The failure of the claim gives $E_j=E_{x_j,t_j}$, with
$x_j\in K\cap\partial M$ and $0<t_j<j^{-1}$, together with
minimizing capillary cones $C_{i,j}$ centered at
$y_{0,j}=0,y_{1,j},\ldots,y_{k+1,j}\in B_{L_0}\cap\partial\mathbb R^{m+1}_+$
such that
\begin{align*}
 \operatorname{dist}\bigl(y_{i,j},\operatorname{Aff}
 (y_{0,j},\ldots,y_{i-1,j})\bigr)&\geq\rho
 &&(i=1,\ldots,k+1),\\
 \mathcal H^{m+1}\bigl((E_j\mathbin\triangle C_{i,j})\cap B_{R_j}\bigr)
 &<\theta_j &&(i=0,\ldots,k+1),
\end{align*}
whereas
\[
 \mathcal H^{m+1}\bigl((E_j\mathbin\triangle C)\cap B_{R_0}\bigr)\geq\eta
\]
for every $(k+1)$-symmetric minimizing capillary cone $C$ at the origin.

As $t_j\to0$, the rescaled metrics become Euclidean, the contact
coefficients become constant after passing to a subsequence, and the
volume potentials vanish. Capillary compactness
\cite[Theorem~2.9]{DPMreg} therefore gives, along a further subsequence,
$E_j\to E_\infty$ locally in $L^1$ and $y_{i,j}\to y_i^\infty$,
where $E_\infty$ minimizes the limiting capillary functional.

For each $i$, the approximation inequality gives
$C_{i,j}\to E_\infty$ locally in $L^1$, since $R_j\to\infty$
and $\theta_j\to0$. Passing dilation invariance to the limit shows
that $E_\infty$ is conical about every $y_i^\infty$, including $0$.
Composing dilations about $0$ and $y_i^\infty$ gives invariance under
translations along $y_i^\infty$. Hence $E_\infty$ is invariant under
translations along $V=\operatorname{span}\{y_1^\infty,\ldots,y_{k+1}^\infty\}$.

The separation condition ensures that $\dim V=k+1$: for
$i=1,\ldots,k+1$ and any fixed $a_1,\ldots,a_{i-1}\in\mathbb R$,
\[
 \left|y_i^\infty-\sum_{\ell<i}a_\ell y_\ell^\infty\right|
 =\lim_{j\to\infty}
 \left|y_{i,j}-\sum_{\ell<i}a_\ell y_{\ell,j}\right|\geq\rho.
\]
Thus each $y_i^\infty$ lies outside the span of the preceding vectors.

Consequently, $E_\infty$ is a $(k+1)$-symmetric minimizing capillary
cone at the origin. Taking $C=E_\infty$ in the counterexample inequality
contradicts local $L^1$ convergence:
\[
 0<\eta\leq
 \mathcal H^{m+1}\bigl((E_j\mathbin\triangle E_\infty)\cap B_{R_0}\bigr)
 \longrightarrow0.
\]
\end{proof}

\smallskip
\noindent\emph{Counting bad scales.}
Fix $0<\gamma<1/2$, to be chosen later.
Apply the Claim \ref{claim:coneSplitting}
with $\rho=\gamma$ to obtain $R_{\mathrm{cs}},\theta,t_{\mathrm{cs}}$,
then the Claim \ref{claim:conicalApproximation}
with radius $R_{\mathrm{cs}}+L_0+2$ and tolerance $\theta/4$
to obtain $a,\delta,r_{\mathrm{rig}}$. Set
\[
 r_j=\gamma^jR,\qquad j\geq0,\qquad
 0<R<\min\{r_*,ar_{\mathrm{rig}},t_{\mathrm{cs}}\}.
\]
For $x\in K\cap\partial M$, define
\[
 \Delta_j(x)=\Phi_x(a^{-1}r_j)-\Phi_x(ar_j).
\]
The scale $r_j$ is \emph{good at $x$} if $\Delta_j(x)\leq\delta$,
and \emph{bad} otherwise. At a good scale, $E_{x,r_j}$ is
$\theta/4$-close to a minimizing capillary cone on
$B_{R_{\mathrm{cs}}+L_0+2}$.

Following the quantitative differentiation argument of Cheeger--Naber
\cite[Section~4, (4.11)--(4.12)]{CheegerNaber}, we bound the number
of bad scales. Here the energy or mass density is replaced by the
corrected signed capillary density \eqref{eq:Assouad-density}.
The wetted-boundary contribution can have either sign, so we use
the corrections and uniform bounds established in
\hyperlink{step:quantitative-rigidity}{\textbf{Step~1}}. The additive constant $C_0$ cancels from $\Delta_j$.
The capillary conical approximation tests the intervals
$[ar_j,a^{-1}r_j]$. To account for their overlap, choose an integer
$L\geq1$ with $\gamma^L<a^2$.
The intervals $[ar_j,a^{-1}r_j]$ are disjoint within each residue
class modulo $L$, so monotonicity and $0\leq\Phi_x\leq\Lambda$ give
\[
 \delta\,\#\{j\geq0:\Delta_j(x)>\delta\}
 \leq\sum_{j\geq0}\Delta_j(x)\leq L\Lambda.
\]
Hence there are at most $Q:=\lceil L\Lambda/\delta\rceil$ bad scales
at any center, uniformly in $x$ and $R$.

\smallskip
\noindent\emph{Construction of the cover.}
We follow the pattern decomposition and covering argument of
Cheeger--Naber \cite[Section~4 and Lemma~8.1]{CheegerNaber}.
The boundary adaptation requires the capillary approximations at
nearby centers to be compared in one Fermi chart.
For $y\in B_{r_j}^M(z)\cap\partial M$, use the rescaled chart at $z$.
The coordinate transition is uniformly $C^1$-close, as $R\to0$, to
$u\mapsto v_y+O_yu$, where $v_y=r_j^{-1}\Psi_z^{-1}(y)\in B_{L_0}$
and $O_y$ is orthogonal and fixes the last coordinate vector.
This rigid map preserves the half-space and capillary minimality,
and places the cone vertex at $v_y$.
The extra radius $L_0+2$ covers $B_{R_{\mathrm{cs}}}$ in the chart
at $z$, while uniform local perimeter bounds control the additional
$L^1$ error. Shrinking $R$ therefore turns each original
$\theta/4$-approximation into a $\theta$-approximation on this common
ball, uniformly in the centers and $r_j\leq R$.
Fix $J\geq1$ and partition
$\mathcal S^k_{\eta,r_J}\cap B_R^M(x)$ into sets $E_T$ with the same
good--bad pattern $T=(T_0,\ldots,T_{J-1})$. Each nonempty pattern
has at most $Q$ bad entries. As in the cited covering argument,
refine balls of radius $r_j$ to radius $r_{j+1}$, keeping their
centers in $E_T$. The refinement bounds here are:

\begin{itemize}[leftmargin=*,itemsep=0.5ex]
\item If $T_j=\mathrm{good}$, the common-chart approximations and
Claim \ref{claim:coneSplitting}
confine $E_T\cap B_{r_j}^M(z)$ to a tube of width $Cr_{j+1}$ about
a $k$-plane in these coordinates. Otherwise, the usual successive
selection of $k+1$ separated directions gives a $(k+1)$-symmetric
approximation at $z$, contradicting $z\in\mathcal S^k_{\eta,r_J}$
and $r_j\in[r_J,r_*]$. Thus at most $C\gamma^{-k}$ balls suffice.

\item If $T_j=\mathrm{bad}$, boundary doubling gives at most
$C\gamma^{-m}$ balls of radius $r_{j+1}$.
\end{itemize}
In both cases, $C$ absorbs coordinate changes and recentering and is
independent of $\gamma$ and $J$. The upper bound on $R$ may depend on $\gamma$.

The standard iteration, starting from a uniformly bounded cover at
radius $R$, and summation over at most $2^J$ patterns give
\begin{equation*}
 N_{r_J}^M\bigl(\mathcal S^k_{\eta,r_J}\cap B_R^M(x)\bigr)
 \leq C(2C)^J\gamma^{-kJ}\gamma^{-(m-k)Q}.
\end{equation*}
Choose $\gamma$ with $2C\leq\gamma^{-\varepsilon}$, then fix the
approximation parameters and $Q$. Since $R/r_J=\gamma^{-J}$,
this gives the exponent $k+\varepsilon$ at the discrete radii.
For $r_{J+1}<r\leq r_J$, use
$\mathcal S^k_{\eta,r}\subset\mathcal S^k_{\eta,r_J}$ and boundary
doubling to pass to radius $r$. For $J=0$, apply doubling directly
to $B_R^M(x)$. Absorbing the factors depending on the fixed $\gamma$
and $Q$, we obtain
\begin{equation}\label{eq:effective-stratum-packing}
 N_r^M\bigl(\mathcal S^k_{\eta,r}\cap B_R^M(x)\bigr)
 \leq C_{\eta,\varepsilon}
 \left(\frac Rr\right)^{k+\varepsilon}.
\end{equation}
The constant is uniform in $x$ and $0<r<R$ for sufficiently small $R$.

\medskip
\noindent\hypertarget{step:ambient-covering}{\textbf{Step 3. Ambient covering of the full singular set.}}
We combine the boundary stratum estimate
\eqref{eq:effective-stratum-packing} with the standard interior
quantitative stratification estimate.

\smallskip
\noindent\emph{Boundary stratum inclusion.}
An $(m-4)$-symmetric minimizing capillary cone splits as
$C=C_0\times\mathbb R^{m-4}$, where $C_0$ has a four-dimensional
interface. Such cones are flat or trivial by
\cite[Theorem~1.16 and Corollary~1.17]{WangZhangRegularity}.
Capillary compactness and boundary $\varepsilon$-regularity
\cite{DPMreg} therefore give uniform constants $\eta_0,c>0$ such
that, after decreasing $r_*$, $\eta_0$-closeness to such a cone at
scale $t\leq r_*$ implies regularity throughout $B_{ct}^M(x)$.
For trivial cones, the density estimates exclude the interface
from this ball. In particular,
\begin{equation}\label{eq:critical-stratum-inclusion}
 S\cap\partial M\subset\mathcal S^{m-5}_{\eta_0,r}
 \qquad(0<r<r_*).
\end{equation}
Thus \eqref{eq:effective-stratum-packing} gives the required
covering bound for boundary singular points.

\smallskip
\noindent\emph{Interior estimate.}
Write $S_i=S\cap M^\circ$. For $m\leq6$, classical interior
regularity gives $S_i=\emptyset$, so the boundary estimate suffices.
For $m\geq7$, the standard interior estimate is
\begin{equation}\label{eq:interior-singular-packing}
 N_r^M\bigl(S_i\cap B_R^M(q)\bigr)
 \leq C_\varepsilon\left(\frac Rr\right)^{m-7+\varepsilon},
\end{equation}
provided $0<r<R<R_{\mathrm{int}}$ and
$B_{4R}^M(q)\Subset M^\circ$ lies in the fixed minimizing
neighborhood. The constants are uniform in $q\in K$.
This follows from the regularity argument of
\cite[Theorem~7.7]{CheegerNaber} and the bounded-mean-curvature
stratum estimate \cite[Theorem~1.3]{NaberValtorta}.
Here the smooth volume potential gives perimeter almost-minimality
and bounded mean curvature in $\widetilde g$, both with constants
$O(t)$ after rescaling by $t$. Compactness with perimeter convergence
\cite[Theorem~2.9]{DPMreg} and Allard regularity \cite{Allard1972}
therefore place $S_i$ in the $(m-7)$-stratum with a uniform tolerance:
every limiting $(m-6)$-symmetric minimizing cone is a
multiplicity-one plane. The area bound
\eqref{eq:Assouad-upper-area} and uniform metric comparison give
\eqref{eq:interior-singular-packing} in $g$.
See also \cite[Theorem~3.34]{BrendleWang} for the corresponding
Minkowski dimension bound for interior $\mu$-bubbles.

\smallskip
\noindent\emph{Combining the boundary and interior estimates.}
Assume $m\geq7$, fix $p\in K\cap\partial M$, and take $R$ small
enough that all enlarged scales below are admissible.
For $x\in S_i\cap B_R^M(p)$, let $d(x)=d_M(x,\partial M)$ and
let $\pi(x)$ be its nearest boundary point.
Choose $A>\max\{1,c^{-1}\}$. Then
\begin{equation*}
 \pi(x)\in\mathcal S^{m-5}_{\eta_0,Ad(x)}.
\end{equation*}
Indeed, an $\eta_0$-approximation by an $(m-4)$-symmetric cone at
any scale $t\in[Ad(x),r_*]$ would imply regularity in
$B_{ct}^M(\pi(x))$, which contains $x$, a contradiction.

For $d_j=2^jr<R$, set
\begin{equation*}
 E_j=\{x\in S_i\cap B_R^M(p):d_j\leq d(x)<2d_j\}.
\end{equation*}
By monotonicity of the strata in the lower cutoff,
\begin{equation*}
 \pi(E_j)\subset
 \mathcal S^{m-5}_{\eta_0,2Ad_j}\cap B_{2R}^M(p).
\end{equation*}
Applying \eqref{eq:effective-stratum-packing} at radius $2Ad_j$
with outer radius $4AR$ covers $\pi(E_j)$ by at most
$C_\varepsilon(R/d_j)^{m-5+\varepsilon/2}$ balls.
After recentering and doubling these balls, each lifted portion
of $E_j$ has diameter at most $C_A d_j$ and distance at least
$d_j$ from $\partial M$. Ambient doubling covers it by a bounded
number of balls $B_{c_0d_j}^M(q)$ with
$B_{4c_0d_j}^M(q)\Subset M^\circ$.
The interior estimate \eqref{eq:interior-singular-packing},
with loss $\varepsilon/2$, now gives
\begin{equation*}
 \begin{split}
 N_r^M(E_j)
 &\leq C_\varepsilon
 \left(\frac R{d_j}\right)^{m-5+\varepsilon/2}
 \left(\frac {d_j}r\right)^{m-7+\varepsilon/2}\\
 &=C_\varepsilon\left(\frac Rr\right)^{m-5+\varepsilon/2}2^{-2j},
 \end{split}
\end{equation*}
using doubling when $r$ is comparable to $d_j$.
Summing in $j$ covers the portion with $d\geq r$.

For the remaining strip $S\cap B_R^M(p)\cap\{d<r\}$,
all projections lie in $\mathcal S^{m-5}_{\eta_0,Ar}$, including
the boundary points by \eqref{eq:critical-stratum-inclusion}.
The boundary stratum estimate covers these projections by
$C_\varepsilon(R/r)^{m-5+\varepsilon}$ balls of radius $Ar$.
Enlarging by $r$ and applying ambient doubling covers the strip
by the same order of balls of radius $r$.

Finally, for $s\in S$, use the interior estimate if
$B_{4R}^M(s)\Subset M^\circ$. Otherwise choose $p\in\partial M$
with $d_M(s,p)\leq4R$ and apply the boundary-centered estimate
to $B_R^M(s)\subset B_{5R}^M(p)$.
Thus, uniformly for sufficiently small $0<r<R$,
\begin{equation}\label{eq:ambient-singular-packing}
 N_r^M\bigl(S\cap B_R^M(s)\bigr)
 \leq C_\varepsilon\left(\frac Rr\right)^{m-5+\varepsilon}.
\end{equation}

\medskip
\noindent\hypertarget{step:intrinsic-packing}{\textbf{Step 4. Intrinsic packing and the Assouad bound.}}
We use the volume bounds of Lemma~\ref{prop:localVolume}
to pass from ambient covering to intrinsic packing.
Fix $s\in S$ and $0<r<R<R_\varepsilon$, with $R_\varepsilon$
small enough that the estimates apply at all enlarged radii below.
By \eqref{eq:ambient-singular-packing}, after recentering a cover
by balls of radius $r/2$, we can cover $S\cap B_{3R}^M(s)$ by
at most $C_\varepsilon(R/r)^{m-5+\varepsilon}$ balls
$B_r^M(z_i)$ with $z_i\in S$.
Since $d_M\leq d_Y$,
\begin{equation*}
 B_{2R}^Y(s)\cap\{x:d_Y(x,S)<r\}
 \subset\bigcup_i B_{2r}^M(z_i).
\end{equation*}
The upper area bound \eqref{eq:Assouad-upper-area} therefore gives
\begin{equation}\label{eq:Assouad-tubular}
 \begin{split}
 &\mathcal H_g^m\bigl(B_{2R}^Y(s)
 \cap\{x:d_Y(x,S)<r\}\bigr)\\
 &\qquad\leq\sum_i\mathcal H_g^m(Y\cap B_{2r}^M(z_i))
 \leq C_\varepsilon R^{m-5+\varepsilon}r^{5-\varepsilon}.
 \end{split}
\end{equation}

Let $F\subset S\cap B_R^Y(s)$ be any finite intrinsically
$r$-separated set. The balls $B_{r/3}^Y(y)$, $y\in F$, are
pairwise disjoint and lie in the neighborhood in
\eqref{eq:Assouad-tubular}. By the intrinsic lower volume bound
\eqref{eq:Assouad-noncollapsing},
\begin{equation*}
 \#F\,c(r/3)^m
 \leq C_\varepsilon R^{m-5+\varepsilon}r^{5-\varepsilon},
 \qquad
 \#F\leq C_\varepsilon(R/r)^{m-5+\varepsilon}.
\end{equation*}
This proves \eqref{eq:Assouad-packing}. The disjoint-ball volume
comparison is standard (see for example
\cite[Section~3.7, following Theorem~3.34]{BrendleWang} for its
ambient version). Here the intrinsic lower volume bound gives
packing with respect to $d_Y$.

Taking a maximal intrinsically $r$-separated set gives the same
bound for the number of intrinsic $r$-balls covering
$S\cap B_R^Y(s)$. To extend it to all scales, fix
$0<\rho<R_\varepsilon/4$. The intrinsic lower volume bound and
$\mathcal H_g^m(Y)<\infty$ bound the cardinality of every
$\rho$-separated subset of $S$, so $S$ admits a finite intrinsic
$\rho$-net. For $r<\rho\leq R$, apply the local estimate in the
balls of this fixed cover. For $\rho\leq r<R$, the net itself
suffices. Thus the covering bound holds for all $0<r<R$, after
adjusting $C_\varepsilon$. Letting $\varepsilon\to 0$ gives
$\dim_{\mathcal A}(S,d_Y)\leq m-5$.
\end{proof}

\begin{remark}[The Assouad bound]
The Assouad bound follows from uniform control over both the center
and the scales, rather than from a fixed-ball Minkowski estimate
alone. Here the local area bounds, smooth geometric data, and
uniform separation of the contact coefficient from $\pm1$ make
the rescaled estimates uniform. The boundary regularity and layer
arguments preserve this uniformity, yielding
\eqref{eq:ambient-singular-packing}. The upper area and intrinsic
lower volume bounds then give the intrinsic packing estimate
\eqref{eq:Assouad-packing}.

The interior estimates of Cheeger--Naber \cite{CheegerNaber} and
Naber--Valtorta \cite{NaberValtorta} likewise imply local ambient
Assouad bounds when the rescaled hypotheses and regularity
thresholds are uniform. Our argument additionally treats the
capillary boundary and transfers the estimate to the intrinsic
metric.
\end{remark}

The ambient covering estimate also gives the following boundary
area bound near $S$.

\begin{corollary}[Boundary tubular estimate]\label{cor:boundary-tubular}
For every $\varepsilon>0$, there are $C_\varepsilon,R_\varepsilon>0$
such that
\begin{equation}\label{eq:boundary-tubular}
 \mathcal H_g^{m-1}\bigl(B_R^Y(x)\cap\partial Y
 \cap\{y:d_Y(y,S)\leq r\}\bigr)
 \leq C_\varepsilon R^{m-5+\varepsilon}r^{4-\varepsilon}
\end{equation}
for all $x\in S$ and $0<r<R<R_\varepsilon$.
\end{corollary}

\begin{proof}
Apply the boundary trace inequality from the proof of
Lemma~\ref{prop:localVolume} to an ambient cutoff equal to
one on $B_t^M(z)$, supported in $B_{2t}^M(z)$, and with gradient
bounded by $C/t$. The inequality extends across $S$ by the cutoff
argument there, and $\mathcal H_g^{m-1}(S)=0$. Hence
\eqref{eq:Assouad-upper-area} gives
\begin{equation*}
 \mathcal H_g^{m-1}(\partial Y\cap B_t^M(z))
 \leq C(t^{-1}+1)\mathcal H_g^m(Y\cap B_{2t}^M(z))
 \leq Ct^{m-1}
\end{equation*}
uniformly for $z\in Y$ and sufficiently small $t$.

Fix $x\in S$ and $0<r<R<R_\varepsilon$, decreasing
$R_\varepsilon$ so that all enlarged radii are admissible.
By \eqref{eq:ambient-singular-packing}, after recentering a cover
by balls of radius $r/2$, we obtain
\begin{equation*}
 S\cap B_{3R}^M(x)\subset\bigcup_{i=1}^N B_r^M(z_i),
 \qquad N\leq C_\varepsilon(R/r)^{m-5+\varepsilon},
 \qquad z_i\in S.
\end{equation*}
If $y\in B_R^Y(x)$ and $d_Y(y,S)\leq r$, choose $z\in S$
with $d_Y(y,z)<2r$. Since $d_M\leq d_Y$, we have
$z\in B_{3R}^M(x)$ and $y\in B_{3r}^M(z_i)$ for some $i$.
Summing the boundary area bound over these balls yields
\begin{equation*}
 \begin{split}
 &\mathcal H_g^{m-1}\bigl(B_R^Y(x)\cap\partial Y
 \cap\{y:d_Y(y,S)\leq r\}\bigr)\\
 &\qquad\leq\sum_{i=1}^N
 \mathcal H_g^{m-1}(\partial Y\cap B_{3r}^M(z_i))
 \leq CNr^{m-1}
 \leq C_\varepsilon R^{m-5+\varepsilon}r^{4-\varepsilon}.
 \end{split}
\end{equation*}
\end{proof}

\smallskip
\noindent\emph{Extension to arbitrary centers.}
To obtain \eqref{eq:interior-tubular-summary} and
\eqref{eq:boundary-tubular-summary} for $x\in Y$, it suffices to
consider a nonempty intersection with $\{d_Y(\,\cdot\,,S)\leq r\}$.
Choose $s\in S$ with $d_Y(x,s)<3R$. Then 
$B_R^Y(x)\subset B_{4R}^Y(s)$. Apply
\eqref{eq:Assouad-tubular} and \eqref{eq:boundary-tubular} at $s$,
using $2r$ in the former to include the level set
$\{d_Y(\,\cdot\,,S)=r\}$. Adjusting $C_\varepsilon$ and
$R_\varepsilon$ absorbs these fixed enlargements.

\bibliographystyle{plain}
\bibliography{scalar_mean.bib}

@unpublished{Guo24,
	author = {Yifan Guo},
	eprint = {2404.00119},
	note = {arXiv:2404.00119},
	title = {Green functions on stationary varifolds},
	url = {https://arxiv.org/pdf/2404.00119.pdf},
	year = {2024}}

@article {MR4249776,
	AUTHOR = {Chen, Zhen-Qing and Kim, Panki and Kumagai, Takashi and Wang,
	Jian},
	TITLE = {Heat kernel upper bounds for symmetric {M}arkov semigroups},
	JOURNAL = {J. Funct. Anal.},
	FJOURNAL = {Journal of Functional Analysis},
	VOLUME = {281},
	YEAR = {2021},
	NUMBER = {4},
	PAGES = {Paper No. 109074, 40},
	ISSN = {0022-1236,1096-0783},
	MRCLASS = {60J35 (31C25 35K08 39B62 47D07 60J25)},
	MRNUMBER = {4249776},
	MRREVIEWER = {Zoran\ Vondra\v cek},
	DOI = {10.1016/j.jfa.2021.109074},
	URL = {https://doi.org/10.1016/j.jfa.2021.109074},
}

@article {MR898496,
	AUTHOR = {Carlen, E. A. and Kusuoka, S. and Stroock, D. W.},
	TITLE = {Upper bounds for symmetric {M}arkov transition functions},
	JOURNAL = {Ann. Inst. H. Poincar\'e{} Probab. Statist.},
	FJOURNAL = {Annales de l'Institut Henri Poincar\'e. Probabilit\'es et
	Statistique},
	VOLUME = {23},
	YEAR = {1987},
	NUMBER = {2},
	PAGES = {245--287},
	ISSN = {0246-0203},
	MRCLASS = {35K10 (60J35 60J60)},
	MRNUMBER = {898496},
	MRREVIEWER = {S.\ R. S. Varadhan},
	URL = {http://www.numdam.org/item?id=AIHPB_1987__23_2_245_0},
}

@article {MR675736,
    AUTHOR = {Kazdan, Jerry L.},
     TITLE = {Deformation to positive scalar curvature on complete
              manifolds},
   JOURNAL = {Math. Ann.},
  FJOURNAL = {Mathematische Annalen},
    VOLUME = {261},
      YEAR = {1982},
    NUMBER = {2},
     PAGES = {227--234},
      ISSN = {0025-5831,1432-1807},
   MRCLASS = {53C20 (58G30)},
  MRNUMBER = {675736},
MRREVIEWER = {V.\ I.\ Oliker},
       DOI = {10.1007/BF01456220},
       URL = {https://doi.org/10.1007/BF01456220},
}

@unpublished{ChaiWan24,
    author = {Xiaoxiang Chai and Xueyuan Wan},
	eprint = {2407.10212},
	note = {arXiv:2407.10212},
	title = {Scalar curvature rigidity of domains in a warped product},
	url = {https://arxiv.org/pdf/2407.10212.pdf},
	year = {2024}}

@unpublished{Simone24,
    author = {Simone Cecchini and Sven Hirsch and Rudolf Zeidler},
	eprint = {2404.17533},
	note = {arXiv:2404.17533},
	title = {Rigidity of spin fill-ins with non-negative scalar curvature},
	url = {https://arxiv.org/pdf/2404.17533.pdf},
	year = {2024}}

@article{Brendle,
  author        = {Brendle, Simon},
  title         = {The isoperimetric inequality for a minimal submanifold in
                   {Euclidean} space},
  journal       = {Journal of the American Mathematical Society},
  volume        = {34},
  number        = {2},
  year          = {2021},
  pages         = {595--603},
  doi           = {10.1090/jams/969}
}

@article{CheegerNaber,
  author        = {Cheeger, Jeff and Naber, Aaron},
  title         = {Quantitative Stratification and the Regularity of Harmonic
                   Maps and Minimal Currents},
  journal       = {Communications on Pure and Applied Mathematics},
  volume        = {66},
  number        = {6},
  year          = {2013},
  pages         = {965--990},
  doi           = {10.1002/cpa.21446}
}

@article{DEGL,
  author        = {De Masi, Luigi and Edelen, Nick and Gasparetto, Carlo and
                   Li, Chao},
  title         = {Regularity of minimal surfaces with capillary boundary
                   conditions},
  journal       = {Communications on Pure and Applied Mathematics},
  volume        = {78},
  number        = {12},
  year          = {2025},
  pages         = {2436--2502},
  doi           = {10.1002/cpa.70008}
}

@article{DPMreg,
  author        = {De Philippis, Guido and Maggi, Francesco},
  title         = {Regularity of Free Boundaries in Anisotropic Capillarity
                   Problems and the Validity of {Young}'s Law},
  journal       = {Archive for Rational Mechanics and Analysis},
  volume        = {216},
  number        = {2},
  year          = {2015},
  pages         = {473--568},
  doi           = {10.1007/s00205-014-0813-2}
}

@article{MichaelSimon,
  author        = {Michael, J. H. and Simon, L. M.},
  title         = {Sobolev and mean-value inequalities on generalized
                   submanifolds of {$\mathbb{R}^n$}},
  journal       = {Communications on Pure and Applied Mathematics},
  volume        = {26},
  number        = {3},
  year          = {1973},
  pages         = {361--379},
  doi           = {10.1002/cpa.3160260305}
}

@article{NaberValtorta,
  author        = {Naber, Aaron and Valtorta, Daniele},
  title         = {The singular structure and regularity of stationary
                   varifolds},
  journal       = {Journal of the European Mathematical Society},
  volume        = {22},
  number        = {10},
  year          = {2020},
  pages         = {3305--3382},
  doi           = {10.4171/jems/987}
}

@misc{WangZhangRegularity,
  author        = {Wang, Gaoming and Zhang, Xuwen},
  title         = {Regularity of stable capillary minimal hypersurfaces},
  howpublished  = {arXiv:2605.20964},
  year          = {2026},
  eprint        = {2605.20964},
  archivePrefix = {arXiv},
  primaryClass  = {math.DG},
  url           = {https://arxiv.org/abs/2605.20964}
}

@unpublished{WangXie26,
	author = {Jinmin Wang and Zhizhang Xie},
	eprint = {2608.15091},
	note = {arXiv:2608.15091},
	title = {A {$2$}-systolic inquality for {$\mathbb S^2\times P$}},
	url = {https://arxiv.org/pdf/2608.15091.pdf},
	year = {2026}}

@unpublished{WangWangXie,
	author = {Jian Wang and Jinmin Wang and Zhizhang Xie},
	eprint = {2606.21325},
	note = {arXiv:2606.21325},
	title = {{$L^\infty$}-metrics on tori and {S}choen's conjecture},
	url = {https://arxiv.org/pdf/2606.21325.pdf},
	year = {2026}}

@unpublished{WangXieConformal,
	author = {Jinmin Wang and Zhizhang Xie},
	eprint = {2512.05769},
	note = {arXiv:2512.05769},
	title = {Scalar-mean rigidity beyond warped product spaces},
	url = {https://arxiv.org/pdf/2512.05769.pdf},
	year = {2025}}

@unpublished{BHHSZ,
	author = {Yuchen Bi and Tianze Hao and Shihang He and Yuguang Shi and Jintian Zhu},
	eprint = {2603.02769},
	note = {arXiv:2603.02769},
	title = {A proof for the Riemannian positive mass theorem up to dimension 19},
	url = {https://arxiv.org/pdf/2603.02769.pdf},
	year = {2026}}

@unpublished{WangXie24,
	author = {Jinmin Wang and Zhizhang Xie},
	eprint = {2407.21312},
	note = {arXiv:2407.21312},
	title = {Scalar curvature rigidity of spheres with subsets removed and {$L^\infty$} metrics},
	url = {https://arxiv.org/pdf/2407.21312.pdf},
	year = {2024}}

@unpublished{ChuLeeZhu24,
	author = {Jianchun Chu and Man-Chun Lee and Jintian Zhu},
	eprint = {arXiv:2405.19724},
	note = {arXiv:2405.19724},
	title = {Llarull's theorem on punctured sphere with $L^\infty$-metric},
	url = {https://arxiv.org/pdf/2405.19724.pdf},
	year = {2024}}

@misc{WangWangZhu,
  author        = {Wang, Jinmin and Wang, Zhichao and Zhu, Bo},
  title         = {Scalar-mean rigidity theorem for compact manifolds with
                   boundary},
  howpublished  = {arXiv:2409.14503},
  year          = {2024},
  eprint        = {2409.14503},
  archivePrefix = {arXiv},
  primaryClass  = {math.DG},
  url           = {https://arxiv.org/abs/2409.14503}
}

@misc{KhuriWangWang,
  author        = {Khuri, Marcus and Wang, Jian and Wang, Jinmin},
  title         = {Riemannian Positive Mass Theorem in All Dimensions in the
                   Presence of Low-Codimension Singularities},
  howpublished  = {arXiv:2606.23529},
  year          = {2026},
  eprint        = {2606.23529},
  archivePrefix = {arXiv},
  primaryClass  = {math.DG},
  url           = {https://arxiv.org/abs/2606.23529}
}

@misc{BrendleWang,
  author        = {Brendle, S. and Wang, Y.},
  title         = {A dimension descent scheme for the positive mass theorem in
                   arbitrary dimension},
  howpublished  = {arXiv:2604.08473},
  year          = {2026},
  eprint        = {2604.08473},
  archivePrefix = {arXiv},
  primaryClass  = {math.DG},
  url           = {https://arxiv.org/abs/2604.08473}
}

@article {SY_pmt_1,
    AUTHOR = {Schoen, Richard and Yau, Shing Tung},
     TITLE = {On the proof of the positive mass conjecture in general
              relativity},
   JOURNAL = {Comm. Math. Phys.},
  FJOURNAL = {Communications in Mathematical Physics},
    VOLUME = {65},
      YEAR = {1979},
    NUMBER = {1},
     PAGES = {45--76},
      ISSN = {0010-3616,1432-0916},
   MRCLASS = {83C99 (53C20 58E20)},
  MRNUMBER = {526976},
MRREVIEWER = {J.\ L.\ Kazdan},
       URL = {http://projecteuclid.org/euclid.cmp/1103904790},
}

@article {Schoen_Yau_imcompressible,
    AUTHOR = {Schoen, R. and Yau, Shing Tung},
     TITLE = {Existence of incompressible minimal surfaces and the topology
              of three-dimensional manifolds with nonnegative scalar
              curvature},
   JOURNAL = {Ann. of Math. (2)},
  FJOURNAL = {Annals of Mathematics. Second Series},
    VOLUME = {110},
      YEAR = {1979},
    NUMBER = {1},
     PAGES = {127--142},
      ISSN = {0003-486X},
   MRCLASS = {58E12 (49F10 53C42)},
  MRNUMBER = {541332},
MRREVIEWER = {Jonathan\ Sacks},
       DOI = {10.2307/1971247},
       URL = {https://doi.org/10.2307/1971247},
}

@incollection {Schoen_Yau_psc_higher,
    AUTHOR = {Schoen, Richard and Yau, Shing-Tung},
     TITLE = {Positive scalar curvature and minimal hypersurface
              singularities},
 BOOKTITLE = {Surveys in differential geometry 2019. {D}ifferential
              geometry, {C}alabi-{Y}au theory, and general relativity.
              {P}art 2},
    SERIES = {Surv. Differ. Geom.},
    VOLUME = {24},
     PAGES = {441--480},
 PUBLISHER = {Int. Press, Boston, MA},
      YEAR = {[2022] \copyright 2022},
      ISBN = {978-1-57146-413-2},
   MRCLASS = {53C20 (58E12)},
  MRNUMBER = {4479726},
MRREVIEWER = {Harish\ Seshadri},
}

@article{Wu_capillarysurfaces,
  author        = {Wu, Yujie},
  title         = {Capillary Surfaces in Manifolds with Nonnegative Scalar
                   Curvature and Strictly Mean Convex Boundary},
  journal       = {International Mathematics Research Notices},
  volume        = {2025},
  number        = {9},
  year          = {2025},
  eid           = {rnaf106},
  note          = {Article No.~rnaf106},
  doi           = {10.1093/imrn/rnaf106},
  eprint        = {2405.03993},
  archivePrefix = {arXiv},
  primaryClass  = {math.DG},
  url           = {https://arxiv.org/abs/2405.03993}
}

@article{chodosh2024_improvedregularity,
  author        = {Chodosh, Otis and Edelen, Nick and Li, Chao},
  title         = {Improved regularity for minimizing capillary hypersurfaces},
  journal       = {Ars Inveniendi Analytica},
  year          = {2025},
  eid           = {2},
  pagetotal     = {27},
  note          = {Paper No.~2, 27 pp.},
  doi           = {10.15781/axnt-ww91},
  eprint        = {2401.08028},
  archivePrefix = {arXiv},
  primaryClass  = {math.DG},
  url           = {https://arxiv.org/abs/2401.08028}
}

@article{CWXZLlarull4,
  author        = {Cecchini, Simone and Wang, Jinmin and Xie, Zhizhang and Zhu,
                   Bo},
  title         = {Scalar curvature rigidity of the four-dimensional sphere},
  journal       = {Mathematische Annalen},
  volume        = {394},
  number        = {3},
  year          = {2026},
  eid           = {58},
  note          = {Article No.~58},
  doi           = {10.1007/s00208-026-03385-w},
  eprint        = {2402.12633},
  archivePrefix = {arXiv},
  primaryClass  = {math.DG},
  url           = {https://arxiv.org/abs/2402.12633},
  fjournal = {Mathematische Annalen},
  issn = {0025-5831,1432-1807},
  mrclass = {53C24 (53C21)},
  mrnumber = {5033745}
}

@article{Gromov_mean_light_scalar,
    AUTHOR = {Gromov, Misha},
     TITLE = {Mean curvature in the light of scalar curvature},
   JOURNAL = {Ann. Inst. Fourier (Grenoble)},
  FJOURNAL = {Universit\'e{} de Grenoble. Annales de l'Institut Fourier},
    VOLUME = {69},
      YEAR = {2019},
    NUMBER = {7},
     PAGES = {3169--3194},
      ISSN = {0373-0956,1777-5310},
   MRCLASS = {53C23 (53C21 53C42)},
  MRNUMBER = {4286832},
MRREVIEWER = {Harish\ Seshadri},
       DOI = {10.5802/aif.3347},
       URL = {https://doi.org/10.5802/aif.3347},
}

@article{Wang:2021tq,
	author = {Jinmin Wang and Zhizhang Xie and Guoliang Yu},
	eprint = {2112.01510},

	title = {On {G}romov's dihedral extremality and rigidity conjectures},
	url = {https://arxiv.org/pdf/2112.01510.pdf},
	year = {2021}}

@article {Cecchini2021:vs,
    AUTHOR = {Cecchini, Simone and Zeidler, Rudolf},
     TITLE = {Scalar and mean curvature comparison via the {D}irac operator},
   JOURNAL = {Geom. Topol.},
  FJOURNAL = {Geometry \& Topology},
    VOLUME = {28},
      YEAR = {2024},
    NUMBER = {3},
     PAGES = {1167--1212},
      ISSN = {1465-3060,1364-0380},
   MRCLASS = {53C27 (53C21 53C23 53C24 58J20)},
  MRNUMBER = {4746412},
MRREVIEWER = {Roger\ Nakad},
       DOI = {10.2140/gt.2024.28.1167},
       URL = {https://doi.org/10.2140/gt.2024.28.1167},
}

@article {MR3257837,
    AUTHOR = {Hijazi, Oussama and Montiel, Sebasti\'an},
     TITLE = {A holographic principle for the existence of parallel spinor
              fields and an inequality of {S}hi-{T}am type},
   JOURNAL = {Asian J. Math.},
  FJOURNAL = {Asian Journal of Mathematics},
    VOLUME = {18},
      YEAR = {2014},
    NUMBER = {3},
     PAGES = {489--506},
      ISSN = {1093-6106,1945-0036},
   MRCLASS = {53C27 (53C40 53C80 58J50)},
  MRNUMBER = {3257837},
MRREVIEWER = {Colette\ Ann\'e},
       DOI = {10.4310/AJM.2014.v18.n3.a6},
       URL = {https://doi.org/10.4310/AJM.2014.v18.n3.a6},
}

@article{Wang:2022vf,
	author = {Jinmin Wang and Zhizhang Xie},
	eprint = {2203.09511},
	title = {On {G}romov's flat corner domination conjecture and {S}toker's conjecture},
	url = {https://arxiv.org/pdf/2203.09511.pdf},
	year = {2022}}

@article{Lottboundary,
	author = {Lott, John},
	doi = {10.1090/proc/15551},
	fjournal = {Proceedings of the American Mathematical Society},
	issn = {0002-9939},
	journal = {Proc. Amer. Math. Soc.},
	mrclass = {53C21 (58J20)},
	mrnumber = {4305995},
	number = {10},
	pages = {4451--4459},
	title = {Index theory for scalar curvature on manifolds with boundary},
	url = {https://doi.org/10.1090/proc/15551},
	volume = {149},
	year = {2021}}

@article{Listing:2010te,
	author = {Mario Listing},
    note = {arXiv:1007.1832},
	eprint = {1007.1832},
	title = {Scalar Curvature on Compact Symmetric Spaces},
	url = {https://arxiv.org/pdf/1007.1832.pdf},
	year = {2010}}

@article{Llarull,
	author = {Llarull, Marcelo},
	doi = {10.1007/s002080050136},
	fjournal = {Mathematische Annalen},
	issn = {0025-5831},
	journal = {Math. Ann.},
	mrclass = {53C21 (58G10 58G25)},
	mrnumber = {1600027},
	mrreviewer = {Uwe Semmelmann},
	number = {1},
	pages = {55--71},
	title = {Sharp estimates and the {D}irac operator},
	url = {https://doi.org/10.1007/s002080050136},
	volume = {310},
	year = {1998}}

@incollection {Gromov_four_lectures,
    AUTHOR = {Gromov, Misha},
     TITLE = {Four lectures on scalar curvature},
 BOOKTITLE = {Perspectives in scalar curvature. {V}ol. 1},
     PAGES = {1--514},
 PUBLISHER = {World Sci. Publ., Hackensack, NJ},
      YEAR = {[2023] \copyright2023},
      ISBN = {978-981-124-998-3; 978-981-124-935-8; 978-981-124-936-5},
   MRCLASS = {53C23 (53-02 53C21)},
  MRNUMBER = {4577903},
}

@article{Shi_Tam,
  author        = {Shi, Yuguang and Tam, Luen-Fai},
  title         = {Positive Mass Theorem and the Boundary Behaviors of Compact
                   Manifolds with Nonnegative Scalar Curvature},
  journal       = {Journal of Differential Geometry},
  volume        = {62},
  number        = {1},
  year          = {2002},
  doi           = {10.4310/jdg/1090425530},
  pages         = {79--125},
  fjournal = {Journal of Differential Geometry},
  issn = {0022-040X,1945-743X},
  mrclass = {53C20 (53C21)},
  mrnumber = {1987378},
  mrreviewer = {Pengzi\ Miao},
  url = {http://projecteuclid.org/euclid.jdg/1090425530}
}

@article{Shi_Tam_extension,
  author        = {Eichmair, Michael and Miao, Pengzi and Wang, Xiaodong},
  title         = {Extension of a theorem of {Shi} and {Tam}},
  journal       = {Calculus of Variations and Partial Differential Equations},
  volume        = {43},
  number        = {1--2},
  year          = {2012},
  pages         = {45--56},
  doi           = {10.1007/s00526-011-0402-2},
  fjournal = {Calculus of Variations and Partial Differential Equations},
  issn = {0944-2669,1432-0835},
  mrclass = {53C21 (53C20 53C24)},
  mrnumber = {2860402},
  mrreviewer = {Lan-Hsuan\ Huang},
  url = {https://doi.org/10.1007/s00526-011-0402-2}
}

@misc{BiZhu,
  author        = {Bi, Yuchen and Zhu, Jintian},
  title         = {Positive Scalar Curvature Obstructions via Singular
                   Dimension Descent},
  howpublished  = {arXiv:2606.20528},
  year          = {2026},
  eprint        = {2606.20528},
  archivePrefix = {arXiv},
  primaryClass  = {math.DG},
  url           = {https://arxiv.org/abs/2606.20528}
}

@book{Fraser_2020,
	author = {Fraser, Jonathan M.},
	month = Oct,
	publisher = {Cambridge University Press},
	title = {Assouad Dimension and Fractal Geometry},
	year = {2020}}

@book{Fukushima_1994,
	author = {Fukushima, Masatoshi and Oshima, Yoichi and Takeda, Masayoshi},
	month = Dec,
	publisher = {DE GRUYTER},
	title = {Dirichlet Forms and Symmetric Markov Processes},
	year = {1994}}

@article{MR0657523,
	author = {Gr\"uter, Michael and Widman, Kjell-Ove},
	journal = {Manuscripta Math.},
	number = {3},
	pages = {303--342},
	title = {The {G}reen function for uniformly elliptic equations},
	volume = {37},
	year = {1982}}

@article{Allard1972,
	author = {Allard, William K.},
	journal = {The Annals of Mathematics},
	month = May,
	number = {3},
	pages = {417},
	title = {On the First Variation of a Varifold},
	volume = {95},
	year = {1972}}

@article{Reilly1977,
  author  = {Reilly, Robert C.},
  title   = {Applications of the Hessian operator in a
             Riemannian manifold},
  journal = {Indiana University Mathematics Journal},
  volume  = {26},
  number  = {3},
  pages   = {459--472},
  year    = {1977},
  doi     = {10.1512/iumj.1977.26.26036}
}

@article{Kasue1983,
  author  = {Kasue, Atsushi},
  title   = {Ricci curvature, geodesics and some geometric
             properties of Riemannian manifolds with boundary},
  journal = {Journal of the Mathematical Society of Japan},
  volume  = {35},
  number  = {1},
  pages   = {117--131},
  year    = {1983},
  doi     = {10.2969/jmsj/03510117}
}

@book{MRBS78,
	address = {New York},
	author = {Reed, Michael and Simon, Barry},
	publisher = {Academic Press [Harcourt Brace Jovanovich Publishers]},
	title = {Methods of modern mathematical physics. {IV}. {A}nalysis of operators},
	year = {1978}}

@article{MR138874,
	author = {Birman, M. \v{S}.},
	journal = {Vestnik Leningrad. Univ.},
	number = {1},
	pages = {22--55},
	title = {Perturbations of the continuous spectrum of a singular elliptic operator by varying the boundary and the boundary conditions},
	volume = {17},
	year = {1962}}
\nocite{WangZhangRegularity, WangXie26}
\end{document}